\documentclass[11pt]{amsart}

\usepackage{amsfonts,amssymb,amsmath,amsthm,latexsym,graphics,epsfig}
\usepackage{verbatim,enumerate,array,booktabs,color,bigstrut,prettyref,tikz-cd}
\usepackage{multirow}
\usepackage[all]{xy}
\usepackage[backref]{hyperref}
\usepackage[OT2,T1]{fontenc}
\usepackage{ctable}

\usepackage{mathtools}

\newrefformat{eq}{\textup{(\ref{#1})}}
\newrefformat{prty}{\textup{(\ref{#1})}}

\definecolor{mylinkcolor}{rgb}{0.8,0,0}
\definecolor{myurlcolor}{rgb}{0,0,0.8}
\definecolor{mycitecolor}{rgb}{0,0,0.8}
\hypersetup{colorlinks=true,urlcolor=myurlcolor,citecolor=mycitecolor,linkcolor=mylinkcolor,linktoc=page,breaklinks=true}

\DeclareSymbolFont{cyrletters}{OT2}{wncyr}{m}{n}
\DeclareMathSymbol{\Sha}{\mathalpha}{cyrletters}{"58}

\newtheorem{defn}{Definition}[section]
\newtheorem{definition}[defn]{Definition}

\newtheorem{corollary}[defn]{Corollary}
\newtheorem{lemma}[defn]{Lemma}

\newtheorem{thm}[defn]{Theorem}

\newtheorem{proposition}[defn]{Proposition}

\theoremstyle{definition}

\newtheorem*{ack}{Acknowledgements}
\newtheorem{remark}[defn]{Remark}

\newtheorem{example}[defn]{Example}

\newcommand{\QQ}{\mathbb Q}

\newcommand{\ZZ}{\mathbb Z}
\newcommand{\Z}{\mathbb Z}
\newcommand{\Q}{\mathbb{Q}}

\newcommand{\FF}{\mathbb F}

\newcommand{\arrow}{\longrightarrow}

\newcommand{\tor}{\mathrm{tors}}

\newcommand\myiso{\stackrel{\mathclap{\normalfont\mbox{\small $p$}}}{-}}
\newcommand\myisot{\stackrel{\mathclap{\normalfont\mbox{\small $3$}}}{-}}

\begin{document}
	
	
	
	\title[2-adic Galois images of non-CM isogeny-torsion graphs defined over $\QQ$]{2-adic Galois images of non-CM isogeny-torsion graphs defined over $\QQ$}
	
\author{Garen Chiloyan}
\email{garen.chiloyan@gmail.com} 
\urladdr{https://sites.google.com/view/garenmath/home}




\begin{abstract}
		Let $\mathcal{E}$ be a $\QQ$-isogeny class of elliptic curves defined over $\QQ$ without CM. The isogeny graph associated to $\mathcal{E}$ is a graph which has a vertex for each elliptic curve in $\mathcal{E}$ and an edge for each $\QQ$-isogeny of prime degree that maps one elliptic curve in $\mathcal{E}$ to another elliptic curve in $\mathcal{E}$, with the degree recorded as a label of the edge. An isogeny-torsion graph is an isogeny graph where, in addition, we label each vertex with the abstract group structure of the torsion subgroup over $\QQ$ of the corresponding elliptic curve. Then, the main statement of the article is a classification of the $2$-adic Galois Image of Galois that occurs at each vertex of all isogeny-torsion graphs consisting of elliptic curves defined over $\QQ$ without CM.  
	\end{abstract}

\maketitle


\section{Introduction}\label{section introduction}

Let $E$ and $E'$ be elliptic curves defined over $\QQ$. Then $E$ is said to be $\QQ$-isogenous to $E'$ if there is a non-constant isogeny $\phi \colon E \to E'$ defined over $\QQ$. This relation is an equivalence relation. Let $\mathcal{E}$ be the $\QQ$-isogeny class of $E$; the set of elliptic curves over $\QQ$, up to isomorphism, that are $\QQ$-isogenous to $E$. A theorem of Kenku states that $\mathcal{E}$ is a set consisting of $1 - 4$, $6$, or $8$ curves.

The isogeny graph associated to $\mathcal{E}$ is a graph that represents as a vertex each elliptic curve defined over $\QQ$ in the $\QQ$-isogeny class $\mathcal{E}$, and the edges represent $\QQ$-isogenies of prime degree that map one curve in $\mathcal{E}$ to another curve in $\mathcal{E}$, with the degree recorded as a label of the edge. An isogeny-torsion graph is an isogeny graph where, in addition, we label each vertex with the abstract group structure of the torsion subgroup over $\QQ$ of the corresponding elliptic curve. Two isogeny-torsion graphs are the same if they have the same number of vertices and all of the labels of the vertices (torsion subgroups) and edges (degrees of isogenies) are the same. The main goal of this article is to classify the $2$-adic Galois image at each vertex of all isogeny-torsion graphs attached to $\QQ$-isogeny classes of elliptic curves defined over $\QQ$ without CM.

\begin{example}\label{T4 ITG example}

There are four elliptic curves in the $\QQ$-isogeny class with LMFDB (\cite{lmfdb}) label \texttt{17.a} which we denote $E$, $E'$, $E''$, and $E'''$. The isogeny graph associated to \texttt{17.a} is shown below on the left and the isogeny-torsion graph associated to \texttt{17.a} is shown below on the right.

\begin{center} $\begin{tikzcd}
     & E'''                                 &      \\
     & E \arrow[u, "2", no head] \arrow[ld, "2"', no head] \arrow[rd, "2", no head] &      \\
E' &                                      & E''
\end{tikzcd} \hspace{10mm} \begin{tikzcd}
     & \Z / 2 \Z                                 &      \\
     & \Z / 2 \Z \times \Z / 2 \Z \arrow[u, "2", no head] \arrow[ld, "2"', no head] \arrow[rd, "2", no head] &      \\
\Z / 4 \Z &                                      & \Z / 4 \Z
\end{tikzcd}$
\end{center}
\end{example}
The rational isogenies of elliptic curves defined over $\QQ$ have been described completely in the literature. One of the most important milestones in the classification of rational isogenies was \cite{mazur1}, where Mazur dealt with the case of isogenies of prime degree. The complete classification of rational isogenies, for prime or composite degree, was completed due to work of Fricke, Kenku, Klein, Kubert, Ligozat, Mazur  and Ogg, among others (see Theorem \ref{thm-ratnoncusps} below, and \cite{lozano0}, Section 9). In particular, the work of Kenku \cite{kenku} shows that there are at most $8$ elliptic curves in each isogeny class over $\QQ$ (see Theorem \ref{thm-kenku} below). Theorem \ref{thm-mainisogenygraphs} follows directly from the classification of rational isogenies over $\QQ$.

\begin{thm}[Theorem 1.2, \cite{gcal-r}]\label{thm-mainisogenygraphs}
	There are $26$ isomorphism types of isogeny graphs that are associated to $\QQ$-isogeny classes of elliptic curves defined over $\QQ$. More precisely, there are sixteen types of (linear) $L_{k}$ graphs of $k = 1-4$ vertices, three types of (nonlinear two-primary torsion) $T_{k}$ graphs of $k = 4$, $6$, or $8$ vertices, six types of (rectangular) $R_k$ graphs of $k = 4$ or $6$ vertices, and one (special) $S$ graph.
\end{thm}

The isogeny class degree of an isogeny graph (and also of the $\QQ$-isogeny class) is the greatest degree of a finite, cyclic, $\QQ$-isogeny on the isogeny graph. In the case of a linear graph of $L_{2}$ or $L_{3}$ type or in the case of a rectangular graph of $R_{4}$ type, the isogeny class degree of the isogeny graph is written in parentheses to distinguish it from other isogeny graphs of the same size and shape. There are examples of isogeny graphs of $R_{4}$ type that are generated by a $10$-isogeny and there are examples of an isogeny graph of $R_{4}$ type that are generated by a $14$-isogeny. One cannot distinguish an $R_{4}(10)$ graph from an $R_{4}(14)$ graph simply by the size and shape of the graph (see below).
\begin{center} \begin{tikzcd}
E_{1} \arrow[d, no head, "5"'] \arrow[r, no head, "2"] & E_{2} \arrow[d, no head, "5"] \\
E_{3} \arrow[r, no head, "2"']                & E_{4}               
\end{tikzcd} \hspace{10mm} \begin{tikzcd}
E_{1} \arrow[d, no head, "7"'] \arrow[r, no head, "2"] & E_{2} \arrow[d, no head, "7"] \\
E_{3} \arrow[r, no head, "2"']                & E_{4}               
\end{tikzcd} \end{center}
The main theorem in \cite{gcal-r} was the classification of isogeny-torsion graphs that occur over $\QQ$.
\begin{thm}[Theorem 1.3, \cite{gcal-r}] There are $52$ isomorphism types of isogeny-torsion graphs that are associated to elliptic curves defined over $\QQ$. In particular, there are $23$ types of $L_{k}$ graphs, $13$ types of $T_{k}$ graphs, $12$ types of $R_{k}$ graphs, and $4$ types of $S$ graphs. Moreover, there are $16$ isomorphism types of isogeny-torsion graphs that are associated to elliptic curves over $\QQ$ with complex multiplication.
\end{thm}

\begin{example}
Reconsider the $\QQ$-isogeny class \texttt{17.a}. The image of the $2$-adic Galois representation attached to the elliptic curves in the $\QQ$-isogeny class are given below (with RZB labels. See \cite{rouse}).
\begin{center} $\begin{tikzcd}
     & \operatorname{H}_{242b} &      \\
     & \operatorname{H}_{210b} \arrow[u, "2", no head] \arrow[ld, "2"', no head] \arrow[rd, "2", no head] &      \\
\operatorname{H}_{221a} & & \operatorname{H}_{242d}
\end{tikzcd}$
\end{center}
\end{example}
Let $\mathcal{E}$ be a $\QQ$-isogeny class of elliptic curves defined over $\QQ$ without CM and let $\mathcal{G}$ be the isogeny-torsion graph associated to $\mathcal{E}$. We will say that the $2$-adic Galois image attached to $\mathcal{G}$ is the data of the classification of the $2$-adic Galois image attached to each elliptic curve represented by the vertices of $\mathcal{G}$. The main theorem of this paper is the following:

\begin{thm}\label{main theorem}
Let $\mathcal{G}$ be an isogeny-torsion graph associated to a $\QQ$-isogeny class of non-CM elliptic curves defined over $\QQ$. Then
\begin{enumerate}
    \item if $\mathcal{G}$ is of $L_{1}$ type, then the $2$-adic Galois Image attached to $\mathcal{G}$ is one of the $22$ arrangements in Table \ref{2-adic Galois Images of L1 Graphs},
    \item if $\mathcal{G}$ is of $L_{2}(2)$ type, then the $2$-adic Galois Image attached to $\mathcal{G}$ is one of the $80$ arrangements in Table \ref{2-adic Galois Images of L22 Graphs},
    \item if $\mathcal{G}$ is of $T_{4}$ type, then the $2$-adic Galois Image attached to $\mathcal{G}$ is one of the $60$ arrangements in Table \ref{2-adic Galois Images of T4 Graphs},
    \item if $\mathcal{G}$ is of $T_{6}$ type, then the $2$-adic Galois Image attached to $\mathcal{G}$ is one of the $81$ arrangements in Table \ref{2-adic Galois Images of T6 Graphs},
    \item if $\mathcal{G}$ is of $T_{8}$ type, then the $2$-adic Galois Image attached to $\mathcal{G}$ is one of the $53$ arrangements in Table \ref{2-adic Galois Images of T8 Graphs},
    \item if $\mathcal{G}$ is of $S$ type, then the $2$-adic Galois Image attached to $\mathcal{G}$ is one of the $5$ arrangements in Table \ref{2-adic Galois Images of S Graphs Table},
    \item if $\mathcal{G}$ is of $R_{6}$ type, then the $2$-adic Galois Image attached to $\mathcal{G}$ is one of the $2$ arrangements in Table \ref{2-adic Galois Images of R6 Graphs Table},
    \item if $\mathcal{G}$ is of $R_{4}$ type, then the $2$-adic Galois Image attached to $\mathcal{G}$ is one of the $13$ arrangements in Table \ref{2-adic Galois Images of R4 Graphs Table},
    \item if $\mathcal{G}$ is of $L_{3}(9)$ or $L_{3}(25)$ type, then the $2$-adic Galois Image attached to $\mathcal{G}$ is conjugate to $\operatorname{GL}(2, \ZZ_{2})$ (see Table \ref{2-adic Galois Images of L3 Graphs Table}),
    \item if $\mathcal{G}$ is of $L_{2}(p)$ type where $p$ is an odd prime, then the $2$-adic Galois Image attached to $\mathcal{G}$ is one of the $34$ arrangements in Table \ref{2-adic Galois Images of L2 Graphs Table Odd}.
\end{enumerate}
\end{thm}

In Section \ref{section background} we will talk about the necessary algebraic properties of elliptic curves, isogeny graphs and isogeny-torsion graphs. Table \ref{L_k graphs} - \ref{R_k graphs} in Section \ref{section background} has examples of all isogeny-torsion graphs defined over $\QQ$. Section \ref{section work by rouse and zureick-brown} briefly touches on the work by Rouse and Zureick-Brown, classifying the $2$-adic Galois image attached to non-CM elliptic curves defined over $\QQ$. Section \ref{section lemmas and corollaries} has many lemmas and corollaries, culminating in Lemma \ref{2-adic Galois Images}, Corollary \ref{coprime isogeny-degree}, and Corollaries \ref{2-adic Galois Images corollary 1} - \ref{2-adic Galois Images corollary 4} which will be used to classify the $2$-adic Galois Images attached to non-CM isogeny-torsion graphs defined over $\QQ$.

\begin{ack}
The author would like to express his utmost gratitude to \'Alvaro Lozano-Robledo, Jeremy Rouse, Drew Sutherland, and David Zureick-Brown for their patience and many helpful conversations on this topic.
\end{ack}

\section{Philosophy and structure of the paper}\label{philosophy}

Originally, the intention of this project was to classify the $2$-adic Galois Image attached to all isogeny-torsion graphs over $\QQ$ (with the cases of CM and non-CM in one paper). It soon became clear the methodology of the proofs between non-CM isogeny-torsion graphs and CM isogeny-torsion graphs are rather different and it is better to split the classification into two cases; though in essence it is the same question. The main ideas motivating the papers \cite{CHILOYANinfinite}, \cite{CMisogenytorsiongraphs}, \cite{gcal-r}, and this paper is to think about elliptic curve theory, not necessarily from the viewpoint of individual elliptic curves over $\QQ$, but $\QQ$-isogeny classes of elliptic curves defined over $\QQ$.

The main result in \cite{gcal-r} is the classification of the isogeny-torsion graphs associated to $\QQ$-isogeny classes of elliptic curves defined over $\QQ$. Originally, the authors in \cite{gcal-r} wanted to classify the torsion subgroups of a pair of $\QQ$-isogenous elliptic curves defined over $\QQ$, which was extended to classifying the torsion subgroups of all elliptic curves in a $\QQ$-isogeny class, which was extended to the main result of \cite{gcal-r}.

Let $E/\QQ$ be an elliptic curve, let $\mathcal{E}$ be the $\QQ$-isogeny class of $E$, and let $\mathcal{G}$ be the isogeny-torsion graph associated to $\mathcal{E}$. As $\mathcal{G}$ consists of elliptic curves over $\QQ$, we may say that $\mathcal{G}$ is defined over $\QQ$. The elliptic curve $E$ has CM, if and only if all of the elliptic curves over $\QQ$ in $\mathcal{E}$ have CM. Thus, we may say that $\mathcal{G}$ has CM or not. If we determine the $2$-adic Galois Image attached to all elliptic curves over $\QQ$ in $\mathcal{E}$, then in some sense, it is like determining the $2$-adic Galois Image of $\mathcal{E}$ and the $2$-adic Galois Image attached to $\mathcal{G}$. Let $d$ be a non-zero integer and let $E^{(d)}$ be the quadratic twist of $E$ by $d$. Then the $\QQ$-isogeny class of $E^{(d)}$ consists of the quadratic twist of all elliptic curves over $\QQ$ in $\mathcal{E}$ by $d$. Thus, taking the quadratic twist of $E$ by $d$ would in some sense, cause $\mathcal{E}$ and $\mathcal{G}$ to be twisted by $d$. So we may consider the isogeny-torsion graph associated to the $\QQ$-isogeny class of $E^{(d)}$ to be the quadratic twist of $\mathcal{G}$ by $d$.

With the discussion of the previous paragraph, many of the algebraic properties of an elliptic curve over $\QQ$ can be extended to its associated $\QQ$-isogeny class and to its associated isogeny-torsion graph. For example, one may investigate if $E$ has CM or not, compute $E(\QQ)_{\texttt{tors}}$ and $\rho_{E,2^{\infty}}(G_{\QQ})$, and take quadratic twists of $E$, etc.; all properties that could be, in some sense, extend to the isogeny-torsion graph $\mathcal{G}$ associated to $E$. In this way, $\mathcal{G}$ is a mathematical object that shares many of the algebraic properties of $E$.

Section \ref{section background} will provide a very brief introduction to elliptic curves and their associated objects for example, isogeny graphs and isogeny-torsion graphs. The section concludes with a quick introduction to quadratic twists of elliptic curves and an extension of the definition to groups (to apply to the image of Galois representations). Section \ref{section lemmas and corollaries} contains many lemmas of a group-theoretic flavor. These lemmas will serve as the backbone of the paper. The key lemma in Section \ref{section lemmas and corollaries} is Lemma \ref{2-adic Galois Images}. Let $E/\QQ$ and $E'/\QQ$ be elliptic curves that are isogenous by an isogeny with a cyclic kernel. Then given $\rho_{E,2^{\infty}}(G_{\QQ})$, Lemma \ref{2-adic Galois Images} gives us a way to determine $\rho_{E',2^{\infty}}(G_{\QQ})$. In other words, the $2$-adic Galois image attached to any single elliptic curve in a $\QQ$-isogeny class of elliptic curves over $\QQ$ determines the $2$-adic Galois image attached to all other elliptic curves over $\QQ$ in the $\QQ$-isogeny class (see also Corollary \ref{coprime isogeny-degree} and Corollary \ref{2-adic Galois Images corollary 1} - Corollary \ref{2-adic Galois Images corollary 4}).

Let $E/\QQ$ be a non-CM elliptic curve. Then by the work in \cite{rouse}, $\rho_{E,2^{\infty}}(G_{\QQ})$ is the full lift of $\overline{\rho}_{E,32}(G_{\QQ})$. The authors of \cite{rouse} organized the classification of the image of the $2$-adic Galois representation attached to non-CM elliptic curves over $\QQ$ into the RZB database. Equipped with Lemma \ref{2-adic Galois Images} and the RZB database, one can classify the image of the $2$-adic Galois representation attached to all isogeny-torsion graphs of type $L_{1}$, $L_{2}(2)$, $T_{4}$, $T_{6}$, and $T_{8}$. In other words, one can construct Table \ref{2-adic Galois Images of L1 Graphs} - Table \ref{2-adic Galois Images of T8 Graphs}. Actually, there is a shortcut. The LMFDB has examples of all but twelve arrangements in the RZB database. To classify the image of the $2$-adic Galois represenatation attached to all isogeny-torsion graphs of, say, $T_{4}$, type, one needs to determine all subgroups $H$ of $\operatorname{GL}(2, \ZZ_{2})$ that can serve as the $2$-adic Galois image attached to any ``ideal'' elliptic curve over $\QQ$ in $T_{4}$. The author prefers to work with the elliptic curve in a $T_{4}$ graph with full two-torsion defined over $\QQ$. One then searches the RZB database for $H$, takes the example elliptic curve $E/\QQ$ listed in the reference such that $\rho_{E,2^{\infty}}(G_{\QQ})$ is conjugate to $H$, look up $E$ in the LMFDB (if possible), and write down the image of the $2$-adic Galois representation attached to the elliptic curves over $\QQ$ in the $\QQ$-isogeny class associated to $E$. This is enough to classify the image of the $2$-adic Galois representation attached to all isogeny-torsion graphs that \textbf{do not} have an isogeny of odd prime degree.

Let $M$ and $N$ be integers greater than or equal to $2$ that are coprime. Let $\operatorname{J}$ be a subgroup of $\operatorname{GL}(2, \ZZ / MN \ZZ) \approx \operatorname{GL}(2, \ZZ / M \ZZ) \times \operatorname{GL}(2, \ZZ / N \ZZ)$. Then $\operatorname{J}$ is said to be a product group and is isomorphic to $\operatorname{J}_{M} \times \operatorname{J}_{N}$ for some subgroup $\operatorname{J}_{M}$ of $\operatorname{GL}(2, \ZZ / M \ZZ)$ and some subgroup $\operatorname{J}_{N}$ of $\operatorname{GL}(2, \ZZ / N \ZZ)$. The modular curve generated by $\operatorname{J}$ is said to be the fiber product of $\operatorname{J}_{M}$ and $\operatorname{J}_{N}$. The most difficult part of this project was classifyng the image of the $2$-adic Galois representation attached to isogeny-torsion graphs with an isogeny of odd prime degree. For example, is it possible for an isogeny-torsion graph of $L_{2}(3)$ type with torsion configuration $([3],[1])$ to have $2$-adic Galois configuration $(\operatorname{H}_{20a}, \operatorname{H}_{20a})$? The answer to this question is no and the proof requires the classification of the rational points on the fiber product of $\operatorname{H}_{20a}$ and $\operatorname{X}_{1}(3)$; in other words, the modular curve generated by $\operatorname{H}_{20a} \times \operatorname{X}_{1}(3)$.

There is an interplay between reduction and group containment that reduces the number of modular curves we need to investigate. For example, let $\operatorname{H} \leq \operatorname{H}'$ be subgroups of $\operatorname{GL}(2, \ZZ / N \ZZ)$ for some positive integer $N$ and let $\operatorname{X}$ be the modular curve generated by $\operatorname{H}$ and let $\operatorname{X}'$ be the modular curve generated by $\operatorname{H}'$. If $E/\QQ$ is a non-CM elliptic curve such that $\overline{\rho}_{E,N}(G_{\QQ})$ is not conjugate to a subgroup of $\operatorname{H}'$, then $\overline{\rho}_{E,N}(G_{\QQ})$ is not conjugate to a subgroup of $\operatorname{H}$. Similarly, if all of the rational points on $\operatorname{X}'$ are CM points or cusps, then all of the rational points on $\operatorname{X}$ are CM points or cusps. The reason why this is true is because there is a commutative diagram

$$\begin{tikzcd}
\operatorname{X}_{\operatorname{H}} \arrow[rrd, "\operatorname{\pi}_{\operatorname{H}}"'] \arrow[rr, "\phi"] &  & \operatorname{X}_{\operatorname{H'}} \arrow[d, "\pi_{\operatorname{H'}}"] \\
                                                                                                             &  & \mathbb{P}^{1}(\mathbb{Q})                                               
\end{tikzcd}$$
where $\pi_{\operatorname{H}}$ and $\pi_{\operatorname{H}'}$ are the maps from $\operatorname{X}$ and $\operatorname{X}'$, respectively, to the \textit{j}-line (see Section 2 in \cite{SZ}). Now let $N$ and $N'$ be positive integers such that $N$ divides $N'$. Let $\operatorname{K}$ be a subgroup of $\operatorname{GL}(2, \ZZ / N \ZZ)$ and let $\operatorname{K}'$ be a subgroup of $\operatorname{GL}(2, \ZZ / N' \ZZ)$ such that the reduction of $\operatorname{K}'$ modulo $N$ is conjugate to $\operatorname{K}$. Denote the modular curve generated by $\operatorname{K}$ to be $\operatorname{Y}$ and denote the modular curve generated by $\operatorname{K}'$ to be $\operatorname{Y}'$. If $E/\QQ$ is a non-CM elliptic curve such that $\overline{\rho}_{E,N}(G_{\QQ})$ is not conjugate to $\operatorname{K}$, then $\overline{\rho}_{E,N'}(G_{\QQ})$ is not conjugate to $\operatorname{K}'$. Similarly, if all of the rational points on $\operatorname{Y}$ are CM points or cusps, then all of the rational points on $\operatorname{Y}'$ are CM points or cusps. In Section \ref{Determination of 2-adic Galois Images}, we start the classification of the rational points on the fiber products in question. We leave the fiber products of genus $1$ to Section \ref{elliptic curves} and conclude. The labels for the groups used in the proofs come from the RZB database. On the other hand, in the tables, the groups are denoted of the form $A = \texttt{N.i.g.n}$ where $A$ is the label of some group in the RZB database and $\texttt{N.i.g.n}$ is the label of that same group coming from the RSVZB database. The term $N$ denotes the level of the group, $i$ the index of the group in $\operatorname{GL}(2, \ZZ / N \ZZ)$, $g$ denotes the genus of the modular curve generated by the group, and $n$ is a tiebreaker (see pages 9-10 in \cite{Rouse2021elladicIO} for how the groups are organized). Finally, for an elliptic curve $E/\QQ$ and a positive integer $N$ we denote the image of the $2$-adic Galois representation attached to $E$ and the image of the mod-$N$ Galois representation attached to $E$ as $\rho_{E,2^{\infty}}(G_{\QQ})$ and $\overline{\rho}_{E,N}(G_{\QQ})$, respectively. The \textit{j}-invariants were computed with the results in the RZB database and the SZ database (see the tables at the end of \cite{SZ}) or at times taken from the LMFDB. Models of fine modular curves were taken from the LMFDB.

\section{Background}\label{section background}

\subsection{Elliptic curves, isogeny graphs, and isogeny-torsion graphs}

Let $E/\QQ$ be an elliptic curve. Then $E$ has the structure of an abelian group with identity $\mathcal{O}$. Let $N$ be a positive integer. The set of all points on $E$ of order dividing $N$ with coordinates in $\overline{\QQ}$ is a group, denoted $E[N]$ and is isomorphic to $\ZZ / N \ZZ \times \ZZ / N \ZZ$. An element of $E[N]$ is called an $N$-torsion point. Let $E'/\QQ$ be an elliptic curve. An isogeny mapping $E$ to $E'$ is a non-constant rational morphism $\phi \colon E \to E'$ that maps the identity of $E$ to the identity of $E'$. An isogeny is a group homomorphism with kernel of finite order. The degree of an isogeny agrees with the order of its kernel.

Let $M$ be an integer and let $[M] \colon E \to E$ be the map such that
\begin{center}
    $\begin{cases}
    [M](P) = \underbrace{P +  \ldots + P}_{\text{M}} & M \geq 1 \\
    [M](P) = \underbrace{(-P) +  \ldots + (-P) }_{\text{-M}} & M \leq -1 \\
    [M](P) = \mathcal{O} & M = 0.
    \end{cases}$
\end{center}
Maps of the form $[M]$ are called multiplication-by-$M$ maps. If $M$ is a non-zero integer, then the degree of $[M]$ is equal to $\left| M \right|^{2}$. The multiplication-by-$M$ maps are elements of the endomorphism ring of $E$, $\operatorname{End}(E)$. If $\operatorname{End}(E)$ contains a map that is not a multiplication-by-$M$ map, then $E$ is said to have complex multiplication (or that $E$ is CM). Otherwise $E$ does not have complex multiplication (or that $E$ is non-CM). If $E$ has CM, then $\operatorname{End}(E)$ is ring-isomorphic to an order in a quadratic field.

\begin{example}
Let $E$ be the elliptic curve with LMFDB label \texttt{11.a1}. Then $E$ does not have CM. In other words, $\operatorname{End}(E) \cong \ZZ$.
\end{example}

\begin{example}
Let $E$ be the elliptic curve $y^{2}z = x^{3} - xz^{2}$. Consider the isogeny $[i] \colon E \to E$ that maps $\mathcal{O} = [0:1:0]$ to $\mathcal{O}$ and maps a point $(a:b:1)$ in $E$ to the point $(-a:ib:1)$. The degree of $[i]$ is equal to $1$ as non-zero elements of $E$ are mapped by $[i]$ to non-zero elements of $E$. As $[i]$ is not equal to the identity map nor the inversion map, $[i]$ is an endomorphism of $E$ that is not a multiplication-by-$M$ map. Hence, $E$ has CM and $\operatorname{End}(E) = \ZZ + [i] \cdot \ZZ \cong \ZZ[i]$. Note that the $i$ in $\ZZ + [i] \cdot \ZZ$ designates the map $[i]$ and the $i$ in $\ZZ[i]$ designates a root of the polynomial $x^{2}+1$. 
\end{example}

The group $G_{\QQ}:= \operatorname{Gal}\left(\overline{\QQ}/\QQ\right)$ acts on $E[N]$ for all positive integers $N$. From this action, we have the mod-$N$ Galois representation attached to $E$:
$$ \overline{\rho}_{E,N} \colon G_{\QQ} \to \operatorname{Aut}(E[N]).$$
After identifying $E[N] \cong \ZZ / N \ZZ \times \ZZ / N \ZZ$ and fixing a set of (two) generators of $E[N]$, we may consider the mod-$N$ Galois representation attached to $E$ as
$$ \overline{\rho}_{E,N} \colon G_{\QQ} \to \operatorname{GL}(2,\ZZ / N \ZZ).$$
Moreover, for a prime $\ell$, we have $\rho_{E,\ell^{\infty}}(G_{\QQ}) = \varprojlim_{N \geq 1}\overline{\rho}_{E,\ell^{N}}(G_{\QQ})$. The group $\rho_{E,2^{\infty}}(G_{\QQ})$ is of level $2^{m}$ if $m$ is the least non-negative integer such that for each positive integer $n$, $\overline{\rho}_{E,2^{m+n}}(G_{\QQ})$ is the full lift of $\overline{\rho}_{E,2^{m}}(G_{\QQ})$ inside the group $\operatorname{GL}(2, \ZZ / 2^{m+n} \ZZ)$. For example, if $\rho_{E,2^{\infty}}(G_{\QQ}) = \operatorname{GL}(2, \ZZ_{2})$ then the level of $\rho_{E,2^{\infty}}(G_{\QQ})$ is equal to $1$. Let $u$ be an element of $\left(\ZZ / N \ZZ \right)^{\times}$. By the properties of the Weil pairing, there exists an element of $\overline{\rho}_{E,N}(G_{\QQ})$ whose determinant is equal to $u$. Moreover, $\overline{\rho}_{E,N}(G_{\QQ})$ has an element that behaves like complex conjugation. If $E$ is non-CM, then by Lemma 2.8 of \cite{SZ}, $\overline{\rho}_{E,N}(G_{\Q})$ has an element conjugate to $\left(\begin{array}{cc}
    1 & 1 \\
    0 & -1
\end{array}\right)$ or $\left(\begin{array}{cc}
    1 & 0 \\
    0 & -1
\end{array}\right)$ that represents complex conjugation.

\begin{definition}
Let $E/\QQ$ be a (homogenized) elliptic curve. A point $P$ on $E$ is said to be defined over $\QQ$ or rational if $P = [a:b:c]$ for some $a,b,c \in \QQ$.
\end{definition}
The set of all points on $E$ defined over $\QQ$ is denoted $E(\QQ)$. By the Mordell--Weil theorem, $E(\QQ)$ has the structure of a finitely-generated abelian group. Let $E(\QQ)_{\text{tors}}$ denote the set of all points on $E$ defined over $\QQ$ of finite order.

\begin{thm}[Mazur \cite{mazur1}]\label{thm-mazur}
		Let $E/\Q$ be an elliptic curve. Then
		\[
		E(\Q)_\tor\simeq
		\begin{cases}
		\Z/M\Z &\text{with}\ 1\leq M\leq 10\ \text{or}\ M=12,\ \text{or}\\
		\Z/2\Z\oplus \Z/2N\Z &\text{with}\ 1\leq N\leq 4.
		\end{cases}
		\]
	\end{thm}
Moreover, each of the fifteen torsion subgroups occur for infinitely many \textit{j}-invariants. We now move on to the possible isogenies on elliptic curves over $\QQ$ with finite, cyclic kernel.

\begin{defn}
Let $E/\QQ$ be an elliptic curve. We say a subgroup $H$ of $E$ of finite order is said to be $\QQ$-rational if $\sigma(H) = H$ for all $\sigma \in G_{\QQ}$.
\end{defn}

\begin{remark}

Let $E/\QQ$ be an elliptic curve and let $P$ be a point on $E$ defined over $\QQ$ of finite order. Then the group generated by $P$ is certainly $\QQ$-rational. In general, the elements of a $\QQ$-rational subgroup of $E$ need not be \textit{fixed} by the action of $G_{\QQ}$. For example, $E[3]$ is a $\QQ$-rational group but by Theorem \ref{thm-mazur}, $G_{\QQ}$ fixes one or three of the nine elements of $E[3]$.
\end{remark}
	
\begin{lemma}[III.4.12, \cite{Silverman}]\label{Q-rational}
Let $E/\QQ$ be an elliptic curve. Then for each finite, cyclic, $\QQ$-rational subgroup $H$ of $E$, there is a unique elliptic curve defined over $\QQ$ up to isomorphism denoted $E / H$, and an isogeny $\phi_{H} \colon E \to E / H$ with kernel $H$.
\end{lemma}

\begin{remark}

Note that it is only the elliptic curve $E/H$ that is unique (up to isomorphism) but the isogeny $\phi_{H}$ is not. For any isogeny $\phi$, the isogeny $-\phi$ has the same domain, codomain, and kernel as $\phi$. Moreover, for any positive integer $N$, the morphisms $\phi$ and $[N] \circ \phi$ have the same domain and the same codomain. This is why the bijection in Lemma \ref{Q-rational} is with \textit{cyclic}, $\QQ$-rational subgroups of an elliptic curve instead of with all $\QQ$-rational subgroups of an elliptic curve.
\end{remark}

The $\QQ$-rational points on the modular curves $\operatorname{X}_{0}(N)$ have been described completely in the literature, for all $N \geq 1$. One of the most important milestones in their classification was \cite{mazur1}, where Mazur dealt with the case when $N$ is prime. The complete classification of $\QQ$-rational points on $\operatorname{X}_{0}(N)$, for any $N$, was completed due to work by Fricke, Kenku, Klein, Kubert, Ligozat, Mazur  and Ogg, among others (see the summary tables in \cite{lozano0}).
	
	\begin{thm}\label{thm-ratnoncusps} Let $N$ be a positive integer such that $\operatorname{X}_{0}(N)$ has a non-cuspidal $\QQ$-rational point. Then:
		\begin{enumerate}
			\item $N\leq 10$, or $N= 12,13, 16,18$ or $25$. In this case $\operatorname{X}_{0}(N)$ is a curve of genus $0$ and its $\QQ$-rational points form an infinite $1$-parameter family, or
			\item $N=11,14,15,17,19,21$, or $27$. In this case $\operatorname{X}_{0}(N)$ is a curve of genus $1$, i.e.,~$\operatorname{X}_{0}(N)$ is an elliptic curve over $\QQ$, but in all cases the Mordell-Weil group $\operatorname{X}_{0}(N)(\QQ)$ is finite, or 
			
			\item $N=37,43,67$ or $163$. In this case $\operatorname{X}_{0}(N)$ is a curve of genus $\geq 2$ and (by Faltings' theorem) there are only finitely many $\QQ$-rational points on $\operatorname{X}_{0}(N)$, all of which are known explicitly.
		\end{enumerate}
	\end{thm}

\begin{defn}
		Let $E/\QQ$ be an elliptic curve. We define $C(E)$ to be the number of distinct finite $\QQ$-rational cyclic subgroups of $E$ (including the trivial subgroup), and we define $C_{p}(E)$ similarly to $C(E)$ but only counting $\QQ$-rational cyclic subgroups of $E$ of order a power of $p$ (like in the definition of $C(E)$, this includes the trivial subgroup), for each prime $p$. 
	\end{defn}

Notice that it follows from the definition that $C(E)=\prod_{p} C_{p}(E)$. 
	
	\begin{thm}[Kenku, \cite{kenku}]\label{thm-kenku} There are at most eight $\QQ$-isomorphism classes of elliptic curves in each $\QQ$-isogeny class. More concretely, let $E / \QQ$ be an elliptic curve. Then $C(E)=\prod_{p} C_{p}(E) \leq 8$. Moreover, each factor $C_{p}(E)$ is bounded as follows:
		\begin{center}
			\begin{tabular}{c|ccccccccccccc}
				$p$ & $2$ & $3$ & $5$ & $7$ & $11$ & $13$ & $17$ & $19$ & $37$ & $43$ & $67$ & $163$ & \text{else}\\
				\hline 
				$C_p\leq $ & $8$ & $4$ & $3$ & $2$ & $2$ & $2$ & $2$ & $2$ & $2$ & $2$ & $2$ & $2$ & $1$.
			\end{tabular}
		\end{center}
		Moreover:
		\begin{enumerate}
			\item If $C_{p}(E) = 2$ for a prime $p$ greater than $7$, then $C_{q}(E) = 1$ for all other primes $q$. 
			\item Suppose $C_{7}(E) = 2$, then $C(E) \leq 4$. Moreover, we have $C_{3}(E) = 2$, or $C_{2}(E) = 2$, or $C(E) = 2$.
			\item $C_{5}(E) \leq 3$ and if $C_{5}(E) = 3$, then $C(E) = 3$.
			\item If $C_{5}(E) = 2$, then $C(E) \leq 4$. Moreover, either $C_{3}(E) = 2$, or $C_{2}(E) = 2$, or $C(E) = 2$. 
			\item $C_{3}(E) \leq 4$ and if $C_{3}(E) = 4$, then $C(E) = 4$. 
			\item If $C_{3}(E) = 3,$ then $C(E) \leq 6$. Moreover, $C_{2}(E) = 2$ or $C(E) = 3$.
			\item If $C_{3}(E) = 2$, then $C_{2}(E) \leq 4$.
		\end{enumerate}
	\end{thm}

Instead of viewing each elliptic curve over $\QQ$ in a $\QQ$-isogeny class individually, we can view them all together. The best way to visualize the $\QQ$-isogeny class is to use isogeny graphs. When present, the subscript $k$ is the number of vertices of an isogeny graph. When present, the integer $(N)$ denotes the maximal degree of an isogeny in the isogeny graph. Note that if present, $N$ equals the isogeny class degree of the isogeny class.

\begin{thm}[Theorem 1.2, \cite{gcal-r}]
	There are $26$ isomorphism types of isogeny graphs that are associated to $\QQ$-isogeny classes of elliptic curves defined over $\QQ$. More precisely, there are $16$ types of (linear) $L_{k}$ graphs (with $k = 1 - 4$ vertices), $3$ types of (nonlinear two-primary torsion) $T_{k}$ graphs (with $k = 4$, $6$, or $8$ vertices), $6$ types of (rectangular) $R_{k}$ graphs (with $k = 4$ or $6$ vertices), and $1$ type of (special) $S$ graph.
	
	Moreover, there are $11$ isomorphism types of isogeny graphs that are associated to elliptic curves over $\QQ$ with complex multiplication, namely the types $L_{2}(p)$ for $p=2,3,11,19,43,67,163$, $L_{4}$, $T_{4}$, $R_{4}(6)$, and $R_{4}(14)$. Finally, the isogeny graphs of type $L_{4}$, $R_{4}(14)$, and $L_{2}(p)$ for $p \in \{19, 43, 67, 167\}$ occur exclusively for elliptic curves with CM.
	\end{thm}

The main theorem in \cite{gcal-r} was the classification of isogeny-torsion graphs that occur over $\QQ$
\begin{thm}[Theorem 1.3, \cite{gcal-r}] There are $52$ isomorphism types of isogeny-torsion graphs that are associated to $\QQ$-isogeny classes of elliptic curves defined over $\QQ$. In particular, there are $23$ types of $L_{k}$ graphs (see Table \ref{L_k graphs}), $13$ types of $T_{k}$ graphs (see Table \ref{T_k graphs}), $12$ types of $R_{k}$ graphs (see Table \ref{R_k graphs}), and $4$ types of $S$ graphs (see Table \ref{S_k graphs}).
\end{thm}

\newpage

\begin{table}[h!]
	\renewcommand{\arraystretch}{1.25}
	\begin{tabular}{ |c|c|c|c| }
		\hline
		Isogeny Graph & Label & Isomorphism Types & LMFDB Label (Isogeny Class) \\
		
		\hline
		$E_1$ & $L_{1}$ & ([1]) & 37.a \\
		\hline
		
		\multirow{12}*{$E_1 \myiso E_2$} & {$L_{2}(2)$} & ([2],[2]) & $46.a$ \\
		\cline{2-4} 
		& \multirow{2}*{$L_{2}(3)$} & $([1],[1])$ & $176.c$ \\
		\cline{3-4}
		& & $([3],[1])$ & $44.a$ \\
		\cline{2-4} 
		& \multirow{2}*{$L_{2}(5)$} & $([1],[1])$ & $75.c$ \\
		\cline{3-4}
		& & $([5],[1])$ & $38.b$ \\
		\cline{2-4} 
		& \multirow{2}*{$L_{2}(7)$} & $([1],[1])$ & $208.d$ \\
		\cline{3-4}
		& & $([7],[1])$ & $26.b$ \\
		\cline{2-4}
		& $L_{2}(11)$ & $([1],[1])$ & $121.a$ \\
		\cline{2-4}
		& $L_{2}(13)$ & $([1],[1])$ & $147.b$ \\
		\cline{2-4}
		& $L_{2}(17)$ & $([1],[1])$ & $14450.b$ \\
		\cline{2-4}
		& $L_{2}(19)$ & $([1],[1])$ & $361.a$ \\
		\cline{2-4}
		& $L_{2}(37)$ & $([1],[1])$ & $1225.b$ \\
		\cline{2-4}
		& $L_{2}(43)$ & $([1],[1])$ & $1849.b$ \\
		\cline{2-4}
		& $L_{2}(67)$ & $([1],[1])$ & $4489.b$ \\
		\cline{2-4}
		& $L_{2}(163)$ & $([1],[1])$ & $26569.b$ \\
		\hline
		
		\multirow{5}*{$E_{1}\myiso E_{2}\myiso E_{3}$} & \multirow{3}*{$L_{3}(9)$} & $([1],[1],[1])$ & $175.b$ \\
		\cline{3-4}
		& & $([3],[3],[1])$ & $19.a$ \\
		\cline{3-4}
		& & $([9],[3],[1])$ & $54.b$ \\
		\cline{2-4}
		& \multirow{2}*{$L_{3}(25)$} & $([1],[1],[1])$ & $99.d$ \\
		\cline{3-4}
		& & $([5],[5],[1])$ & $11.a$ \\
		\hline
		
		\multirow{2}*{$E_{1} \myisot E_{2} \myisot E_{3} \myisot E_{4}$} & \multirow{2}*{$L_{4}$} & $([1],[1],[1],[1])$ & $432.e$ \\
		\cline{3-4}
		& & $([3],[3],[3],[1])$ & $27.a$ \\
		\hline
	\end{tabular}
	\caption{The list of all $L_{k}$ rational isogeny-torsion graphs}
	\label{L_k graphs}
\end{table}

\begin{table}[h!]
 	\renewcommand{\arraystretch}{1.6}
	\begin{tabular} { |c|c|c|c| }
		\hline
		
		Graph Type & Label & Isomorphism Types & LMFDB Label \\
		\hline
		\multirow{4}*{\includegraphics[scale=0.25]{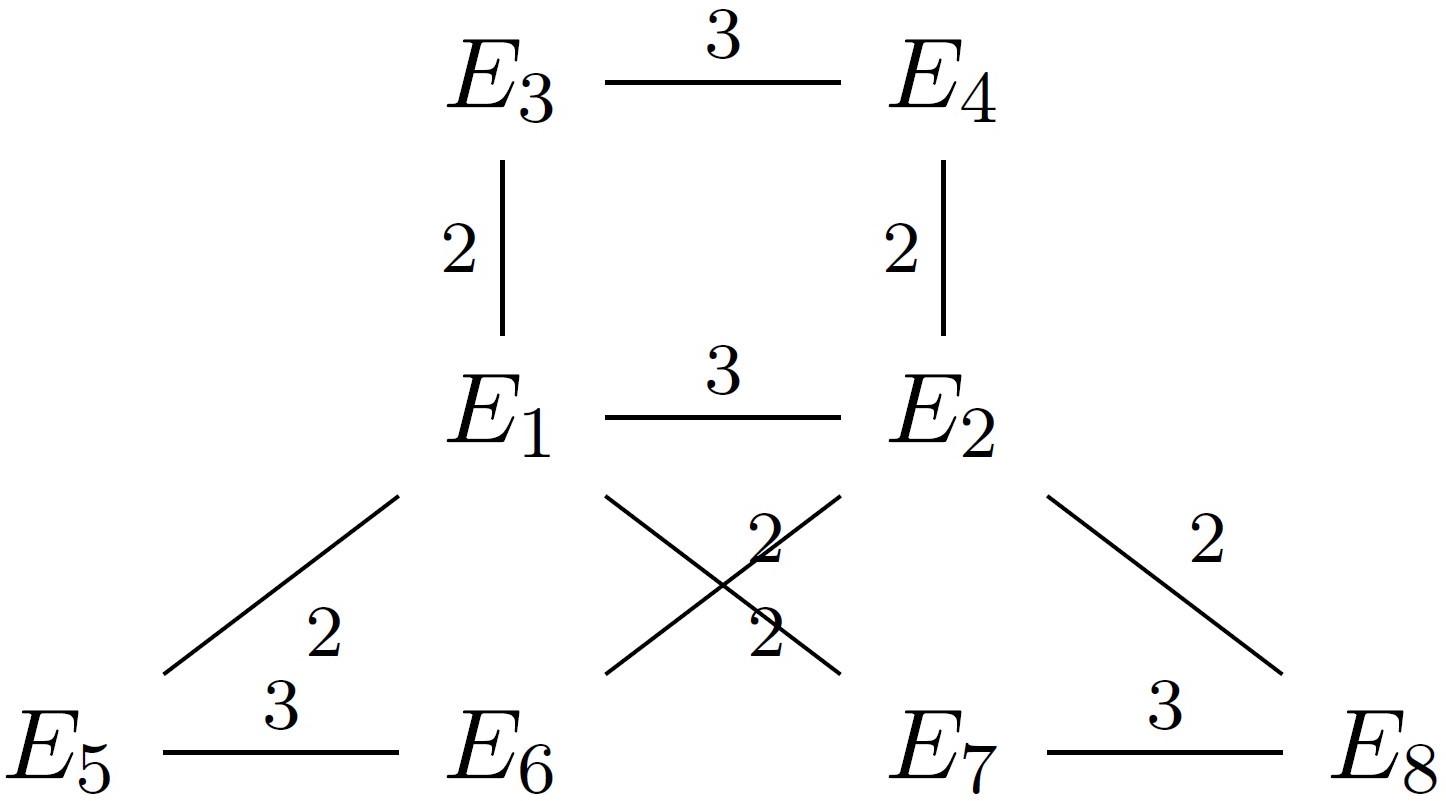}}& \multirow{4}*{$S$} & ([2,2],[2,2],[2],[2],[2],[2],[2],[2]) & 240.b \\
		\cline{3-4}
		& & ([2,2],[2,2],[4],[4],[2],[2],[2],[2]) & 150.b \\
		\cline{3-4}
		& & ([2,6],[2,2],[6],[2],[6],[2],[6],[2]) & 30.a \\
		\cline{3-4}
		& & ([2,6],[2,2],[12],[4],[6],[2],[6],[2]) & 90.c \\
		\hline
	\end{tabular}
	
	\caption{The list of all $S$ rational isogeny-torsion graphs}
	\label{S_k graphs}
\end{table}

\newpage

\begin{table}[h!]
 	\renewcommand{\arraystretch}{1.3}
 	\begin{tabular}{ |c|c|c|c| }
 		\hline
 		Graph Type & Label & Isomorphism Types & LMFDB Label  \\
 		\hline
 		
 		\multirow{4}*{\includegraphics[scale=0.21]{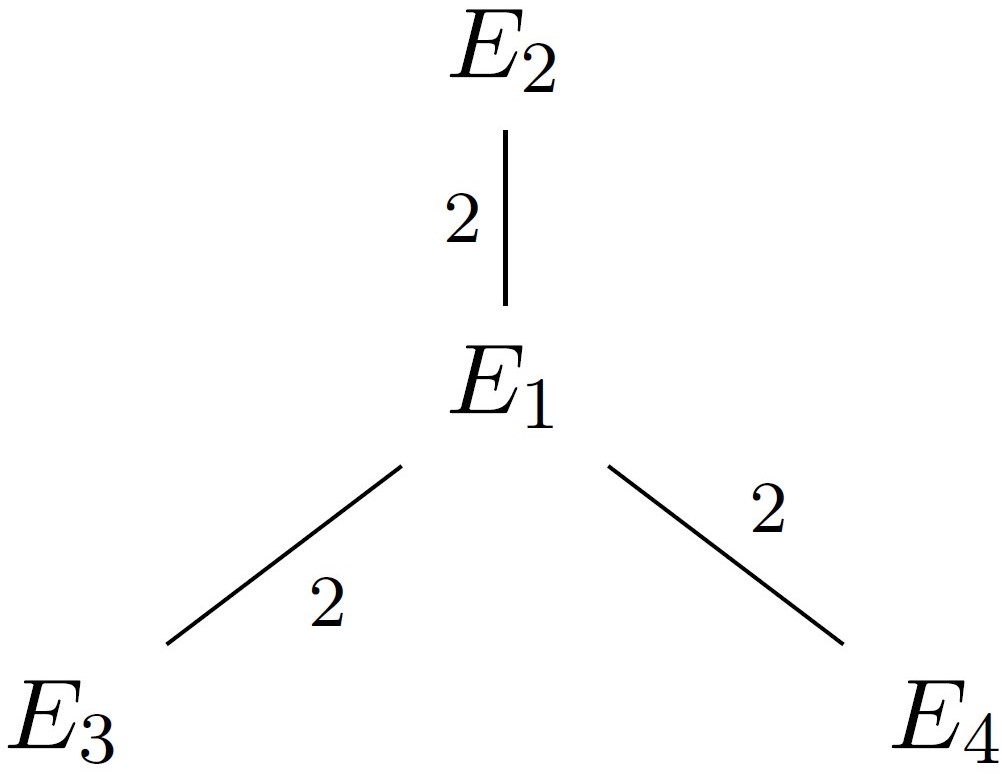}}& \multirow{4}*{$T_{4}$} & ([2,2], [2], [2], [2]) & 120.a \\
 		\cline{3-4}
 		& & ([2,2], [4], [2], [2]) & 33.a \\
 		\cline{3-4}
 		& & ([2,2], [4], [4], [2]) & 17.a \\
 		& & & \\
 		\hline
 		
 		\multirow{4}*{\includegraphics[scale=0.21]{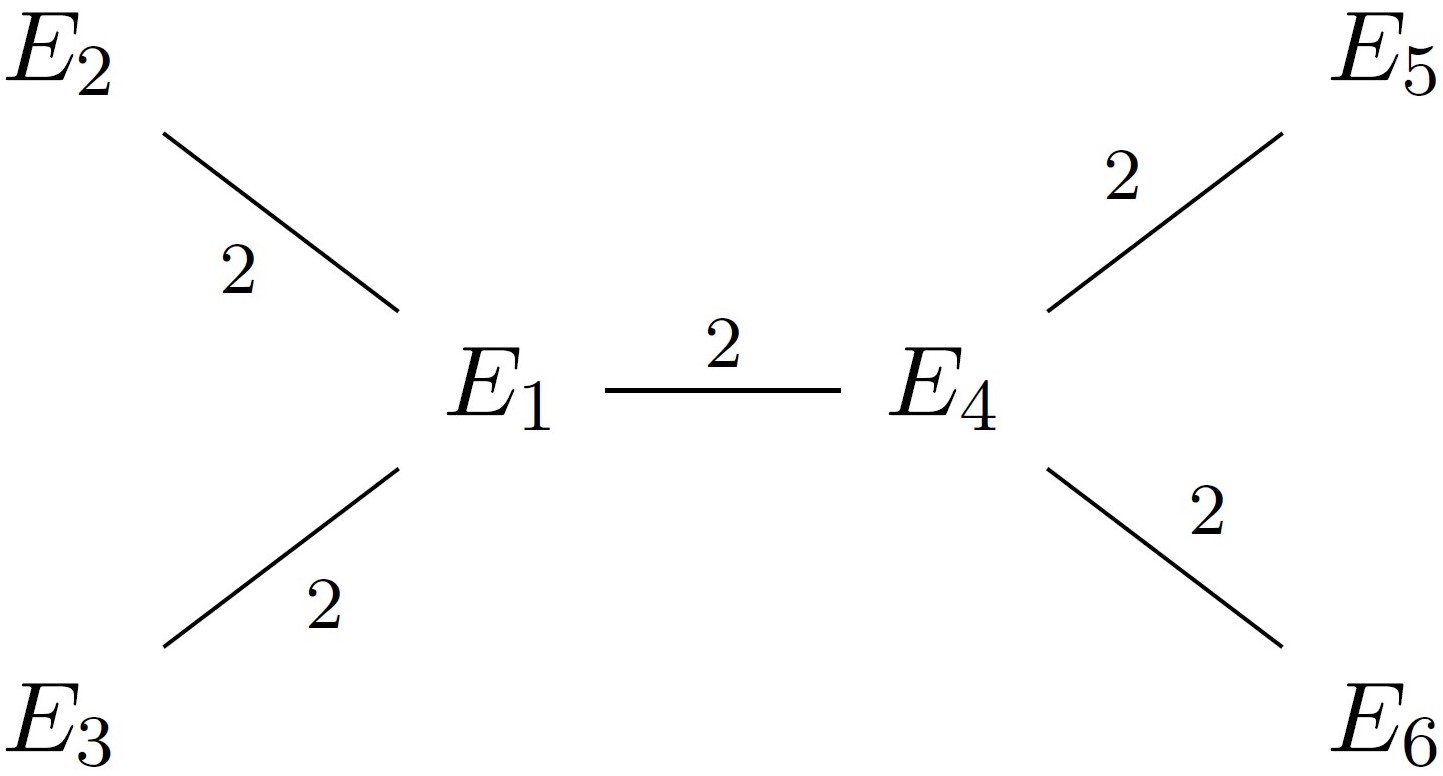}}& \multirow{4}*{$T_{6}$} & ([2,4],[4],[4],[2,2],[2],[2]) & 24.a \\
 		\cline{3-4}
 		& & ([2,4],[8],[4],[2,2],[2],[2]) & 21.a \\
 		\cline{3-4}
 		& & ([2,2],[2],[2],[2,2],[2],[2]) & 126.a \\
 		\cline{3-4}
 		& & ([2,2],[4],[2],[2,2],[2],[2]) & 63.a \\
 		\hline
 		
 		\multirow{6}*{\includegraphics[scale=0.32]{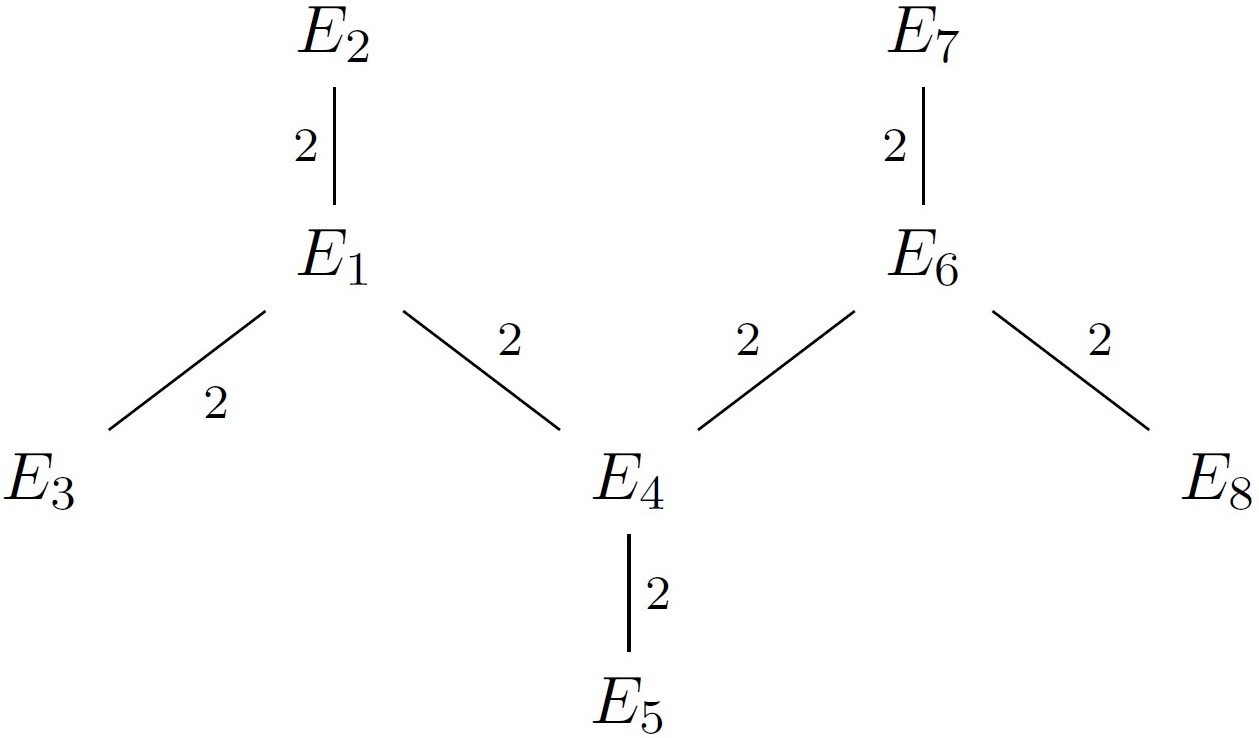}}& \multirow{6}*{$T_{8}$} & ([2,8],[8],[8],[2,4],[4],[2,2],[2],[2]) & 210.e \\
 		\cline{3-4}
 		& & ([2,4],[4],[4],[2,4],[4],[2,2],[2],[2]) & 195.a \\
 		\cline{3-4}
 		& & ([2,4],[4],[4],[2,4],[8],[2,2],[2],[2]) & 15.a \\
 		\cline{3-4}
 		& & ([2,4],[8],[4],[2,4],[4],[2,2],[2],[2]) & 1230.f \\
 		\cline{3-4}
 		& & ([2,2],[2],[2],[2,2],[2],[2,2],[2],[2]) & 45.a \\
 		\cline{3-4}
 		& & ([2,2],[4],[2],[2,2],[2],[2,2],[2],[2]) & 75.b \\
 		
 		\hline
 		
 	\end{tabular}
 	\caption{The list of all $T_{k}$ rational isogeny-torsion graphs}
 	\label{T_k graphs}
 \end{table}
 
 		\begin{table}[h!]
 	\renewcommand{\arraystretch}{1.3}
 	\begin{tabular}{ |c|c|c|c| }
 		\hline
 		Graph Type & Label & Isomorphism Types & LMFDB Label (Isogeny Class) \\
 		
 		\hline
 		\multirow{10}*{\includegraphics[scale=0.3]{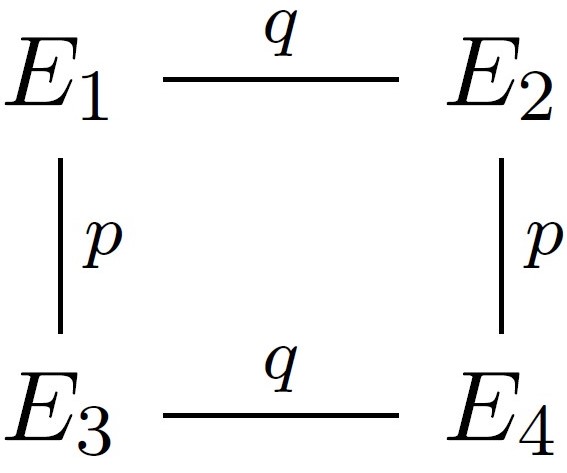}} & \multirow{2}*{$R_{4}(6)$} & $([2],[2],[2],[2])$ & $80.b$ \\
 		\cline{3-4}
 		& & $([6],[6],[2],[2])$ & $20.a$ \\
 		\cline{2-4}
 		& \multirow{2}*{$R_{4}(10)$} & $([2],[2],[2],[2])$ & $150.a$ \\
 		\cline{3-4}
 		& & $([10],[10],[2],[2])$ & $66.c$ \\
 		\cline{2-4}
 		& $R_{4}(14)$ & $([2],[2],[2],[2])$ & $49.a$ \\
 		\cline{2-4}
 		& \multirow{3}*{$R_{4}(15)$} & $([1],[1],[1],[1])$ & $400.d$ \\
 		\cline{3-4}
 		& & $([3],[3],[1],[1])$ & $50.a$ \\
 		\cline{3-4}
 		& & $([5],[5],[1],[1])$ & $50.b$ \\
 		\cline{2-4}
 		& \multirow{2}*{$R_{4}(21)$} & $([1],[1],[1],[1])$ & $1296.f$ \\
 		\cline{3-4}
 		& & $([3],[3],[1],[1])$ & $162.b$ \\
 		\hline
 		\multirow{3}*{\includegraphics[scale=0.25]{R6new}} & \multirow{3}*{$R_{6}$} & $([2],[2],[2],[2],[2],[2])$ & $98.a$ \\
 		\cline{3-4}
 		& & $([6],[6],[6],[6],[2],[2])$ & $14.a$ \\
 		& & & \\
 		\hline
 	\end{tabular}
 	\caption{The list of all $R_{k}$ rational isogeny-torsion graphs}
 	\label{R_k graphs}
 \end{table}
 
 \newpage
	
\subsection{Quadratic Twists}

\begin{lemma}\label{group quadratic twists}
Let $N$ be a positive integer and let $H$ be a subgroup of $\operatorname{GL}(2,\ZZ / N \ZZ)$ that does not contain $\operatorname{-Id}$. Let $H' = \left\langle \operatorname{-Id}, H \right\rangle$. Then $H' \cong \left\langle \operatorname{-Id} \right\rangle \times H$.
\end{lemma}

\begin{proof}
Note that $H$ is a subgroup of $H'$ of index $2$ and $\operatorname{-Id}$ is in the center of $\operatorname{GL}(2, \ZZ / N \ZZ)$. Hence, the order of $H'$ is equal to the order of $\left\langle \operatorname{-Id} \right\rangle \times \left\langle H \right\rangle$. Define
$$\psi \colon \left\langle \operatorname{-Id} \right\rangle \times H \to H'$$
by $\psi(x,h) = xh$. This is a group homomorphism as $\operatorname{-Id}$ is in the center of $\operatorname{GL}(2,\Z/N\Z)$. We are done if we prove that $\psi$ is injective. Let $(x,h) \in \left\langle \operatorname{-Id} \right\rangle \times H$ such that $\psi(x,h) = xh = \operatorname{Id}$. Then $h = x^{-1} = x$ as the order of $x$ is equal to $1$ or $2$. As $\operatorname{-Id} \notin H$, $h = x = \operatorname{Id}$ and so, $\psi$ is injective. \end{proof}

\begin{definition}
Let $G$ and $H$ be subgroups of $\operatorname{GL}(2, \ZZ_{2})$. Then we will say that $G$ and $H$ are quadratic twists if $G$ is the same as $H$, up to multiplication of some elements (possibly none) of $H$ by $\operatorname{-Id}$. In particular, if $\left\langle G, \operatorname{-Id} \right\rangle =  \left\langle H, \operatorname{-Id} \right\rangle$, then $H$ and $G$ are quadratic twists.
\end{definition}

\begin{lemma}\label{index of quadratic twists}
Let $N$ be a positive integer, let $H$ be a subgroup of $\operatorname{GL}(2,\Z/N\Z)$, and let $H' = \left\langle H, \operatorname{-Id} \right\rangle$. Let $\chi$ be a character of $H$ of degree two. Then $\chi(H) = H'$ or $\chi(H)$ is a subgroup of $H'$ of index $2$.
\end{lemma}

\begin{proof}

The character $\chi$ multiplies some elements of $H$ by $\operatorname{-Id}$. If $\operatorname{-Id} \in \chi(H)$, then we can multiply all of the elements of $H$ that $\chi$ multiplied by $\operatorname{-Id}$ by $\operatorname{-Id}$ again, and recoup all of the elements of $H$. Thus, $\chi(H)$ is a subgroup of $H'$ that contains both $\operatorname{-Id}$ and $H$ and so, $\chi(H) = H'$.

On the other hand, let us say that $\operatorname{-Id} \notin \chi(H)$. Let $\chi(H)' = \left\langle \chi(H), \operatorname{-Id} \right\rangle$. By the same argument from before, we can multiply all of the elements of $H$ that $\chi$ multiplied by $\operatorname{-Id}$ by $\operatorname{-Id}$ again, and recoup all of the elements of $H$ in $\chi(H)'$. In other words, $\chi(H)' = \left\langle \operatorname{-Id}, \chi(H) \right\rangle = \left\langle \operatorname{-Id}, H \right\rangle = H'$. By Lemma \ref{group quadratic twists}, $H' = \left\langle \chi(H), \operatorname{-Id} \right\rangle \cong \left\langle \operatorname{-Id} \right\rangle \times \chi(H)$ and $\chi(H)$ is a subgroup of $H'$ of index $2$.
\end{proof}

Let $E : y^{2} = x^{3} + Ax + B$ be an elliptic curve and let $d$ be a non-zero integer. Then the quadratic twist of $E$ by $d$ is the elliptic curve $E^{(d)} : y^{2} = x^{3} + d^{2}Ax + d^{3}B$. Equivalently, $E^{(d)}$ is isomorphic to the elliptic curve $E^{(d)} : dy^{2} = x^{3} + Ax + B$. Then $E$ is isomorphic to $E^{(d)}$ over $\QQ(\sqrt{d})$ by the map
$$\phi \colon E \to E^{(d)}$$
defined by fixing $\mathcal{O}$ and mapping any non-zero point $(a,b)$ on $E$ to $\left(a,\frac{b}{\sqrt{d}}\right)$. Moreover, the \textit{j}-invariant of $E$ is equal to the \textit{j}-invariant of $E^{(d)}$. Conversely, if $E'/\QQ$ is an elliptic curve such that the \textit{j}-invariant of $E$ is equal to the \textit{j}-invariant of $E'$, then $E$ is a (possibly trivial) quadratic twist of $E'$.

\begin{corollary}\label{subgroups of index 2 contain -Id}
Let $E/\QQ$ be an elliptic curve and let $N$ be a positive integer. Suppose that all subgroups of $\overline{\rho}_{E,N}(G_{\QQ})$ of index $2$ contain $\operatorname{-Id}$. Let $E^{\chi}$ be a quadratic twist of $E$. Then $\overline{\rho}_{E^{\chi},N}(G_{\QQ})$ is conjugate to $\overline{\rho}_{E,N}(G_{\QQ})$.
\end{corollary}

\begin{proof}
Denote $\overline{\rho}_{E,N}(G_{\QQ})$ by $H$. Then $H$ contains $\operatorname{-Id}$. By Lemma \ref{index of quadratic twists}, $\overline{\rho}_{E^{\chi},N}(G_{\Q})$ is conjugate to $H$ or is conjugate to a subgroup of $H$ of index $2$. Moreover, $\overline{\rho}_{E^{\chi},N}(G_{\Q})$ is the same as $H$, up to multiplication of some elements of $H$ by $\operatorname{-Id}$. As all subgroups of $H$ of index $2$ contain $\operatorname{-Id}$, we can just multiply the elements of $\overline{\rho}_{E^{\chi},N}(G_{\Q})$ that $\chi$ multiplied by $\operatorname{-Id}$ again by $\operatorname{-Id}$ to recoup all elements of $H$. Hence, $\overline{\rho}_{E^{\chi},N}(G_{\QQ})$ is conjugate to $H$.
\end{proof}

\begin{remark}
Let $E/\QQ$ be an elliptic curve and let $N$ be an integer greater than or equal to $3$. Suppose that $\overline{\rho}_{E,N}(G_{\QQ}) = H' = \left\langle H, \operatorname{-Id}\right\rangle$ where $H$ is a subgroup of $\operatorname{GL}(2, \ZZ / N \ZZ)$ that does not contain $\operatorname{-Id}$. Then there is a non-zero integer $d$ such that $\overline{\rho}_{E^{(d)},N}(G_{\QQ})$ is conjugate to $H$ (see Remark 1.1.3  and Section 10 in \cite{Rouse2021elladicIO}). Conversely, if $\overline{\rho}_{E,N}(G_{\QQ})$ is conjugate to $H$, then $\overline{\rho}_{E^{(d)},N}(G_{\QQ})$ is conjugate to $H'$ where $E^{(d)}$ is a quadratic twist of $E$ by a non-zero, square-free integer $d$, such that $\QQ(E[N])$ does not contain $\QQ(\sqrt{d})$.
\end{remark}

\section{Work by Rouse and Zureick-Brown}\label{section work by rouse and zureick-brown}

Rouse and Zureick-Brown classified the image of the $2$-adic Galois representation attached to non-CM elliptic curves defined over $\QQ$. This paper extends the work of Rouse and Zureick-Brown to an analogous classification for the $2$-adic Galois Image of Galois representations attached to isogeny-torsion graphs over $\QQ$ without CM.
	
\begin{thm}[Rouse, Zureick-Brown, Corollary 1.3, \cite{rouse}]\label{thm-rzb} Let $E$ be an elliptic curve over $\QQ$ without complex multiplication. Then, the index of $\rho_{E,2^\infty}(G_{\QQ})$ in $\operatorname{GL}(2,\Z_{2})$ divides $64$ or $96$; all such indices occur. Moreover, the image of $\rho_{E,2^{\infty}}(G_{\QQ})$ is the inverse image in $\operatorname{GL}(2,\ZZ_{2})$ of the image of $\overline{\rho}_{E,32}(G_{\QQ})$. For non-CM elliptic curve, there are precisely $1208$ possible images for $\rho_{E,2^{\infty}}$.
	\end{thm}

\begin{remark}

    Work in \cite{rouse} shows that if $E$ is non-CM, then the level of $\rho_{E,2^{\infty}}(G_{\QQ})$ is equal to $2^{m}$ for some integer $m$ with $0 \leq m \leq 5$.
	\end{remark}
	
	All groups that appear in the RZB database can be found in Tables \ref{2-adic Galois Images of L1 Graphs} - \ref{2-adic Galois Images of T8 Graphs}. The RZB database organizes subgroups of $\operatorname{GL}(2, \ZZ_{2})$ with nomenclature $\operatorname{X}_{n}$ where $n$ is a positive integer, for example $\operatorname{X}_{24}$ or $\operatorname{X}_{n \alpha}$ where $\alpha$ is some letter, for example, $\operatorname{X}_{24e}$. The author of this paper would like to use the nomenclature $\operatorname{H}_{n}$ to denote the group $\operatorname{X}_{n}$ from the RZB database and instead use $\operatorname{X}_{n}$ to denote the modular curve generated by $\operatorname{H}_{n}$. Similarly, we use the nomenclature $\operatorname{H}_{n\alpha}$ to denote the group $\operatorname{X}_{n\alpha}$ from the RZB database and we use the nomenclature $\operatorname{X}_{n\alpha}$ to denote the modular curve generated by $\operatorname{H}_{n\alpha}$.
	
	Groups from the RZB database of the form $\operatorname{H}_{n}$ are subgroups of $\operatorname{GL}(2, \ZZ_{2})$ that contain $\operatorname{-Id}$. Groups from the RZB database of the form $\operatorname{H}_{n \alpha}$ do not contain $\operatorname{-Id}$ and are quadratic twists of $\operatorname{H}_{n}$ (and other groups of the form $\operatorname{H}_{n\alpha}$). For example, $\operatorname{H}_{24}$ contains $\operatorname{-Id}$ and $\operatorname{H}_{24e}$ does not. Moreover, $\operatorname{H}_{24e}$ and $\operatorname{H}_{24}$ are quadratic twists as $\operatorname{H}_{24} = \left\langle \operatorname{H}_{24e}, \operatorname{-Id} \right\rangle$. Suppose that $n$ is a positive integer such that $\operatorname{H}_{n}$ is a group in the RZB database. If there are no letters $\alpha$ such that $\operatorname{H}_{n\alpha}$ is a group in the RZB database, then the only quadratic twist of $\operatorname{H}_{n}$ is $\operatorname{H}_{n}$ itself. For example, $\operatorname{H}_{1} = \operatorname{GL}(2,\ZZ_{2})$ and there are no groups of the form $\operatorname{H}_{1\alpha}$ in the RZB database for any letter $\alpha$. The full lift of $\operatorname{H}_{1}$ modulo $N$ is $\operatorname{GL}(2, \ZZ / N \ZZ)$. Moreover, $\begin{bmatrix} 0 & 1 \\ -1 & 0 \end{bmatrix}^{2} = \operatorname{-Id}$ and by Lemma \ref{subgroups of index 2 contain -Id}, the only quadratic twist of $\operatorname{H}_{1}$ is $\operatorname{H}_{1}$ itself. In general, to prove that some group $H$ in the RZB database has no non-trivial quadratic twists, it suffices to prove that the full lift of $H$ to level $32$ has no non-trivial quadratic twists. We will make use of the fact that some groups in the RZB database have no non-trivial twists without much proof for example in the classification of $2$-adic Galois Images attached to isogeny-torsion graphs of $L_{2}(11)$ type.
	
	The RZB database considers group action on the right. In other words, if $E/\QQ$ is an elliptic curve and $\overline{\rho}_{E,4}(G_{\QQ})$ is conjugate to $H = \left\langle \begin{bmatrix} A_{1} & C_{1} \\ B_{1} & D_{1} \end{bmatrix}, \ldots, \begin{bmatrix} A_{s} & C_{s} \\ B_{s} & D_{s} \end{bmatrix} \right\rangle$, then $\overline{\rho}_{E,4}(G_{\QQ})$ is actually conjugate to its transpose, namely, $\left\langle \begin{bmatrix} A_{1} & B_{1} \\ C_{1} & D_{1} \end{bmatrix}, \ldots, \begin{bmatrix} A_{s} & B_{s} \\ C_{s} & D_{s} \end{bmatrix} \right\rangle$.
	
	If $E(\QQ)_{\text{tors}} \cong \ZZ / 2 \ZZ \times \ZZ / 8 \ZZ$, then, $\overline{\rho}_{E,8}(G_{\QQ})$ is conjugate to $\left\langle
	\begin{bmatrix}	1 & 0 \\ 2 & 1 \end{bmatrix},
	\begin{bmatrix}	3 & 0 \\ 0 & 1 \end{bmatrix},
	\begin{bmatrix}	5 & 0 \\ 0 & 1 \end{bmatrix}
	\right\rangle$. This group appears in the RZB database, but it does not designate elliptic curves over $\QQ$ with a rational torsion subgroup of order $16$. The elliptic curves over $\QQ$ with a rational torsion subgroup of order $16$ correspond to non-cuspidal, $\QQ$-rational points on the modular curve $\operatorname{X}_{193n}$, generated by group $\operatorname{H}_{193n} = \left\langle \begin{bmatrix} 3 & 6 \\ 0 & 1 \end{bmatrix}, 
	\begin{bmatrix} 7 & 0 \\ 0 & 1 \end{bmatrix}, 
	\begin{bmatrix} 5 & 0 \\ 0 & 1 \end{bmatrix} \right\rangle$, which is not conjugate to the former group, but to its transpose. Sometimes, a group in the RZB database will be conjugate to its own transpose but not in general. We have to be careful of this situation.

\section{Lemmas and corollaries}\label{section lemmas and corollaries}
We continue with some more lemmas that will be used to classify the image of the $2$-adic Galois representation attached to isogeny-torsion graphs over $\QQ$ without CM.

\begin{lemma}\label{generators of Z/2^NZ}
	Let $E/\QQ$ be an elliptic curve and let $P_{2}$ and $Q_{2}$ be distinct points on $E$ of order $2$. For each integer $m \geq 2$, let $P_{2^{m}}, Q_{2^{m}} \in E$, such that $[2]P_{2^{m}} = P_{2^{m-1}}$ and $[2]Q_{2^{m}} = Q_{2^{m-1}}$. Then $E[2^{m}] = \left\langle P_{2^{m}}, Q_{2^{m}} \right\rangle$.
	\end{lemma}
	
	\begin{proof}
    The multiplication-by-$2$ map is surjective as it is a non-constant map between curves of genus $1$. Clearly, $P_{2^{m}}$ and $Q_{2^{m}}$ are points of order equal to $2^{m}$. It remains to prove that $\left\langle P_{2^{m}} \right\rangle$ and $\left\langle Q_{2^{m}} \right\rangle$ intersect trivially. If $\left\langle P_{2^{m}} \right\rangle$ and $\left\langle Q_{2^{m}} \right\rangle$ intersect non-trivially, they would share a point of order $2$. The point of order $2$ in $\left\langle P_{2^{m}} \right\rangle$ is $P_{2}$ and the point of order $2$ in $\left\langle Q_{2^{m}} \right\rangle$ is $Q_{2}$ which are distinct.
	\end{proof}
	
	\begin{remark}\label{remark generators of Z/2^NZ}
	Let $E/\QQ$ be an elliptic curve. For the rest of the paper, we will fix $P_{2}$ and $Q_{2}$ to designate two distinct points on $E$ of order $2$. For each integer $m$, greater than or equal to $2$, we will denote $P_{2^{m}}$ and $Q_{2^{m}}$ to be points on $E$ such that $[2]P_{2^{m}} = P_{2^{m-1}}$ and $[2]Q_{2^{m}} = Q_{2^{m-1}}$. Then, $E[2^{m}] = \left\langle P_{2^{m}}, Q_{2^{m}} \right\rangle$ because $\left\langle P_{2^{m}} \right\rangle$ and $\left\langle Q_{2^{m}} \right\rangle$ are two cyclic groups of order $2^{m}$ that intersect trivially.
	\end{remark}

\begin{lemma}\label{order of 1+2^n}
Let $n$ be an integer greater than or equal to $2$. Then for each positive integer $r$, we have
$$(1+2^{n})^{2^{r}} = 1+2^{n+r}+s$$
where $s$ is a positive integer such that $v_{2}(s) \geq n+r+1$.
\end{lemma}

\begin{proof}

We proceed by induction on $r$. When $r = 1$, we have $(1+2^{n})^{2} = 1+2^{n+1}+2^{2n}$. As $n \geq 2$, we have $v_{2}(2^{2n}) = 2n = n + n \geq n+2 = n+r+1$. Now we assume the lemma is true in the case of a fixed integer $r \geq 2$ and prove the lemma is true for the integer $r+1$. We have
$$(1+2^{n})^{2^{r+1}} = \left(\left(1+2^{n}\right)^{2^{r}}\right)^{2} = (1+2^{n+r}+s)^{2} = 1+2^{n+r+1}+2s+2^{n+r+1}s+2^{2(n+r)}+s^{2}$$
where $s$ is a positive integer such that $v_{2}(s) \geq n+r+1$. Let $m = 2s+2^{n+r+1}s+2^{2(n+r)}+s^{2}$. Next, $v_{2}(2s) = 1 + v_{2}(s) \geq n+r+2$. Next, $v_{2}(2^{n+r+1}s) = n+r+1 + v_{2}(s) \geq 2(n+r+1) \geq n+r+2$. Next, $v_{2}\left(2^{2(n+r)}\right) = 2(n+r) = 2n+2r \geq n+r+2$. Finally, $v_{2}(s^{2}) = 2v_{2}(s) = 2(n+r+1) \geq n+r+2$. As all summands of $m$ are positive, $m$ is a positive integer such that $v_{2}(m) \geq n+r+2$.

\end{proof}

\begin{lemma}\label{kernel of pi 1}
Let $m$ and $n$ be integers such that $m \geq n \geq 2$. Let $\pi \colon \operatorname{GL}\left(2, \ZZ / 2^{m} \ZZ\right) \to \operatorname{GL}\left(2, \ZZ / 2^{n} \ZZ\right)$ be the natural projection map.
\begin{enumerate}
    \item Then the order of $1+2^{n}$ in $\left(\ZZ / 2^{m} \ZZ\right)^{\times}$ is equal to $2^{m-n}$,
    
    \item $\operatorname{Ker}(\pi) = \left\langle \begin{bmatrix} 1+2^{n} & 0 \\ 0 & 1 \end{bmatrix}, \begin{bmatrix} 1 & 0 \\ 0 & 1+2^{n} \end{bmatrix}, \begin{bmatrix} 1 & 2^{n} \\ 0 & 1 \end{bmatrix}, \begin{bmatrix} 1 & 0 \\ 2^{n} & 1 \end{bmatrix} \right\rangle$.
\end{enumerate}
\end{lemma}

\begin{proof} The lemma is true when $m = n \geq 2$. Suppose $m > n \geq 2$. We break up the rest of the proof in steps.

\begin{enumerate}

    \item Note that $m-n$ is a positive integer. By Lemma \ref{order of 1+2^n},
    $$(1+2^{n})^{2^{m-n}} = 1+2^{m}+s$$
    for some positive integer $s$ such that $v_{2}(s) \geq m+1$. Reducing $(1+2^{n})^{2^{m-n}}$ modulo $2^{m}$, we have $\overline{(1+2^{n})^{2^{m-n}}} = \overline{1+2^{m}+s} = \overline{1}$. Thus, the order of $1+2^{n}$ is a power of $2$ less than or equal to $2^{m-n}$.
    
    We break up the proof in the cases when $m = n+1$ and $m \geq n+2$. If $m = n+1$, then $m - n = 1$. If the order of $1+2^{n}$ is less than $2^{m-n} = 2$, then $1+2^{n}$ would be the identity in $\ZZ / 2^{m} \ZZ$ but that is not true as $m > n$. Hence, the order of $1+2^{n}$ is equal to $2^{m-n}$ as in the lemma when $m = n+1$. If on the other hand, $m \geq n+2$, then $m-n-1$ is a positive integer and using Lemma \ref{order of 1+2^n} again, we see that $$(1+2^{n})^{2^{m-n-1}} = 1+2^{m-1}+s$$
    where $s$ is a positive integer such that $v_{2}(s) \geq m$. Then reducing $(1+2^{n})^{2^{m-n-1}}$ modulo $2^{m}$, we get $\overline{(1+2^{n})^{2^{m-n-1}}} = \overline{1+2^{m-1}+s} = \overline{1 + 2^{m-1}}$ which is not congruent to $1$ modulo $2^{m}$. Thus, again, the order of $2^{n}$ modulo $2^{m}$ is equal to $2^{m-n}$.
    
    \item Let $\begin{bmatrix}
    A & B \\ C & D
    \end{bmatrix}$ be an element of $\operatorname{Ker}(\pi)$. Then $\begin{bmatrix}
    A-1 & B \\ C & D-1
    \end{bmatrix} \equiv \begin{bmatrix} 0 & 0 \\ 0 & 0 \end{bmatrix} \mod 2^{n}$. Clearly, $\operatorname{Ker}(\pi)$ is a group of order $\left(\frac{2^{m}}{2^{n}}\right)^{4} = 2^{4(m-n)}$. Let $G_{1} = \begin{bmatrix} 1+2^{n} & 0 \\ 0 & 1 \end{bmatrix}$, let $G_{2} = \begin{bmatrix} 1 & 2^{n} \\ 0 & 1 \end{bmatrix}$, let $G_{3} = \begin{bmatrix} 1 & 0 \\ 2^{n} & 1 \end{bmatrix}$, and let $G_{4} = \begin{bmatrix} 1 & 0 \\ 0 & 1 + 2^{n} \end{bmatrix}$. By Lemma \ref{order of 1+2^n}, the order of $G_{1}$ and the order of $G_{4}$ are both equal to $2^{m-n}$. By another inductive argument, we can see that the order of $G_{2}$ and $G_{3}$ are both equal to $2^{m-n}$. For $i \in \{1, 2, 3, 4\}$ and $j \in \{1, 2, 3, 4\}$ with $j \neq i$, we have $\left\langle G_{i} \right\rangle \bigcap \left\langle G_{j} \right\rangle = \{\operatorname{Id}\}$. Thus, the group $\left\langle G_{1}, G_{2}, G_{3}, G_{4} \right\rangle$ is contained in $\operatorname{Ker}(\pi)$ and has order equal to $2^{4(m-n)}$.
\end{enumerate}
\end{proof}

\begin{lemma}\label{kernel of pi 2}
Let $m$ be an integer greater than or equal to $2$. Let $\pi \colon \operatorname{GL}(2, \ZZ / 2^{m} \ZZ) \to \operatorname{GL}(2, \ZZ / 2 \ZZ)$ be the natural projection map. Then

\begin{enumerate}
\item The order of $3$ and the order of $5$ in $\left(\ZZ / 2^{s} \ZZ\right)^{\times}$ are both equal to $2^{s-2}$ for all $s \geq 3$

\item $5$ is not contained in the subgroup of $\left( \ZZ / 2^{d} \ZZ \right)^{\times}$ generated by $3$ for all $d \geq 3$ and $\left(\ZZ / 2^{d} \ZZ \right)^{\times} = \left\langle 3, 5 \right\rangle$.

\item $\operatorname{Ker}(\pi) = \left\langle \begin{bmatrix} 3 & 0 \\ 0 & 1 \end{bmatrix}, \begin{bmatrix} 5 & 0 \\ 0 & 1 \end{bmatrix}, \begin{bmatrix} 1 & 0 \\ 0 & 3 \end{bmatrix}, \begin{bmatrix} 1 & 0 \\ 0 & 5 \end{bmatrix}, \begin{bmatrix} 1 & 2 \\ 0 & 1 \end{bmatrix}, \begin{bmatrix} 1 & 0 \\ 2 & 1 \end{bmatrix} \right\rangle$.
\end{enumerate}
\end{lemma}

\begin{proof} Again we break up the proof into steps.

\begin{enumerate}
    \item First we note that for any integer $k \geq 3$, the order of $\left(\ZZ / 2^{k} \ZZ\right)^{\times}$ is a power of $2$. Hence, the order of $3$ and $5$ in $\left(\ZZ / 2^{k} \ZZ\right)^{\times}$ is a power of $2$. Consider $3 = 1+2$. We claim that the order of $3$ in $\left(\ZZ / 2^{s} \ZZ\right)^{\times}$ is equal to $2^{s-2}$ for all $s \geq 3$. Note that $3^{2} = 1+2^{3}$ so the claim is true when $s = 3$ so let $s \geq 4$. Using the first part of Lemma \ref{kernel of pi 1}, we see that the order of $3^{2} = 9 = 1 + 2^{3}$ in $\left( \ZZ / 2^{s} \ZZ\right)^{\times}$ is precisely equal to $2^{s-3}$. Hence, the order of $3$ in $\left(\ZZ / 2^{s} \ZZ\right)^{\times}$ is equal to $2^{s-2}$.
    
    Next, note that $5 = 1 + 2^{2}$. When $s \geq 3 \geq 2$, we may apply the first part of Lemma \ref{kernel of pi 1} to conclude that the order of $5$ in $\left(\ZZ / 2^{s} \ZZ \right)^{\times}$ is equal to $2^{s-2}$.

\item Assume by way of contradiction that there is an integer $d \geq 3$ such that $5 \in \left\langle 3 \right\rangle$ modulo $2^{d}$. As the orders of $3$ and $5$ in $\left(\ZZ / 2^{d} \right)^{\times}$ are the same and equal to a power of $2$, this means that $5$ would generate the same subgroup of $\left(\ZZ / 2^{d} \ZZ\right)^{\times}$ generated by $3$. Hence, $3$ is an odd power of $5$ in $\left(\ZZ / 2^{d} \ZZ\right)^{\times}$.

Then there is an odd integer $k$ such that $3 = 5^{k} = (1+2^{2})^{k}$. Let $0 = k_{1} < \ldots < k_{s}$ be a set of integers such that
$$k = 2^{k_{1}} + \ldots + 2^{k_{s}} = 1 + \ldots + 2^{k_{s}}$$
is the $2$-adic decomposition of $k$. In other words,
$$3 = (1+2^{2})^{k} = (1+2^{2})^{1+ \ldots + 2^{k_{s}}} = (1+2^{2}) \cdots (1+2^{2})^{2^{k_{s}}} = (1+2^{2}) \cdots (1 + \ldots + 2^{2^{k_{s}+1}}).$$
After more expansion, we see that $3 = 1 + (s-1) \cdot 2^{2} + r$ where $r $ is a positive integer such that $v_{2}(r) \geq 3$. If we subtract both sides of the equation by $3$, we get that $-2 + (s-1) \cdot 2^{2} + r$ is equal to $0$ modulo $2^{d}$. Note that $v_{2}(-2+(s-1) \cdot 2^{2} + r) = 1$ but $d \geq 3$, a contradiction.

The order of $\left(\ZZ / 2^{d}\right)^{\times}$ is equal to $2^{d-1}$ for all integers $d \geq 3$. The order of $3$ in $\left(\ZZ / 2^{d} \right)^{\times}$ is equal to $2^{d-2}$ and $5$ is not contained in $\left\langle 3 \right\rangle$. Hence, $\left(\ZZ / 2^{d} \ZZ\right)^{\times} = \left\langle 3, 5 \right\rangle$.

\item Clearly, $\operatorname{Ker}(\pi)$ is the set of all matrices of the form $\begin{bmatrix}
A & B \\ C & D
\end{bmatrix}$ where $A$ and $D$ are odd and $B$ and $C$ are even. Then $\operatorname{Ker}(\pi)$ is a group of order $2^{4(m-1)}$. Clearly the lemma is true when $m = 1$ or $2$. Suppose that $m \geq 3$. Let $G_{1} = \left\langle \begin{bmatrix}
3 & 0 \\ 0 & 1
\end{bmatrix}, \begin{bmatrix}
5 & 0 \\ 0 & 1
\end{bmatrix} \right\rangle$, $G_{2} = \left\langle \begin{bmatrix} 1 & 2 \\ 0 & 1 \end{bmatrix} \right\rangle$, $G_{3} = \left\langle\begin{bmatrix} 1 & 0 \\ 2 & 1 \end{bmatrix} \right\rangle$, and $G_{4} = \left\langle \begin{bmatrix} 1 & 0 \\ 0 & 3 \end{bmatrix}, \begin{bmatrix} 1 & 0 \\ 0 & 5 \end{bmatrix} \right\rangle$. For $i \in \{1, 2, 3, 4\}$, the order of $G_{i}$ is equal to $2^{m-1}$ and for $j \in \{1, 2, 3, 4\}$ and $i \neq j$, we have $G_{i} \bigcap G_{j} = \{\operatorname{Id}\}$. Thus, $\left\langle G_{1}, G_{2}, G_{3}, G_{4} \right\rangle$ is a subgroup of $\operatorname{Ker}(\pi)$ of order $2^{4(m-1)}$.
\end{enumerate}
\end{proof}

\begin{lemma}\label{lifting with kernel}

Let $E/\QQ$ be an elliptic curve without CM. Suppose that $\rho_{E,2^{\infty}}(G_{\QQ})$ is a group of level $2^{m}$ conjugate to
$$\left\langle
\begin{bmatrix} A_{1} & C_{1} \\ B_{1} & D_{1} \end{bmatrix},
\ldots,
\begin{bmatrix} A_{s} & C_{s} \\ B_{s} & D_{s} \end{bmatrix}
\right\rangle \subseteq \operatorname{GL}(2, \ZZ / 2^{m} \ZZ).$$
Let $r$ be a non-negative integer and let $\pi \colon \operatorname{GL}(2, \ZZ / 2^{m+r} \ZZ) \to \operatorname{GL}(2, \ZZ / 2^{m} \ZZ)$ be the natural projection map. Then $\overline{\rho}_{E,2^{m+r}}(G_{\QQ})$ is conjugate to
$$\left\langle \begin{bmatrix} A_{1} & C_{1} \\ B_{1} & D_{1} \end{bmatrix},
\ldots,
\begin{bmatrix} A_{s} & C_{s} \\ B_{s} & D_{s} \end{bmatrix}, \operatorname{Ker}(\pi) \right\rangle.$$
In particular,
\begin{enumerate}
\item if $m = 0$, $\overline{\rho}_{E,2^{m+r}}(G_{\QQ})$ is conjugate to $\operatorname{GL}(2, \ZZ / 2^{m+r} \ZZ)$
\item if $m = 1$, then $\overline{\rho}_{E,2^{m+r}}(G_{\QQ})$ is conjugate to
$$\left\langle
\begin{bmatrix} A_{1} & C_{1} \\ B_{1} & D_{1} \end{bmatrix},
\ldots,
\begin{bmatrix} A_{s} & C_{s} \\ B_{s} & D_{s}\end{bmatrix},
\begin{bmatrix} 3 & 0 \\ 0 & 1 \end{bmatrix},
\begin{bmatrix} 5 & 0 \\ 0 & 1 \end{bmatrix},
\begin{bmatrix} 1 & 0 \\ 0 & 3 \end{bmatrix},
\begin{bmatrix} 1 & 0 \\ 0 & 5 \end{bmatrix},
\begin{bmatrix} 1 & 2 \\ 0 & 1 \end{bmatrix},
\begin{bmatrix} 1 & 0 \\ 2 & 1 \end{bmatrix}
\right\rangle \subseteq \operatorname{GL}\left(2, \ZZ / 2^{m+r} \ZZ\right)$$
\item if $5 \geq m \geq 2$, then $\overline{\rho}_{E,2^{m+r}}(G_{\QQ})$ is conjugate to
$$\left\langle
\begin{bmatrix} A_{1} & C_{1} \\ B_{1} & D_{1} \end{bmatrix},
\ldots,
\begin{bmatrix} A_{s} & C_{s} \\ B_{s} & D_{s}\end{bmatrix},
\begin{bmatrix} 1+2^{m} & 0 \\ 0 & 1 \end{bmatrix},
\begin{bmatrix} 1 & 0 \\ 0 & 1+2^{m} \end{bmatrix},
\begin{bmatrix} 1 & 2^{m} \\ 0 & 1 \end{bmatrix},
\begin{bmatrix} 1 & 0 \\ 2^{m} & 1 \end{bmatrix}
\right\rangle \subseteq \operatorname{GL}(2, \ZZ / 2^{m+r} \ZZ)$$
\end{enumerate}
\end{lemma}

\begin{proof}
For $i \in \{1, \ldots, s\}$, let $g_{i} = \begin{bmatrix} A_{i} & C_{i} \\ B_{i} & D_{i} \end{bmatrix} \in \operatorname{GL}(2, \ZZ / 2^{m} \ZZ)$ and moreover, let $\widehat{g}_{i} = \begin{bmatrix} A_{i} & C_{i} \\ B_{i} & D_{i} \end{bmatrix} \in \operatorname{GL}(2, \ZZ / 2^{m+r} \ZZ)$. Let $x$ be a matrix in $\overline{\rho}_{E,2^{m+r}}(G_{\Q})$. Then $\pi(x) = y \in \left\langle g_{1}, \ldots, g_{s}\right\rangle$. Let $\widehat{y}$ be the product of those same matrices in $\operatorname{GL}(2, \ZZ / 2^{m} \ZZ)$ but viewed in $\operatorname{GL}(2, \ZZ / 2^{m+r} \ZZ)$. For example, if $y = g_{1} \cdot g_{3} \cdot g_{2} \cdot g_{1}^{-1}$, then let $\widehat{y} = \widehat{g_{1}} \cdot \widehat{g_{3}} \cdot \widehat{g_{2}} \cdot \widehat{g_{1}}^{-1}$. Then $\pi(x) = y = \pi(\widehat{y})$ and so, $x \cdot \widehat{y}^{-1} \in \operatorname{Ker}(\pi)$. Hence, $x \in \left\langle \widehat{g_{1}}, \ldots, \widehat{g_{r}}, \operatorname{Ker}(\pi) \right\rangle$. As $\rho_{E,2^{\infty}}(G_{\QQ})$ is a group of level $2^{m}$, $\overline{\rho}_{E,2^{m+r}}(G_{\QQ})$ is precisely equal to $\left\langle \widehat{g_{1}}, \ldots, \widehat{g_{r}}, \operatorname{Ker}(\pi) \right\rangle$. If the level of $\rho_{E,2^{\infty}}(G_{\QQ})$ is equal to $2^{0} = 1$, then $\rho_{E,2^{\infty}}(G_{\QQ}) = \operatorname{GL}(2, \ZZ_{2})$. The rest of the lemma follows from Lemma \ref{kernel of pi 2} and Lemma \ref{kernel of pi 1}.
\end{proof}

\begin{lemma}\label{Borel 1}

Let $E/\QQ$ be an elliptic curve with a cyclic, $\QQ$-rational subgroup of order $2^{r}$ for some non-negative integer $r$. Let $m$ be a non-negative integer. Then $\overline{\rho}_{E,2^{m+r}}(G_{\QQ})$ is conjugate to a subgroup of $\operatorname{GL}(2, \ZZ / 2^{m+r} \ZZ)$ of the form
$\left\langle
\begin{bmatrix} A_{1} & C_{1} \\ B_{1} & D_{1} \end{bmatrix}, \ldots, \begin{bmatrix} A_{s} & C_{s} \\ B_{s} & D_{s} \end{bmatrix}
\right\rangle$
such that $2^{r}$ divides $C_{1}$, \ldots, $C_{s}$.

\end{lemma}

\begin{proof}
Let $Q_{2^{r}}$ be a point on $E$ of order $2^{r}$ that generates a $\QQ$-rational group. Let $Q_{2^{m+r}}$ be a point on $E$ such that $[2^{m}]Q_{2^{m+r}} = Q_{2^{r}}$. Let $P_{2^{m+r}}$ be a point on $E$ such that $E[2^{m+r}] = \left\langle P_{2^{m+r}}, Q_{2^{m+r}} \right\rangle$. Let $\sigma \in G_{\QQ}$. Then there are integers $C$ and $D$ such that $\sigma(Q_{2^{m+r}}) = [C]P_{2^{m+r}} + [D]Q_{2^{m+r}}$. Moreover,
$$\sigma(Q_{2^{r}}) = \sigma([2^{m}]Q_{2^{m+r}}) = [2^{m}]\sigma(Q_{2^{m+r}}) = [2^{m}]([C]P_{2^{m+r}} + [D]Q_{2^{m+r}}) = [2^{m}C]P_{2^{m+r}} + [D]Q_{2^{r}}.$$
Note that $Q_{2^{r}}$ generates a $\QQ$-rational group. Thus, $[2^{m}C]P_{2^{m+r}} \in \left\langle Q_{2^{r}} \right\rangle$. As $\left\langle P_{2^{m+r}} \right\rangle \cap \left\langle Q_{2^{m+r}} \right\rangle = \left\{\mathcal{O}\right\}$, $[2^{m}C]P_{2^{m+r}} = \mathcal{O}$ and hence, $2^{m+r}$ divides $2^{m}C$. We conclude that $2^{r}$ divides $C$.
\end{proof}

\begin{lemma}\label{Borel 2}

Let $E/\QQ$ be an elliptic curve and let $r$ be a non-negative integer such that $E$ contains a cyclic, $\QQ$-rational subgroup of order $2^{r}$. Let $m$ be a non-negative integer less than or equal to $r$. Then $\overline{\rho}_{E,2^{m}}(G_{\QQ})$ is conjugate to a subgroup of $\operatorname{GL}(2, \ZZ / 2^{m} \ZZ)$ of the form
$\left\langle
\begin{bmatrix} A_{1} & 0 \\ B_{1} & D_{1} \end{bmatrix}, \ldots, \begin{bmatrix} A_{s} & 0 \\ B_{s} & D_{s} \end{bmatrix}
\right\rangle$.

\end{lemma}

\begin{proof}

Let $Q_{2^{r}}$ be a point on $E$ that generates a $\QQ$-rational group of order $2^{r}$. Let $\sigma \in G_{\QQ}$. Then there is an integer $D$ such that $\sigma(Q_{2^{r}}) = [D]Q_{2^{r}}$. Let $P_{2^{r}}$ be a point on $E$ such that $E[2^{r}] = \left\langle P_{2^{r}}, Q_{2^{r}} \right\rangle$. Then there are integers $A$ and $B$ such that $\sigma(P_{2^{r}}) = [A]P_{2^{r}} + [B]Q_{2^{r}}$. Hence, as a matrix, $\sigma$ has the representation $\begin{bmatrix}
A & 0 \\ B & D
\end{bmatrix}$. The rest of the proof follows from the fact that $\overline{\rho}_{E,2^{m}}(G_{\QQ})$ is simply the reduction of $\overline{\rho}_{E,2^{r}}(G_{\QQ})$ modulo $2^{m}$.

\end{proof}

\begin{lemma}\label{2-adic Galois Images}

Let $E/\QQ$ and $E'/\QQ$ be elliptic curves that are $\QQ$-isogenous by an isogeny $\phi$ such that $\operatorname{Ker}(\phi)$ is finite, cyclic, and $\QQ$-rational. Let $r$ be greatest non-negative integer such that $2^{r}$ divides the order of $\operatorname{Ker}(\phi)$ and let $m$ be the non-negative integer such that $\rho_{E,2^{\infty}}(G_{\QQ})$ is a group of level $2^{m}$. Then $0 \leq r \leq m \leq 5 \leq 5+r$ and there is a basis of $E\left[2^{m}\right]$ such that $\overline{\rho}_{E,2^{m}}(G_{\QQ})$ is conjugate to a group of the form
$$\left\langle \begin{bmatrix}
A_{1} & C_{1} \\ B_{1} & D_{1}
\end{bmatrix}, \ldots, \begin{bmatrix} A_{s} & C_{s} \\ B_{s} & D_{s} \end{bmatrix} \right\rangle$$
such that $2^{r}$ divides $C_{1}, \ldots, C_{s}$. Moreover,
    \begin{itemize}
        \item if $m = 0$, then both $\overline{\rho}_{E,2^{r+5}}(G_{\QQ})$ and $\overline{\rho}_{E',32}(G_{\QQ})$ are conjugate to $\operatorname{GL}(2, \ZZ / 2^{r+5} \ZZ)$.
        \item If $m = 1$, then $\overline{\rho}_{E,2^{r+5}}(G_{\QQ})$ is conjugate to
        $$\left\langle \begin{bmatrix}
A_{1} & C_{1} \\ B_{1} & D_{1}
\end{bmatrix}, \ldots, \begin{bmatrix} A_{s} & C_{s} \\ B_{s} & D_{s} \end{bmatrix}, \begin{bmatrix}
1 & 2 \\ 0 & 1
\end{bmatrix}, \begin{bmatrix}
1 & 0 \\ 2 & 1
\end{bmatrix}, \begin{bmatrix}
3 & 0 \\ 0 & 1
\end{bmatrix}, \begin{bmatrix}
5 & 0 \\ 0 & 1
\end{bmatrix}, \begin{bmatrix}
1 & 0 \\ 0 & 3
\end{bmatrix}, \begin{bmatrix}
1 & 0 \\ 0 & 5
\end{bmatrix} \right\rangle$$
and $\overline{\rho}_{E',32}(G_{\QQ})$ is conjugate to 
$$\left\langle \begin{bmatrix}
A_{1} & \frac{C_{1}}{2^{r}} \\ 2^{r} \cdot B_{1} & D_{1}
\end{bmatrix}, \ldots, \begin{bmatrix} A_{s} & \frac{C_{s}}{2^{r}} \\ 2^{r} \cdot B_{s} & D_{s} \end{bmatrix}, \begin{bmatrix}
1 & 2^{1-r} \\ 0 & 1
\end{bmatrix}, \begin{bmatrix}
1 & 0 \\ 2^{r+1} & 1
\end{bmatrix}, \begin{bmatrix}
3 & 0 \\ 0 & 1
\end{bmatrix}, \begin{bmatrix}
5 & 0 \\ 0 & 1
\end{bmatrix}, \begin{bmatrix}
1 & 0 \\ 0 & 3
\end{bmatrix}, \begin{bmatrix}
1 & 0 \\ 0 & 5
\end{bmatrix} \right\rangle.$$
        \item If $2 \leq m \leq 5$, then $\overline{\rho}_{E,2^{r+5}}(G_{\QQ})$ is conjugate to
        $$\left\langle \begin{bmatrix}
A_{1} & C_{1} \\ B_{1} & D_{1}
\end{bmatrix}, \ldots, \begin{bmatrix} A_{s} & C_{s} \\ B_{s} & D_{s} \end{bmatrix}, \begin{bmatrix}
1 & 2^{m} \\ 0 & 1
\end{bmatrix}, \begin{bmatrix}
1 & 0 \\ 2^{m} & 1
\end{bmatrix}, \begin{bmatrix}
1+2^{m} & 0 \\ 0 & 1
\end{bmatrix}, \begin{bmatrix}
1 & 0 \\ 0 & 1+2^{m} \end{bmatrix} \right\rangle$$
and $\overline{\rho}_{E',32}(G_{\QQ})$ is conjugate to
$$\left\langle \begin{bmatrix}
A_{1} & \frac{C_{1}}{2^{r}} \\ 2^{r} \cdot B_{1} & D_{1}
\end{bmatrix}, \ldots, \begin{bmatrix} A_{s} & \frac{C_{s}}{2^{r}} \\ 2^{r} \cdot B_{s} & D_{s} \end{bmatrix}, \begin{bmatrix}
1 & 2^{m-r} \\ 0 & 1
\end{bmatrix}, \begin{bmatrix}
1 & 0 \\ 2^{m+r} & 1
\end{bmatrix}, \begin{bmatrix}
1+2^{m} & 0 \\ 0 & 1
\end{bmatrix}, \begin{bmatrix}
1 & 0 \\ 0 & 1+2^{m} \end{bmatrix} \right\rangle$$
    \end{itemize}

\end{lemma}

\begin{proof}
The fact that the level of $\rho_{E,2^{\infty}}(G_{\QQ})$ is a level of at most $32$ is directly from Theorem \ref{thm-rzb}. Now we prove that $r \leq m$. Suppose that $m = 0$. Then $\rho_{E,2^{\infty}}(G_{\QQ})$ is a group of level $1$. Then $\overline{\rho}_{E,1}(G_{\QQ})$ is the trivial subgroup of $\operatorname{GL}(2, \{1\})$ and for each positive integer $d$, $\overline{\rho}_{E,2^{d}}(G_{\QQ})$ is the full lift of the trivial group. This means that $\overline{\rho}_{E,2^{d}}(G_{\QQ})$ is conjugate to $\operatorname{GL}(2, \ZZ / 2^{d} \ZZ)$. If at the same time, $r$ is positive, then $E$ has a point of order $2$ defined over $\QQ$ and by Lemma \ref{Borel 2}, $\overline{\rho}_{E,2}(G_{\QQ})$ is conjugate to $\left\langle \begin{bmatrix} 1 & 0 \\ 1 & 1 \end{bmatrix} \right\rangle$. This contradicts the fact that $\overline{\rho}_{E,2}(G_{\QQ})$ is conjugate to $\operatorname{GL}(2, \ZZ/ 
2 \ZZ)$.

Suppose now that $m$ is a positive integer and assume by way of contradiction that $1 \leq m < m + 1 \leq r$. By Lemma \ref{Borel 2}, we can find a basis of $E[2^{m+1}]$ such that $\overline{\rho}_{E,2^{m+1}}(G_{\QQ})$ is conjugate to a subgroup of $\operatorname{GL}(2, \ZZ / 2^{m+1} \ZZ)$ of the form
$\left\langle \begin{bmatrix} A_{1} & 0 \\ B_{1} & D_{1} \end{bmatrix}, \ldots, \begin{bmatrix}
A_{s} & 0 \\ B_{s} & D_{s}
\end{bmatrix} \right\rangle \subseteq \operatorname{GL}(2, \ZZ / 2^{m+1} \ZZ)$. As $\overline{\rho}_{E,2^{m}}(G_{\QQ})$ is simply the reduction of $\overline{\rho}_{E,2^{m+1}}(G_{\QQ})$ modulo $2^{m}$, there is a basis of $E[2^{m}]$ such that $\overline{\rho}_{E,2^{m}}(G_{\QQ})$ is equal to $\left\langle \begin{bmatrix} A_{1} & 0 \\ B_{1} & D_{1} \end{bmatrix}, \ldots, \begin{bmatrix}
A_{s} & 0 \\ B_{s} & D_{s}
\end{bmatrix} \right\rangle \subseteq \operatorname{GL}(2, \ZZ / 2^{m} \ZZ)$. As the level of $\rho_{E,2^{\infty}}(G_{\QQ})$ is equal to $2^{m}$, we have that $\overline{\rho}_{E,2^{m+1}}(G_{\QQ})$ is equal to the full lift of $\overline{\rho}_{E,2^{m}}(G_{\QQ})$ to level $2^{m}$. By Lemma \ref{lifting with kernel}, $\overline{\rho}_{E,2^{m+1}}(G_{\QQ})$ is conjugate to one of the following two groups

\begin{center}
    $\left\langle \begin{bmatrix} A_{1} & 0 \\ B_{1} & D_{1} \end{bmatrix}, \ldots, \begin{bmatrix}
A_{s} & 0 \\ B_{s} & D_{s}
\end{bmatrix}, \begin{bmatrix}
1+2^{m} & 0 \\ 0 & 1
\end{bmatrix}, \begin{bmatrix}
1 & 0 \\ 0 & 1+2^{m}
\end{bmatrix}, \begin{bmatrix} 1 & 2^{m} \\ 0 & 1 \end{bmatrix}, \begin{bmatrix} 1 & 0 \\ 2^{m} & 1 \end{bmatrix}, \begin{bmatrix}
5 & 0 \\ 0 & 1
\end{bmatrix}, \begin{bmatrix}
1 & 0 \\ 0 & 5
\end{bmatrix} \right\rangle$ \end{center}

or

\begin{center}
$\left\langle \begin{bmatrix} A_{1} & 0 \\ B_{1} & D_{1} \end{bmatrix}, \ldots, \begin{bmatrix}
A_{s} & 0 \\ B_{s} & D_{s}
\end{bmatrix}, \begin{bmatrix}
1+2^{m} & 0 \\ 0 & 1
\end{bmatrix}, \begin{bmatrix}
1 & 0 \\ 0 & 1+2^{m}
\end{bmatrix}, \begin{bmatrix} 1 & 2^{m} \\ 0 & 1 \end{bmatrix}, \begin{bmatrix} 1 & 0 \\ 2^{m} & 1 \end{bmatrix} \right\rangle$
\end{center}
depending on whether $m = 1$ or $2 \leq m \leq 5$. But this is a contradiction due to the presence of the matrix $\begin{bmatrix}
1 & 2^{m} \\ 0 & 1
\end{bmatrix}$. Hence, $0 \leq r \leq m \leq 5$.

Now let $Q_{2^{r}}$ is a generator of the subgroup of $\operatorname{Ker}(\phi)$ of order $2^{r}$ and let $Q_{2^{5+r}}$ be an element of order $2^{5+r}$ such that $\left\langle Q_{2^{5+r}} \right\rangle$ contains $Q_{2^{r}}$. We prove that if $\left\{P_{2^{5+r}}, Q_{2^{5+r}}\right\}$ is a basis of $E[2^{5+r}]$, then $\left\{\phi([2^{r}]P_{2^{5+r}}), \phi(Q_{2^{5+r}})\right\}$ is a basis of $E'[32]$. First note that $[32]\phi([2^{r}]P_{2^{5+r}}) = \phi([2^{5+r}]P_{2^{5+r}}) = \mathcal{O}$ and $[16]\phi([2^{r}]P_{2^{5+r}}) = \phi([2^{4+r}]P_{2^{5+r}})$. Note that $[2^{4+r}]P_{2^{5+r}}$ is a point on $E$ of order $2$. If $[2^{4+r}]P_{2^{5+r}} \in \operatorname{Ker}(\phi)$, then $[2^{4+r}]P_{2^{5+r}}$ is the point in $\left\langle Q_{2^{r}} \right\rangle$ of order $2$, which is a contradiction as $\left\langle P_{2^{5+r}} \right\rangle \bigcap \left\langle Q_{2^{5+r}} \right\rangle = \left\{\mathcal{O} \right\}$. Next, $[32]\phi(Q_{2^{5+r}}) = \phi([32]Q_{2^{5+r}})$ and $[16]\phi(Q_{2^{5+r}}) = \phi([16]Q_{2^{5+r}})$. Note that $[16]Q_{2^{5+r}}$ is a point on $E$ of order $2^{r+1}$, hence, cannot be an element of $\operatorname{Ker}(\phi)$. Thus, the orders of both $\phi([2^{r}]P_{2^{5+r}})$ and $\phi(Q_{2^{5+r}})$ are equal to $32$. Assume that $\left\langle \phi([2^{r}]P_{2^{5+r}}) \right\rangle$ and $\left\langle \phi(Q_{2^{5+r}}) \right\rangle$ intersect non-trivially. As both are non-trivial, cyclic groups of equal, $2$-power order, they would share a point of order $2$. The element in $\left\langle \phi([2^{r}]P_{2^{5+r}}) \right\rangle$ of order $2$ is $\phi([2^{4+r}]P_{2^{5+r}})$ and the element of $\left\langle \phi(Q_{2^{5+r}}) \right\rangle$ of order $2$ is $\phi([16]Q_{2^{5+r}})$. If $\phi([2^{4+r}]P_{2^{5+r}}) = \phi([16]Q_{2^{5+r}})$, then $[2^{4+r}]P_{2^{5+r}} - [16]Q_{2^{5+r}} \in \operatorname{Ker}(\phi)$. In other words, $[2^{4+r}]P_{2^{5+r}} - [16]Q_{2^{5+r}} \in \left\langle Q_{2^{r}} \right\rangle \subseteq \left\langle Q_{2^{5+r}} \right\rangle$ and hence, $[2^{4+r}]P_{2^{5+r}} \in \left\langle Q_{2^{5+r}} \right\rangle$. But then $[2^{4+r}]P_{2^{5+r}}$ is an element of order $2$ in $\left\langle P_{2^{5+r}} \right\rangle \bigcap \left\langle Q_{2^{5+r}} \right\rangle = \left\{\mathcal{O}\right\}$, a contradiction.

By Lemma \ref{Borel 1}, there is a basis of $E[2^{m}]$ such that $\overline{\rho}_{E,2^{m}}(G_{\QQ})$ is conjugate to a subgroup of $\operatorname{GL}\left(2, \ZZ / 2^{m} \ZZ\right)$ of the form
$$\left\langle \begin{bmatrix}
A_{1} & C_{1} \\ B_{1} & D_{1}
\end{bmatrix}, \ldots, \begin{bmatrix}
A_{s} & C_{s} \\ B_{s} & D_{s}
\end{bmatrix} \right\rangle \subseteq \operatorname{GL}(2, \ZZ / 2^{m} \ZZ)$$
such that $2^{r}$ divides $C_{1}, \ldots, C_{s}$. By Theorem \ref{thm-rzb}, $5+r \geq m$ and so, $\overline{\rho}_{E,2^{5+r}}(G_{\QQ})$ is the full lift of $\overline{\rho}_{E,2^{m}}(G_{\QQ})$ to level $2^{5+r}$.

We now classify $\overline{\rho}_{E',32}(G_{\QQ})$, given $\overline{\rho}_{E,2^{r+5}}(G_{\QQ})$. Suppose that $\left\{P_{2^{r+5}}, Q_{2^{r+5}}\right\}$ is a basis of $E[2^{r+5}]$. Then $\left\{\phi([2^{r}]P_{2^{5+r}}), \phi(Q_{2^{5+r}})\right\}$ is a basis of $E'[32]$. Let $\sigma$ be a Galois automorphism of $\overline{\QQ}$ such that $\sigma(P_{2^{r+5}}) = [A]P_{2^{r+5}} + [B]Q_{2^{r+5}}$ and $\sigma(Q_{2^{r+5}}) = [C]P_{2^{r+5}} + [D]Q_{2^{r+5}}$ for some integers $A$, $B$, $C$, and $D$ such that $AD - BC$ is not equal to $0$. Moreover, note that we may say that $2^{r}$ divides $C$. Then
$$\sigma(\phi([2^{r}]P_{2^{5+r}})) = \phi(\sigma([2^{r}]P_{2^{5+r}})) = \phi([2^{r}]\sigma(P_{2^{5+r}})) = [2^{r}]\phi([A]P_{2^{5+r}} + [B]Q_{2^{5+r}}) =$$
$$[A]\phi([2^{r}]P_{2^{5+r}}) + [2^{r} \cdot B]\phi(Q_{2^{5+r}}).$$
Finally,
$$\sigma(\phi(Q_{2^{5+r}})) = \phi(\sigma(Q_{2^{5+r}})) = \phi(\sigma(Q_{2^{5+r}})) = \phi([C]P_{2^{5+r}} + [D]Q_{2^{5+r}}) = $$
$$\phi\left(\left[\frac{C}{2^{r}} \cdot 2^{r}\right] P_{2^{5+r}} + [D]Q_{2^{5+r}}\right) = \phi\left(\left[\frac{C}{2^{r}}\right] [2^{r}]P_{2^{5+r}} + [D]Q_{2^{5+r}}\right) =$$
$$\phi\left(\left[\frac{C}{2^{r}}\right][2^{r}]P_{2^{5+r}}\right) + \phi\left([D]Q_{2^{5+r}}\right) = \left[\frac{C}{2^{r}}\right]\phi([2^{r}]P_{2^{5+r}}) + [D]\phi(Q_{2^{5+r}}).$$
In other words, if $\begin{bmatrix} A & C \\ B & D \end{bmatrix}$ is a matrix representing the group action of a Galois automorphism $\sigma$ of $\overline{\QQ}$ on $E[2^{5+r}]$, then $\begin{bmatrix} A & \frac{C}{2^{r}} \\ 2^{r} \cdot B & D \end{bmatrix}$ represents the group of $\sigma$ on $E'[32]$. As all $\phi$ is surjective, the group action of $\sigma$ on $E'[32]$ is determined by the group action of $\sigma$ on $E[2^{5+r}]$.

We conclude the proof. If $m = 0$, then the level of $\rho_{E,2^{\infty}}(G_{\QQ})$ is equal to $1$ and $\overline{\rho}_{E,32}(G_{\QQ})$ is conjugate to $\operatorname{GL}(2, \ZZ / 32 \ZZ)$. Moreover $r = 0$ and $\left\{ \phi(P_{32}), \phi(Q_{32}) \right\}$ is a basis of $E'[32]$. For each matrix $M = \begin{bmatrix} A & C \\ B & D \end{bmatrix}$ of $\operatorname{GL}(2, \ZZ / 32 \ZZ)$, there is a Galois automorphism $\sigma$ of $\overline{\QQ}$ such that $\sigma(P_{32}) = [A]P_{32} + [B]Q_{32}$ and $\sigma(Q_{32}) = [C]P_{32} + [D]Q_{32}$. Then $\sigma(\phi(P_{32})) = \phi(\sigma(P_{32})) = \phi([A]P_{32} + [B]Q_{32}) = [A]\phi(P_{32}) + [B]\phi(Q_{32})$. Next, we have $\sigma(\phi(Q_{32})) = \phi(\sigma(Q_{32})) = \phi([C]P_{32} + [D]Q_{32}) = [C]\phi(P_{32}) + [D]\phi(Q_{32})$. Hence, $\sigma$ corresponds to the matrix $\begin{bmatrix}
A & C \\ B & D
\end{bmatrix}$ on $E'[32]$. Thus, $\overline{\rho}_{E',32}(G_{\QQ})$ is conjugate to $\operatorname{GL}(2, \ZZ / 32 \ZZ)$ and $\rho_{E',2^{\infty}}(G_{\QQ})$ is conjugate to $\operatorname{GL}(2, \ZZ_{2})$.

Now suppose that $m = 1$ and the level of $\rho_{E,2^{\infty}}(G_{\QQ})$ is equal to $2$. Moreover, $r = 0$ or $r = 1$. We have that $\overline{\rho}_{E,2^{5+r}}(G_{\QQ})$ is conjugate to the full lift of $\overline{\rho}_{E,2}(G_{\QQ})$ to level $2^{5+r}$. By Lemma \ref{Borel 1}, $\overline{\rho}_{E,2}(G_{\QQ})$ is conjugate to a subgroup of $\operatorname{GL}(2, \ZZ / 2 \ZZ)$ of the form
$$\left\langle \begin{bmatrix} A_{1} & C_{1} \\ B_{1} & D_{1} \end{bmatrix}, \ldots, \begin{bmatrix} A_{s} & C_{s} \\ B_{s} & D_{s} \end{bmatrix} \right\rangle \subseteq \operatorname{GL}(2, \ZZ / 2 \ZZ)$$
such that $2^{r}$ (which is equal to $1$ or $2$) divides $C_{1}, \ldots, C_{s}$. By Lemma \ref{lifting with kernel}, $\overline{\rho}_{E,2^{5+r}}(G_{\QQ})$ is equal to (up to base change)
$$\left\langle \begin{bmatrix} A_{1} & C_{1} \\ B_{1} & D_{1} \end{bmatrix}, \ldots, \begin{bmatrix} A_{s} & C_{s} \\ B_{s} & D_{s} \end{bmatrix}, \begin{bmatrix} 1 & 2 \\ 0 & 1 \end{bmatrix}, \begin{bmatrix} 1 & 0 \\ 2 & 1 \end{bmatrix}, \begin{bmatrix} 3 & 0 \\ 0 & 1 \end{bmatrix}, \begin{bmatrix} 5 & 0 \\ 0 & 1 \end{bmatrix}, \begin{bmatrix} 1 & 0 \\ 0 & 3 \end{bmatrix}, \begin{bmatrix} 1 & 0 \\ 0 & 5 \end{bmatrix} \right\rangle \subseteq \operatorname{GL}(2, \ZZ / 2^{5+r} \ZZ).$$
We have that $\left\{\phi([2^{r}]P_{2^{5+r}}), \phi(Q_{2^{5+r}})\right\}$ is a basis of $E'[32]$. Following the previous paragraph, we see that $\overline{\rho}_{E',32}(G_{\QQ})$ is conjugate to
$$\left\langle \begin{bmatrix} A_{1} & \frac{C_{1}}{2^{r}} \\ 2^{r} \cdot B_{1} & D_{1} \end{bmatrix}, \ldots, \begin{bmatrix} A_{s} & \frac{C_{s}}{2^{r}} \\ 2^{r} \cdot B_{s} & D_{s} \end{bmatrix}, \begin{bmatrix} 1 & 2^{1-r} \\ 0 & 1 \end{bmatrix}, \begin{bmatrix} 1 & 0 \\ 2^{1+r} & 1 \end{bmatrix}, \begin{bmatrix} 3 & 0 \\ 0 & 1 \end{bmatrix}, \begin{bmatrix} 5 & 0 \\ 0 & 1 \end{bmatrix}, \begin{bmatrix} 1 & 0 \\ 0 & 3 \end{bmatrix}, \begin{bmatrix} 1 & 0 \\ 0 & 5 \end{bmatrix} \right\rangle$$
modulo $32$. Finally, suppose that $2 \leq m \leq 5$ and the level of $\rho_{E,2^{\infty}}(G_{\QQ})$ is equal to $4 \leq 2^{m} \leq 32$. Moreover, $0 \leq r \leq m$. We say that $\overline{\rho}_{E,2^{5+r}}(G_{\QQ})$ is conjugate to the full lift of $\overline{\rho}_{E,2^{m}}(G_{\QQ})$ to level $2^{5+r}$. By Lemma \ref{Borel 1}, $\overline{\rho}_{E,2^{m}}(G_{\QQ})$ is conjugate to a subgroup of $\operatorname{GL}(2, \ZZ / 2^{m} \ZZ)$ of the form
$$\left\langle \begin{bmatrix}
A_{1} & C_{1} \\ B_{1} & D_{1}
\end{bmatrix}, \ldots, \begin{bmatrix}
A_{s} & C_{s} \\ B_{s} & D_{s}
\end{bmatrix} \right\rangle \subseteq \operatorname{GL}(2, \ZZ / 2^{m} \ZZ)$$
such that $2^{r}$ divides $C_{1}, \ldots, C_{s}$. By Lemma \ref{lifting with kernel}, $\overline{\rho}_{E,2^{5+r}}(G_{\QQ})$ is equal to (up to base change)
$$\left\langle \begin{bmatrix} A_{1} & C_{1} \\ B_{1} & D_{1} \end{bmatrix}, \ldots, \begin{bmatrix} A_{s} & C_{s} \\ B_{s} & D_{s} \end{bmatrix}, \begin{bmatrix} 1 & 2^{m} \\ 0 & 1 \end{bmatrix}, \begin{bmatrix} 1 & 0 \\ 2^{m} & 1 \end{bmatrix}, \begin{bmatrix} 1+2^{m} & 0 \\ 0 & 1 \end{bmatrix}, \begin{bmatrix} 1 & 0 \\ 0 & 1+2^{m} \end{bmatrix} \right\rangle \subseteq \operatorname{GL}(2, \ZZ / 2^{5+r} \ZZ).$$
We have that $\left\{\phi([2^{r}]P_{2^{5+r}}, \phi(Q_{2^{5+r}})\right\}$ is a basis of $E'[32]$. Following the previous paragraph, we have that $\overline{\rho}_{E',32}(G_{\QQ})$ is conjugate to
$$\left\langle \begin{bmatrix} A_{1} & \frac{C_{1}}{2^{r}} \\ 2^{r} \cdot B_{1} & D_{1} \end{bmatrix}, \ldots, \begin{bmatrix} A_{s} & \frac{C_{s}}{2^{r}} \\ 2^{r} \cdot B_{s} & D_{s} \end{bmatrix}, \begin{bmatrix} 1 & 2^{m-r} \\ 0 & 1 \end{bmatrix}, \begin{bmatrix} 1 & 0 \\ 2^{m+r} & 1 \end{bmatrix}, \begin{bmatrix} 1+2^{m} & 0 \\ 0 & 1 \end{bmatrix}, \begin{bmatrix} 1 & 0 \\ 0 & 1+2^{m} \end{bmatrix} \right\rangle.$$

\end{proof}

\begin{corollary}\label{contains -Id}
 
Let $E/\QQ$ and $E'/\QQ$ be elliptic curves. Suppose that $E$ is $\QQ$-isogenous to $E'$ by an isogeny $\phi$ with a finite, cyclic, $\QQ$-rational kernel. Then $\overline{\rho}_{E,32}(G_{\QQ})$ contains $\operatorname{-Id}$ if and only if $\overline{\rho}_{E',32}(G_{\QQ})$ contains $\operatorname{-Id}$.
    
\end{corollary}

\begin{proof}

Suppose that $\operatorname{-Id}$ is an element of $\rho_{E,2^{\infty}}(G_{\QQ})$. Let $N$ be a non-negative integer. Then $\operatorname{-Id}$ is an element of $\overline{\rho}_{E,2^{N}}(G_{\QQ})$. Let $r$ be the non-negative integer such that $2^{r}$ is the greatest power of $2$ that divides the order of $\operatorname{Ker}(\phi)$. As $\operatorname{-Id}$ is in the center of $\operatorname{GL}\left(2, \ZZ / 2^{5+r} \ZZ\right)$, it does not matter what basis we use for $E\left[2^{5+r}\right]$. By Lemma \ref{2-adic Galois Images}, $\operatorname{-Id}$ is an element of $\overline{\rho}_{E',32}(G_{\QQ})$.
\end{proof}

\begin{corollary}\label{coprime isogeny-degree}

Let $E/\QQ$ and $E'/\QQ$ be elliptic curves. Suppose that $E$ is $\QQ$-isogenous to $E'$ by an isogeny of odd degree that is defined over $\QQ$ with a finite, cyclic kernel. Then $\overline{\rho}_{E,32}(G_{\QQ})$ is conjugate to $\overline{\rho}_{E',32}(G_{\QQ})$.
\end{corollary}

\begin{proof}

Use Lemma \ref{2-adic Galois Images} with $r = 0$.

\end{proof}

\begin{corollary}\label{2-adic Galois Images corollary 1}
Let $E/\QQ$ be an elliptic curve such that $E(\QQ)_{\text{tors}} = \left\langle Q_{2} \right\rangle \cong \ZZ / 2 \ZZ$ and $C_{2}(E) = 2$. Then $E$ is $2$-isogenous to $E / \left\langle Q_{2} \right\rangle$. The level of $\rho_{E,2^{\infty}}(G_{\QQ})$ is equal to $2^{m}$ where $1 \leq m \leq 5$. There is a basis $\left\{P_{2^{m}}, Q_{2^{m}}\right\}$ of $E[2^{m}]$ such that $\rho_{E,2^{\infty}}(G_{\QQ})$ is conjugate to the full lift of
$$\overline{\rho}_{E,2^{m}}(G_{\QQ}) = \left\langle
\begin{bmatrix}
A_{1} & C_{1} \\ B_{1} & D_{1}
\end{bmatrix},
\ldots,
\begin{bmatrix}
A_{s} & C_{s} \\ B_{s} & D_{s}
\end{bmatrix}
\right\rangle$$
where $C_{1}$, \ldots, $C_{s}$ are even.

\begin{enumerate}

\item If $m = 1$, then, $\overline{\rho}_{E,2}(G_{\QQ})$ is conjugate to $\left\langle \begin{bmatrix} 1 & 0 \\ 1 & 1 \end{bmatrix} \right\rangle$, and both $\overline{\rho}_{E,32}(G_{\QQ})$ and $\overline{\rho}_{E',32}(G_{\QQ})$ are conjugate to the full lift of $\left\langle \begin{bmatrix} 1 & 0 \\ 1 & 1 \end{bmatrix} \right\rangle$.
\item If $2 \leq m \leq 5$, then $\overline{\rho}_{E,64}(G_{\QQ})$ is conjugate to
$$\left\langle
\begin{bmatrix}
A_{1} & C_{1} \\ B_{1} & D_{1}
\end{bmatrix},
\ldots,
\begin{bmatrix}
A_{s} & C_{s} \\ B_{s} & D_{s}
\end{bmatrix},
\begin{bmatrix}
1+2^{m} & 0 \\ 0 & 1
\end{bmatrix},
\begin{bmatrix}
1 & 0 \\ 0 & 1+2^{m}
\end{bmatrix},
\begin{bmatrix}
1 & 2^{m} \\ 0 & 1
\end{bmatrix},
\begin{bmatrix}
1 & 0 \\ 2^{m} & 1
\end{bmatrix}
\right\rangle$$
and $\overline{\rho}_{E/\left\langle Q_{2} \right\rangle,32}(G_{\QQ})$ is conjugate to
$$\left\langle
\begin{bmatrix}
A_{1} & \frac{C_{1}}{2} \\ 2 \cdot B_{1} & D_{1}
\end{bmatrix},
\ldots,
\begin{bmatrix}
A_{s} & \frac{C_{s}}{2} \\ 2 \cdot B_{s} & D_{s}
\end{bmatrix},
\begin{bmatrix}
1+2^{m} & 0 \\ 0 & 1
\end{bmatrix},
\begin{bmatrix}
1 & 0 \\ 0 & 1+2^{m}
\end{bmatrix},
\begin{bmatrix}
1 & 2^{m-1} \\ 0 & 1
\end{bmatrix},
\begin{bmatrix}
1 & 0 \\ 2^{m+1} & 1
\end{bmatrix}
\right\rangle$$
\end{enumerate}
\end{corollary}

\begin{proof}
In our given situation, the $2$-adic subgraph of the isogeny graph associated to the $\QQ$-isogeny class of $E$ is of $L_{2}(2)$ type:
\begin{center}
\begin{tikzcd}
E \arrow[r, no head, "2"] & E / \left\langle Q_{2} \right\rangle
\end{tikzcd}
\end{center}
where the edge represents a $\QQ$-isogeny of degree $2$. Moreover, $2$ divides the order of $\operatorname{Ker}(\phi)$ and $4$ does not divide the order of $\operatorname{Ker}(\phi)$. We break up the proof into cases:
\begin{enumerate}
    \item If $\rho_{E,2^{\infty}}(G_{\QQ})$ is a group of level $2$, then $\rho_{E,2^{\infty}}(G_{\QQ})$ is the full lift of $\overline{\rho}_{E,2}(G_{\QQ})$. As $E(\QQ)_{\text{tors}} \cong \ZZ / 2 \ZZ$, it is clear to see that $\overline{\rho}_{E,2}(G_{\QQ})$ is conjugate to $\left\langle \begin{bmatrix}
    1 & 0 \\ 1 & 1
    \end{bmatrix} \right\rangle$. By Lemma \ref{lifting with kernel} and Lemma \ref{kernel of pi 2}, $\overline{\rho}_{E,64}(G_{\QQ})$ is conjugate to
    $$\left\langle
    \begin{bmatrix} 1 & 0 \\ 1 & 1 \end{bmatrix},
    \begin{bmatrix} 1 & 2 \\ 0 & 1 \end{bmatrix},
    \begin{bmatrix} 1 & 0 \\ 2 & 1 \end{bmatrix},
    \begin{bmatrix} 3 & 0 \\ 0 & 1 \end{bmatrix},
    \begin{bmatrix} 1 & 0 \\ 0 & 3 \end{bmatrix},
    \begin{bmatrix} 5 & 0 \\ 0 & 1 \end{bmatrix},
    \begin{bmatrix} 1 & 0 \\ 0 & 5 \end{bmatrix}
    \right\rangle.$$
    By Lemma \ref{2-adic Galois Images}, $\overline{\rho}_{E/\left\langle Q_{2} \right\rangle,32}(G_{\QQ})$ is conjugate to
    $$\left\langle
    \begin{bmatrix} 1 & 0 \\ 2 & 1 \end{bmatrix},
    \begin{bmatrix} 1 & 1 \\ 0 & 1 \end{bmatrix},
    \begin{bmatrix} 1 & 0 \\ 4 & 1 \end{bmatrix},
    \begin{bmatrix} 3 & 0 \\ 0 & 1 \end{bmatrix},
    \begin{bmatrix} 1 & 0 \\ 0 & 3 \end{bmatrix},
    \begin{bmatrix} 5 & 0 \\ 0 & 1 \end{bmatrix},
    \begin{bmatrix} 1 & 0 \\ 0 & 5 \end{bmatrix}
    \right\rangle$$
    which is conjugate to $\overline{\rho}_{E,32}(G_{\QQ})$ over $\operatorname{GL}(2, \ZZ / 32 \ZZ)$ by $\begin{bmatrix}
    0 & -1 \\ -1 & 0
    \end{bmatrix}$.
    \item We now move on to the case that $2 \leq m \leq 5$. By Lemma \ref{Borel 1}, there is a basis of $E[2^{m}]$, such that $\overline{\rho}_{E,2^{m}}(G_{\QQ})$ is conjugate to
    $$\left\langle
    \begin{bmatrix}
    A_{1} & C_{1} \\ B_{1} & D_{1}
    \end{bmatrix},
    \ldots,
    \begin{bmatrix}
    A_{s} & C_{s} \\ B_{s} & D_{s}
    \end{bmatrix}
    \right\rangle$$
    such that $C_{1}$, \ldots, $C_{s}$ are even. By Lemma \ref{lifting with kernel} and Lemma \ref{kernel of pi 1}, $\overline{\rho}_{E,64}(G_{\QQ})$ is conjugate to
    $$\left\langle
    \begin{bmatrix}
    A_{1} & C_{1} \\ B_{1} & D_{1}
    \end{bmatrix},
    \ldots,
    \begin{bmatrix}
    A_{s} & C_{s} \\ B_{s} & D_{s}
    \end{bmatrix},
    \begin{bmatrix}
    1 + 2^{m} & 0 \\ 0 & 1
    \end{bmatrix},
    \begin{bmatrix}
    1 & 0 \\ 0 & 1 + 2^{m}
    \end{bmatrix},
    \begin{bmatrix}
    1 & 2^{m} \\ 0 & 1
    \end{bmatrix},
    \begin{bmatrix}
    1 & 0 \\ 2^{m} & 1
    \end{bmatrix}
    \right\rangle.$$
    By Lemma \ref{2-adic Galois Images}, $\overline{\rho}_{E/\left\langle Q_{2} \right\rangle,32}(G_{\QQ})$ is conjugate to
    $$\left\langle
    \begin{bmatrix}
    A_{1} & \frac{C_{1}}{2} \\ 2 \cdot B_{1} & D_{1}
    \end{bmatrix},
    \ldots,
    \begin{bmatrix}
    A_{s} & \frac{C_{s}}{2} \\ 2 \cdot B_{s} & D_{s}
    \end{bmatrix},
    \begin{bmatrix}
    1 + 2^{m} & 0 \\ 0 & 1
    \end{bmatrix},
    \begin{bmatrix}
    1 & 0 \\ 0 & 1 + 2^{m}
    \end{bmatrix},
    \begin{bmatrix}
    1 & 2^{m-1} \\ 0 & 1
    \end{bmatrix},
    \begin{bmatrix}
    1 & 0 \\ 2^{m+1} & 1
    \end{bmatrix}
    \right\rangle.$$
\end{enumerate}
\end{proof}

\begin{corollary}\label{2-adic Galois Images corollary 2}
Let $E/\QQ$ be an elliptic curve such that $E[2](\QQ)_{\text{tors}} = \left\langle P_{2}, Q_{2} \right\rangle \cong \ZZ / 2 \ZZ \times \ZZ / 2 \ZZ$ and $C_{2}(E) = 4$. Then $E$ is $2$-isogenous to the elliptic curves $E/\left\langle Q_{2} \right\rangle$, $E/\left\langle P_{2} \right\rangle$, and $E/\left\langle P_{2}+Q_{2} \right\rangle$. The level of $\rho_{E,2^{\infty}}(G_{\QQ})$ is equal to $2^{m}$ where $1 \leq m \leq 5$. There is a basis $\left\{P_{2^{m}}, Q_{2^{m}}\right\}$ of $E[2^{m}]$ such that $\rho_{E,2^{\infty}}(G_{\QQ})$ is conjugate to the full lift of
$$\overline{\rho}_{E,2^{m}}(G_{\QQ}) = \left\langle \begin{bmatrix} A_{1} & C_{1} \\ B_{1} & D_{1} \end{bmatrix}, \ldots, \begin{bmatrix} A_{s} & C_{s} \\ B_{s} & D_{s} \end{bmatrix} \right\rangle$$
such that $B_{1}$, \ldots, $B_{s}$, $C_{1}$, \ldots, $C_{s}$ are even.
\begin{enumerate}
    \item If $m = 1$, then $\rho_{E,2^{\infty}}(G_{\QQ})$ is the full lift of the trivial subgroup of $\operatorname{GL}(2, \ZZ / 2 \ZZ)$. Moreover, $\rho_{E/\left\langle Q_{2} \right\rangle, 2^{\infty}}(G_{\QQ})$, $\rho_{E/\left\langle P_{2} \right\rangle, 2^{\infty}}(G_{\QQ})$, and $\rho_{E/\left\langle P_{2}+Q_{2} \right\rangle, 2^{\infty}}(G_{\QQ})$ are all conjugate to the full lift of
    $$\left\{\begin{bmatrix} \ast & \ast \\ 0 & \ast \end{bmatrix} \right\} = \left\langle \begin{bmatrix} 3 & 0 \\ 0 & 1 \end{bmatrix}, \begin{bmatrix} 1 & 1 \\ 0 & 1 \end{bmatrix}, \operatorname{-Id} \right\rangle \subseteq \operatorname{GL}(2, \ZZ / 4 \ZZ)$$
    the ``generic'' group that generates a $4$-isogeny.
    \item If $2 \leq m \leq 5$, then $\overline{\rho}_{E,64}(G_{\QQ})$ is conjugate to
    $$\left\langle \begin{bmatrix} A_{1} & C_{1} \\ B_{1} & D_{1} \end{bmatrix}, \ldots, \begin{bmatrix} A_{s} & C_{s} \\ B_{s} & D_{s} \end{bmatrix},
    \begin{bmatrix} 1+2^{m} & 0 \\ 0 & 1 \end{bmatrix},
    \begin{bmatrix} 1 & 0 \\ 0 & 1+2^{m} \end{bmatrix},
    \begin{bmatrix} 1 & 2^{m} \\ 0 & 1 \end{bmatrix},
    \begin{bmatrix} 1 & 0 \\ 2^{m} & 1 \end{bmatrix} \right\rangle$$
    and moreover,
    \begin{itemize}
        \item $\overline{\rho}_{E/\left\langle Q_{2} \right\rangle, 32}(G_{\QQ})$ is conjugate to
        $$\left\langle \begin{bmatrix} A_{1} & \frac{C_{1}}{2} \\ 2 \cdot B_{1} & D_{1} \end{bmatrix}, \ldots, \begin{bmatrix} A_{s} & \frac{C_{s}}{2} \\ 2 \cdot B_{s} & D_{s} \end{bmatrix},
    \begin{bmatrix} 1+2^{m} & 0 \\ 0 & 1 \end{bmatrix},
    \begin{bmatrix} 1 & 0 \\ 0 & 1+2^{m} \end{bmatrix},
    \begin{bmatrix} 1 & 2^{m-1} \\ 0 & 1 \end{bmatrix},
    \begin{bmatrix} 1 & 0 \\ 2^{m+1} & 1 \end{bmatrix}
    \right\rangle$$
    
        \item $\overline{\rho}_{E/\left\langle P_{2} \right\rangle, 32}(G_{\QQ})$ is conjugate to
        $$\left\langle \begin{bmatrix} D_{1} & \frac{B_{1}}{2} \\ 2 \cdot C_{1} & A_{1} \end{bmatrix}, \ldots, \begin{bmatrix} D_{s} & \frac{B_{s}}{2} \\ 2 \cdot C_{s} & A_{s} \end{bmatrix},
    \begin{bmatrix} 1+2^{m} & 0 \\ 0 & 1 \end{bmatrix},
    \begin{bmatrix} 1 & 0 \\ 0 & 1+2^{m} \end{bmatrix},
    \begin{bmatrix} 1 & 2^{m-1} \\ 0 & 1 \end{bmatrix},
    \begin{bmatrix} 1 & 0 \\ 2^{m+1} & 1 \end{bmatrix}
    \right\rangle$$
    \item and $\overline{\rho}_{E/\left\langle P_{2}+Q_{2} \right\rangle, 32}(G_{\QQ})$ is conjugate to
    $$\left\langle \begin{bmatrix} A_{i}-B_{i} & \frac{A_{i}+C_{i}-B_{i}-D_{i}}{2} \\ 2 \cdot B_{i} & B_{i}+D_{i} \end{bmatrix},
    \begin{bmatrix} 1+2^{m} & 0 \\ 0 & 1 \end{bmatrix},
    \begin{bmatrix} 1 & 0 \\ 0 & 1+2^{m} \end{bmatrix},
    \begin{bmatrix} 1 & 2^{m-1} \\ 0 & 1 \end{bmatrix},
    \begin{bmatrix} 1 & 0 \\ 2^{m+1} & 1 \end{bmatrix}
    \right\rangle$$
    for $i \in \{1, \ldots, s\}$.
    \end{itemize}
\end{enumerate}
\end{corollary}

\begin{proof}
In our given situation, the $2$-adic subgraph of the isogeny graph associated to the $\QQ$-isogeny class of $E$ is of $T_{4}$ type
\begin{center}
\begin{tikzcd}
                                   & E/\left\langle P_{2} + Q_{2} \right\rangle        &                                    \\
                                   & E \arrow[u, no head, "2"] \arrow[ld, no head, "2"'] \arrow[rd, no head, "2"] &                                    \\
E/\left\langle P_{2} \right\rangle &                                                   & E/\left\langle Q_{2} \right\rangle
\end{tikzcd}
\end{center}
The order of $\operatorname{Ker}(\phi)$ is divisible by $2$ and not divisible by $4$. We break up the proof again into cases.

\begin{enumerate}
    \item If $\rho_{E,2^{\infty}}(G_{\QQ})$ is a group of level $2$, then $\overline{\rho}_{E,2}(G_{\QQ}) = \left\{\operatorname{Id}\right\}$ and $\overline{\rho}_{E,64}(G_{\QQ})$ is the full lift of $\overline{\rho}_{E,2}(G_{\QQ})$. By Lemma \ref{lifting with kernel} and Lemma \ref{kernel of pi 2}, $\overline{\rho}_{E,64}(G_{\QQ})$ is conjugate to
    $$\left\langle
    \begin{bmatrix} 3 & 0 \\ 0 & 1 \end{bmatrix},
    \begin{bmatrix} 5 & 0 \\ 0 & 1 \end{bmatrix},
    \begin{bmatrix} 1 & 0 \\ 0 & 3 \end{bmatrix},
    \begin{bmatrix} 1 & 0 \\ 0 & 5 \end{bmatrix},
    \begin{bmatrix} 1 & 2 \\ 0 & 1 \end{bmatrix},
    \begin{bmatrix} 1 & 0 \\ 2 & 1 \end{bmatrix}
    \right\rangle.$$
    By Lemma \ref{2-adic Galois Images}, $\overline{\rho}_{E/\left\langle P_{2} \right\rangle,32}(G_{\QQ})$, $\overline{\rho}_{E/\left\langle Q_{2} \right\rangle,32}(G_{\QQ})$, and $\overline{\rho}_{E/\left\langle P_{2} + Q_{2} \right\rangle,32}(G_{\QQ})$ are all conjugate to
    $$\left\langle
    \begin{bmatrix} 3 & 0 \\ 0 & 1 \end{bmatrix},
    \begin{bmatrix} 5 & 0 \\ 0 & 1 \end{bmatrix},
    \begin{bmatrix} 1 & 0 \\ 0 & 3 \end{bmatrix},
    \begin{bmatrix} 1 & 0 \\ 0 & 5 \end{bmatrix},
    \begin{bmatrix} 1 & 1 \\ 0 & 1 \end{bmatrix},
    \begin{bmatrix} 1 & 0 \\ 4 & 1 \end{bmatrix}
    \right\rangle$$
    which is conjugate to
    $$\left\langle
    \begin{bmatrix} 3 & 0 \\ 0 & 1 \end{bmatrix},
    \begin{bmatrix} 1 & 1 \\ 0 & 1 \end{bmatrix},
    \operatorname{-Id},
    \begin{bmatrix} 1 & 4 \\ 0 & 1 \end{bmatrix},
    \begin{bmatrix} 1 & 0 \\ 4 & 1 \end{bmatrix},
    \begin{bmatrix} 5 & 0 \\ 0 & 1 \end{bmatrix},
    \begin{bmatrix} 1 & 0 \\ 0 & 5 \end{bmatrix}
    \right\rangle,$$
    the full lift of $\left\langle
    \begin{bmatrix} 3 & 0 \\ 0 & 1 \end{bmatrix},
    \begin{bmatrix} 1 & 1 \\ 0 & 1 \end{bmatrix},
    \operatorname{-Id}
    \right\rangle \subseteq \operatorname{GL}(2, \ZZ / 4 \ZZ)$ to level $32$.
    
    \item Now suppose that $2 \leq m \leq 5$. 
    \begin{itemize}
    \item We use the basis $E[2] = \left\langle Q_{2}, P_{2} \right\rangle$ and $E[2^{m}] = \left\langle Q_{2^{m}}, P_{2^{m}} \right\rangle$. By Lemma \ref{Borel 1}, $\overline{\rho}_{E,2^{m}}(G_{\QQ})$ is conjugate to a subgroup of $\operatorname{GL}(2, \ZZ / 2^{m} \ZZ)$ of the form
    $$\left\langle
    \begin{bmatrix} A_{1} & C_{1} \\ B_{1} & D_{1} \end{bmatrix},
    \ldots,
    \begin{bmatrix} A_{s} & C_{s} \\ B_{s} & D_{s} \end{bmatrix}
    \right\rangle$$
    such that $C_{1}$, \ldots, $C_{s}$ are even. By Lemma \ref{lifting with kernel} and Lemma \ref{kernel of pi 1}, $\overline{\rho}_{E,64}(G_{\QQ})$ is conjugate to
    $$\left\langle
    \begin{bmatrix} A_{1} & C_{1} \\ B_{1} & D_{1} \end{bmatrix},
    \ldots,
    \begin{bmatrix} A_{s} & C_{s} \\ B_{s} & D_{s} \end{bmatrix},
    \begin{bmatrix} 1+2^{m} & 0 \\ 0 & 1 \end{bmatrix},
    \begin{bmatrix} 1 & 0 \\ 0 & 1+2^{m} \end{bmatrix},
    \begin{bmatrix} 1 & 2^{m} \\ 0 & 1 \end{bmatrix},
    \begin{bmatrix} 1 & 0 \\ 2^{m} & 1 \end{bmatrix}
    \right\rangle.$$
    By Lemma \ref{2-adic Galois Images}, $\overline{\rho}_{E/\left\langle Q_{2} \right\rangle, 32}(G_{\QQ})$ is conjugate to
    $$\left\langle
    \begin{bmatrix} A_{1} & \frac{C_{1}}{2} \\ 2 \cdot B_{1} & D_{1} \end{bmatrix},
    \ldots,
    \begin{bmatrix} A_{s} & \frac{C_{s}}{2} \\ 2 \cdot B_{s} & D_{s} \end{bmatrix},
    \begin{bmatrix} 1+2^{m} & 0 \\ 0 & 1 \end{bmatrix},
    \begin{bmatrix} 1 & 0 \\ 0 & 1+2^{m} \end{bmatrix},
    \begin{bmatrix} 1 & 2^{m-1} \\ 0 & 1 \end{bmatrix},
    \begin{bmatrix} 1 & 0 \\ 2^{m+1} & 1 \end{bmatrix}
    \right\rangle.$$
    \item Using the basis $E[2] = \left\langle Q_{2}, P_{2} \right\rangle$ and $E[2^{m}] = \left\langle Q_{2^{m}}, P_{2^{m}} \right\rangle$, $\overline{\rho}_{E,2^{m}}(G_{\QQ})$ is conjugate to
    $$\left\langle
    \begin{bmatrix} D_{1} & B_{1} \\ C_{1} & A_{1} \end{bmatrix},
    \ldots,
    \begin{bmatrix} D_{s} & B_{s} \\ C_{s} & A_{s} \end{bmatrix}
    \right\rangle$$
    and $B_{1}$, \ldots, $B_{s}$ are even. By Lemma \ref{lifting with kernel} and Lemma \ref{kernel of pi 1}, $\overline{\rho}_{E,64}(G_{\QQ})$ is conjugate to
    $$\left\langle
    \begin{bmatrix} D_{1} & B_{1} \\ C_{1} & A_{1} \end{bmatrix},
    \ldots,
    \begin{bmatrix} D_{s} & B_{s} \\ C_{s} & A_{s} \end{bmatrix},
    \begin{bmatrix} 1+2^{m} & 0 \\ 0 & 1 \end{bmatrix},
    \begin{bmatrix} 1 & 0 \\ 0 & 1+2^{m} \end{bmatrix},
    \begin{bmatrix} 1 & 2^{m} \\ 0 & 1 \end{bmatrix},
    \begin{bmatrix} 1 & 0 \\ 2^{m} & 1 \end{bmatrix}
    \right\rangle.$$
    By Lemma \ref{2-adic Galois Images}, $\overline{\rho}_{E/\left\langle P_{2} \right\rangle, 32}(G_{\QQ})$ is conjugate to
    $$\left\langle
    \begin{bmatrix} D_{1} & \frac{B_{1}}{2} \\ 2 \cdot C_{1} & A_{1} \end{bmatrix},
    \ldots,
    \begin{bmatrix} D_{s} & \frac{B_{s}}{2} \\ 2 \cdot C_{s} & A_{s} \end{bmatrix},
    \begin{bmatrix} 1+2^{m} & 0 \\ 0 & 1 \end{bmatrix},
    \begin{bmatrix} 1 & 0 \\ 0 & 1+2^{m} \end{bmatrix},
    \begin{bmatrix} 1 & 2^{m-1} \\ 0 & 1 \end{bmatrix},
    \begin{bmatrix} 1 & 0 \\ 2^{m+1} & 1 \end{bmatrix}
    \right\rangle.$$
    
    \item Using the basis $E[2] = \left\langle P_{2}, P_{2}+Q_{2} \right\rangle$ and $E[2^{m}] = \left\langle P_{2^{m}}, P_{2^{m}}+Q_{2^{m}} \right\rangle$, $\overline{\rho}_{E,2^{m}}(G_{\QQ})$ is conjugate to
    $$\left\langle
    \begin{bmatrix} A_{1}-B_{1} & A_{1}+C_{1}-B_{1}-D_{1} \\ B_{1} & B_{1}+D_{1} \end{bmatrix},
    \ldots,
    \begin{bmatrix} A_{s}-B_{s} & A_{s}+C_{s}-B_{s}-D_{s} \\ B_{s} & B_{s}+D_{s} \end{bmatrix}
    \right\rangle.$$
    As $B_{1}$, \ldots, $B_{s}$, $C_{1}$, \ldots, $C_{s}$ are all even, $A_{1}$, \ldots, $A_{s}$, $D_{1}$, \ldots, $D_{s}$ are all odd. By Lemma \ref{lifting with kernel} and Lemma \ref{kernel of pi 1}, $\overline{\rho}_{E,64}(G_{\QQ})$ is conjugate to
    $$\left\langle
    \begin{bmatrix} A_{i}-B_{i} & A_{i}+C_{i}-B_{i}-D_{i} \\ B_{i} & B_{i}+D_{i} \end{bmatrix},
    \begin{bmatrix} 1 + 2^{m} & 0 \\ 0 & 1 \end{bmatrix},
    \begin{bmatrix} 1 & 0 \\ 0 & 1+2^{m} \end{bmatrix},
    \begin{bmatrix} 1 & 2^{m} \\ 0 & 1 \end{bmatrix},
    \begin{bmatrix} 1 & 0 \\ 2^{m} & 1 \end{bmatrix}
    \right\rangle$$
    for $i \in \{1, \ldots, s\}$. By Lemma \ref{2-adic Galois Images}, $\overline{\rho}_{E/\left\langle P_{2}+Q_{2} \right\rangle, 32}(G_{\QQ})$ is conjugate to
    $$\left\langle
    \begin{bmatrix} A_{i}-B_{i} & \frac{A_{i}+C_{i}-B_{i}-D_{i}}{2} \\ 2 \cdot B_{i} & B_{i}+D_{i} \end{bmatrix},
    \begin{bmatrix} 1 + 2^{m} & 0 \\ 0 & 1 \end{bmatrix},
    \begin{bmatrix} 1 & 0 \\ 0 & 1+2^{m} \end{bmatrix},
    \begin{bmatrix} 1 & 2^{m-1} \\ 0 & 1 \end{bmatrix},
    \begin{bmatrix} 1 & 0 \\ 2^{m+1} & 1 \end{bmatrix}
    \right\rangle$$
    for $i \in \{1, \ldots, s\}$
    \end{itemize}
\end{enumerate}
\end{proof}

\begin{corollary}\label{2-adic Galois Images corollary 3}
Let $E/\QQ$ be an elliptic curve that has full two-torsion defined over $\QQ$, generated by $P_{2}$ and $Q_{2}$. Suppose that $E$ has a $\QQ$-rational subgroup generated by $Q_{4}$ and suppose that $C_{2}(E) = 6$. Then the cyclic, $\QQ$-rational subgroups of $E$ are the ones generated by $\mathcal{O}$, $Q_{2}$, $P_{2}$, $P_{2}+Q_{2}$, $Q_{4}$, and $P_{2}+Q_{4}$. Moreover, the level of $\rho_{E,2^{\infty}}(G_{\QQ})$ is equal to $2^{m}$ for some integer $m \geq 2$, and $\overline{\rho}_{E,2^{m}}(G_{\QQ})$ is conjugate to a subgroup of $\operatorname{GL}(2, \ZZ / 2^{m} \ZZ)$ of the form
$$\left\langle
\begin{bmatrix} A_{1} & C_{1} \\ B_{1} & D_{1} \end{bmatrix},
\ldots,
\begin{bmatrix} A_{s} & C_{s} \\ B_{s} & D_{s} \end{bmatrix}
\right\rangle$$
such that $B_{1}$, \ldots, $B_{s}$ are even and $C_{1}$, \ldots, $C_{s}$ are divisible by $4$. Moreover,
\begin{itemize}
    \item $\overline{\rho}_{E / \left\langle P_{2} \right\rangle, 32}(G_{\QQ})$ is conjugate to
    $$\left\langle
\begin{bmatrix} D_{1} & \frac{B_{1}}{2} \\ 2 \cdot C_{1} & A_{1} \end{bmatrix},
\ldots,
\begin{bmatrix} D_{s} & \frac{B_{s}}{2} \\ 2 \cdot C_{s} & A_{s} \end{bmatrix},
\begin{bmatrix} 1+2^{m} & 0 \\ 0 & 1 \end{bmatrix},
\begin{bmatrix} 1 & 0 \\ 0 & 1+2^{m} \end{bmatrix},
\begin{bmatrix} 1 & 2^{m-1} \\ 0 & 1 \end{bmatrix},
\begin{bmatrix} 1 & 0 \\ 2^{m+1} & 1 \end{bmatrix}
\right\rangle$$

\item $\overline{\rho}_{E/\left\langle P_{2}+Q_{2} \right\rangle, 32}(G_{\QQ})$ is conjugate to
$$\left\langle \begin{bmatrix} A_{i}-B_{i} & \frac{A_{i}+C_{i}-B_{i}-D_{i}}{2} \\ 2 \cdot B_{i} & B_{i}+D_{i} \end{bmatrix},
    \begin{bmatrix} 1+2^{m} & 0 \\ 0 & 1 \end{bmatrix},
    \begin{bmatrix} 1 & 0 \\ 0 & 1+2^{m} \end{bmatrix},
    \begin{bmatrix} 1 & 2^{m-1} \\ 0 & 1 \end{bmatrix},
    \begin{bmatrix} 1 & 0 \\ 2^{m+1} & 1 \end{bmatrix}
    \right\rangle$$
    for $i \in \{1, \ldots, s\}$.
\item $\overline{\rho}_{E/\left\langle Q_{2} \right\rangle, 32}(G_{\QQ})$ is conjugate to
$$\left\langle
\begin{bmatrix} A_{1} & \frac{C_{1}}{2} \\ 2 \cdot B_{1} & D_{1} \end{bmatrix},
\ldots,
\begin{bmatrix} A_{s} & \frac{C_{s}}{2} \\ 2 \cdot B_{s} & D_{s} \end{bmatrix},
\begin{bmatrix} 1+2^{m} & 0 \\ 0 & 1 \end{bmatrix},
\begin{bmatrix} 1 & 0 \\ 0 & 1+2^{m} \end{bmatrix},
\begin{bmatrix} 1 & 2^{m-1} \\ 0 & 1 \end{bmatrix},
\begin{bmatrix} 1 & 0 \\ 2^{m+1} & 1 \end{bmatrix}
\right\rangle.$$
\item $\overline{\rho}_{E/\left\langle Q_{4} \right\rangle, 32}(G_{\QQ})$ is conjugate to
$$\left\langle
\begin{bmatrix} A_{1} & \frac{C_{1}}{4} \\ 4 \cdot B_{1} & D_{1} \end{bmatrix},
\ldots,
\begin{bmatrix} A_{s} & \frac{C_{s}}{4} \\ 4 \cdot B_{s} & D_{s} \end{bmatrix},
\begin{bmatrix} 1+2^{m} & 0 \\ 0 & 1 \end{bmatrix},
\begin{bmatrix} 1 & 0 \\ 0 & 1+2^{m} \end{bmatrix},
\begin{bmatrix} 1 & 2^{m-2} \\ 0 & 1 \end{bmatrix},
\begin{bmatrix} 1 & 0 \\ 2^{m+2} & 1 \end{bmatrix}
\right\rangle.$$
\item $\overline{\rho}_{E/\left\langle P_{2}+Q_{4} \right\rangle, 32}(G_{\QQ})$ is conjugate to
$$\left\langle
\begin{bmatrix} A_{i}-2B_{i} & \frac{2A_{i}+C_{i}-4B_{i}-2D_{i}}{4} \\ 4 \cdot B_{i} & 2B_{i}+D_{i} \end{bmatrix},
\begin{bmatrix} 1+2^{m} & 0 \\ 0 & 1 \end{bmatrix},
\begin{bmatrix} 1 & 0 \\ 0 & 1+2^{m} \end{bmatrix},
\begin{bmatrix} 1 & 2^{m-2} \\ 0 & 1 \end{bmatrix},
\begin{bmatrix} 1 & 0 \\ 2^{m+2} & 1 \end{bmatrix}
\right\rangle$$
for $i \in \{1, \ldots, s\}$.
    \end{itemize}
\end{corollary}

\begin{proof}
In our given situation, the isogeny graph associated to the $\QQ$-isogeny class of $E$ is of $T_{6}$ type.
\begin{center}
    \begin{tikzcd}
E/\left\langle P_{2} + Q_{2} \right\rangle &                                                     &  &                                                                    & E/\left\langle Q_{4} \right\rangle       \\
                                           & E \arrow[lu, no head, "2"'] \arrow[ld, no head, "2"'] \arrow[rr, no head, "2"] &  & E/\left\langle Q_{2} \right\rangle \arrow[ru, no head, "2"] \arrow[rd, no head, "2"] &                                          \\
E/\left\langle P_{2} \right\rangle         &                                                     &  &                                                                    & E/\left\langle P_{2}+Q_{4} \right\rangle
\end{tikzcd}
\end{center}
Note that $E$ has full two-torsion defined over $\QQ$ and a cyclic, $\QQ$-rational subgroup of order $4$. With the basis $E[2^{m}] = \left\langle P_{2^{m}}, Q_{2^{m}} \right\rangle$, by Lemma \ref{Borel 1}, $\overline{\rho}_{E,2^{m}}(G_{\QQ})$ is conjugate to a subgroup of $\operatorname{GL}(2, \ZZ / 2^{m} \ZZ)$ of the form
$$\left\langle
\begin{bmatrix} A_{1} & C_{1} \\ B_{1} & D_{1} \end{bmatrix},
\ldots,
\begin{bmatrix} A_{s} & C_{s} \\ B_{s} & D_{s} \end{bmatrix}
\right\rangle$$
such that $B_{1}$, \ldots, $B_{s}$ are even and $C_{1}$, \ldots, $C_{s}$ are divisible by $4$. The determination of $\overline{\rho}_{E/\left\langle P_{2} \right\rangle, 32}(G_{\QQ})$, $\overline{\rho}_{E/\left\langle P_{2} + Q_{2} \right\rangle, 32}(G_{\QQ})$, and $\overline{\rho}_{E/\left\langle Q_{2} \right\rangle, 32}(G_{\QQ})$, are identical to the cases dealt with in Corollary \ref{2-adic Galois Images corollary 2}. We move on to the case of $\overline{\rho}_{E / \left\langle Q_{4} \right\rangle, 32}(G_{\QQ})$ and $\overline{\rho}_{E / \left\langle P_{2}+Q_{4} \right\rangle, 32}(G_{\QQ})$.

\begin{itemize}
    \item Using the basis $E[2^{m}] = \left\langle P_{2^{m}}, Q_{2^{m}} \right\rangle$, by Lemma \ref{lifting with kernel} and Lemma \ref{kernel of pi 1}, $\overline{\rho}_{E,128}(G_{\QQ})$ is conjugate to
    $$\left\langle
    \begin{bmatrix} A_{1} & C_{1} \\ B_{1} & D_{1} \end{bmatrix},
    \ldots
    \begin{bmatrix} A_{s} & C_{s} \\ B_{s} & D_{s} \end{bmatrix},
    \begin{bmatrix} 1 + 2^{m} & 0 \\ 0 & 1 \end{bmatrix},
    \begin{bmatrix} 1 & 0 \\ 0 & 1 + 2^{m} \end{bmatrix},
    \begin{bmatrix} 1 & 2^{m} \\ 0 & 1 \end{bmatrix},
    \begin{bmatrix} 1 & 0 \\ 2^{m} & 1 \end{bmatrix}
    \right\rangle.$$
    By Lemma \ref{2-adic Galois Images}, $\overline{\rho}_{E/\left\langle Q_{4} \right\rangle, 32}(G_{\QQ})$ is conjugate to
    $$\left\langle
    \begin{bmatrix} A_{1} & \frac{C_{1}}{4} \\ 4 \cdot B_{1} & D_{1} \end{bmatrix},
    \ldots
    \begin{bmatrix} A_{s} & \frac{C_{s}}{4} \\ 4 \cdot B_{s} & D_{s} \end{bmatrix},
    \begin{bmatrix} 1 + 2^{m} & 0 \\ 0 & 1 \end{bmatrix},
    \begin{bmatrix} 1 & 0 \\ 0 & 1 + 2^{m} \end{bmatrix},
    \begin{bmatrix} 1 & 2^{m-2} \\ 0 & 1 \end{bmatrix},
    \begin{bmatrix} 1 & 0 \\ 2^{m+2} & 1 \end{bmatrix}
    \right\rangle.$$
    
    \item Using the basis $E[2^{m}] = \left\langle P_{2^{m}}, P_{2^{m-1}}+Q_{2^{m}} \right\rangle$, $\overline{\rho}_{E,2^{m}}(G_{\QQ})$ is conjugate to
    $$\left\langle
    \begin{bmatrix}
    A_{1}-2B_{1} & 2A_{1}+C_{1}-4B_{1}-2D_{1} \\ B_{1} & 2B_{1} + D_{1}
    \end{bmatrix},
    \ldots,
    \begin{bmatrix}
    A_{s}-2B_{s} & 2A_{s}+C_{s}-4B_{s}-2D_{s} \\ B_{s} & 2B_{s} + D_{s}
    \end{bmatrix}
    \right\rangle.$$
    By Lemma \ref{lifting with kernel} and Lemma \ref{kernel of pi 1}, $\overline{\rho}_{E,128}(G_{\QQ})$ is conjugate to
    $$\left\langle
    \begin{bmatrix}
    A_{i}-2B_{i} & 2A_{i}+C_{i}-4B_{i}-2D_{i} \\ B_{i} & 2B_{i} + D_{i}
    \end{bmatrix},
    \begin{bmatrix}
    1 + 2^{m} & 0 \\ 0 & 1
    \end{bmatrix},
    \begin{bmatrix}
    1 & 0 \\ 0 & 1+ 2^{m}
    \end{bmatrix},
    \begin{bmatrix}
    1 & 2^{m} \\ 0 & 1
    \end{bmatrix},
    \begin{bmatrix}
    1 & 0 \\ 2^{m} & 1
    \end{bmatrix}
    \right\rangle$$
    for $i \in \{1, \ldots, s \}$. Note that for $i \in \{1, \ldots, s\}$, $B_{i}$ is even and $C_{i}$ is divisible by $4$. This forces $A_{i}$ and $D_{i}$ to be odd and $2A_{i}+C_{i}-4B_{i}-2D_{i}$ to be divisible by $4$. By Lemma \ref{2-adic Galois Images}, $\overline{\rho}_{E/\left\langle P_{2}+Q_{4} \right\rangle,32}(G_{\QQ})$ is conjugate to
    $$\left\langle
    \begin{bmatrix}
    A_{i}-2B_{i} & \frac{2A_{i}+C_{i}-4B_{i}-2D_{i}}{4} \\ 4 \cdot B_{i} & 2B_{i} + D_{i}
    \end{bmatrix},
    \begin{bmatrix}
    1 + 2^{m} & 0 \\ 0 & 1
    \end{bmatrix},
    \begin{bmatrix}
    1 & 0 \\ 0 & 1+ 2^{m}
    \end{bmatrix},
    \begin{bmatrix}
    1 & 2^{m-2} \\ 0 & 1
    \end{bmatrix},
    \begin{bmatrix}
    1 & 0 \\ 2^{m+2} & 1
    \end{bmatrix}
    \right\rangle$$
    for $i \in \{1, \ldots, s\}$.
\end{itemize}

\end{proof}

\begin{corollary}\label{2-adic Galois Images corollary 4}
Let $E/\QQ$ be an elliptic curve such that $P_{4}$ and $Q_{4}$ are points on $E$ of order $4$ that generate $\QQ$-rational groups. Moreover, suppose that $\left\langle P_{4} \right\rangle \cap \left\langle Q_{4} \right\rangle = \{\mathcal{O}\}$. Then $E$ has full two-torsion defined over $\QQ$ and the cyclic, $\QQ$-rational subgroups of $E$ are generated by $\mathcal{O}$, $P_{2}$, $Q_{2}$, $P_{2}+Q_{2}$, $Q_{4}$, $P_{2}+Q_{4}$, $P_{4}$, and $P_{4}+Q_{2}$. The level of $\rho_{E,2^{\infty}}(G_{\QQ})$ is equal to $2^{m}$ for some integer $m \geq 2$ and $\overline{\rho}_{E,2^{m}}(G_{\QQ})$ is conjugate to a subgroup of $\operatorname{GL}(2, \ZZ / 2^{m} \ZZ)$ of the form
$$\left\langle
\begin{bmatrix} A_{1} & C_{1} \\ B_{1} & D_{1} \end{bmatrix},
\ldots,
\begin{bmatrix} A_{s} & C_{s} \\ B_{s} & D_{s} \end{bmatrix}
\right\rangle$$
such that $A_{1}$, \ldots, $A_{s}$, $D_{1}$, \ldots, $D_{s}$ are odd and $B_{1}$, \ldots, $B_{s}$, $C_{1}$, \ldots, $C_{s}$ are divisible by $4$. Finally, for $i \in \{1, \ldots, s\}$
\begin{itemize}
    \item $\overline{\rho}_{E / \left\langle P_{2} \right\rangle, 32}(G_{\QQ})$ is conjugate to
    $$\left\langle
\begin{bmatrix} D_{i} & \frac{B_{i}}{2} \\ 2 \cdot C_{i} & A_{i} \end{bmatrix},
\begin{bmatrix} 1+2^{m} & 0 \\ 0 & 1 \end{bmatrix},
\begin{bmatrix} 1 & 0 \\ 0 & 1+2^{m} \end{bmatrix},
\begin{bmatrix} 1 & 2^{m-1} \\ 0 & 1 \end{bmatrix},
\begin{bmatrix} 1 & 0 \\ 2^{m+1} & 1 \end{bmatrix}
\right\rangle$$
for $i \in \left\{1, \ldots, s\right\}$.
\item $\overline{\rho}_{E/\left\langle P_{2}+Q_{2} \right\rangle, 32}(G_{\QQ})$ is conjugate to
$$\left\langle \begin{bmatrix} A_{i}-B_{i} & \frac{A_{i}+C_{i}-B_{i}-D_{i}}{2} \\ 2 \cdot B_{i} & B_{i}+D_{i} \end{bmatrix},
    \begin{bmatrix} 1+2^{m} & 0 \\ 0 & 1 \end{bmatrix},
    \begin{bmatrix} 1 & 0 \\ 0 & 1+2^{m} \end{bmatrix},
    \begin{bmatrix} 1 & 2^{m-1} \\ 0 & 1 \end{bmatrix},
    \begin{bmatrix} 1 & 0 \\ 2^{m+1} & 1 \end{bmatrix}
    \right\rangle$$
    for $i \in \left\{1, \ldots, s\right\}$.
\item $\overline{\rho}_{E/\left\langle Q_{2} \right\rangle, 32}(G_{\QQ})$ is conjugate to
$$\left\langle
\begin{bmatrix} A_{i} & \frac{C_{i}}{2} \\ 2 \cdot B_{i} & D_{i} \end{bmatrix},
\begin{bmatrix} 1+2^{m} & 0 \\ 0 & 1 \end{bmatrix},
\begin{bmatrix} 1 & 0 \\ 0 & 1+2^{m} \end{bmatrix},
\begin{bmatrix} 1 & 2^{m-1} \\ 0 & 1 \end{bmatrix},
\begin{bmatrix} 1 & 0 \\ 2^{m+1} & 1 \end{bmatrix}
\right\rangle$$
for $i \in \left\{1, \ldots, s\right\}$.
\item $\overline{\rho}_{E/\left\langle Q_{4} \right\rangle, 32}(G_{\QQ})$ is conjugate to
$$\left\langle
\begin{bmatrix} A_{i} & \frac{C_{i}}{4} \\ 4 \cdot B_{i} & D_{i} \end{bmatrix},
\begin{bmatrix} 1+2^{m} & 0 \\ 0 & 1 \end{bmatrix},
\begin{bmatrix} 1 & 0 \\ 0 & 1+2^{m} \end{bmatrix},
\begin{bmatrix} 1 & 2^{m-2} \\ 0 & 1 \end{bmatrix},
\begin{bmatrix} 1 & 0 \\ 2^{m+2} & 1 \end{bmatrix}
\right\rangle$$
for $i \in \left\{1, \ldots, s\right\}$.
\item $\overline{\rho}_{E/\left\langle P_{2}+Q_{4} \right\rangle, 32}(G_{\QQ})$ is conjugate to
$$\left\langle
\begin{bmatrix} A_{i}-2B_{i} & \frac{2A_{i}+C_{i}-4B_{i}-2D_{i}}{4} \\ 4 \cdot B_{i} & 2B_{i}+D_{i} \end{bmatrix},
\begin{bmatrix} 1+2^{m} & 0 \\ 0 & 1 \end{bmatrix},
\begin{bmatrix} 1 & 0 \\ 0 & 1+2^{m} \end{bmatrix},
\begin{bmatrix} 1 & 2^{m-2} \\ 0 & 1 \end{bmatrix},
\begin{bmatrix} 1 & 0 \\ 2^{m+2} & 1 \end{bmatrix}
\right\rangle$$
for $i \in \left\{1, \ldots, s\right\}$.
\item $\overline{\rho}_{E/\left\langle P_{4} \right\rangle, 32}(G_{\QQ})$ is conjugate to
$$\left\langle
\begin{bmatrix} D_{i} & \frac{B_{i}}{4} \\ 4 \cdot C_{i} & A_{i} \end{bmatrix},
\begin{bmatrix} 1+2^{m} & 0 \\ 0 & 1 \end{bmatrix},
\begin{bmatrix} 1 & 0 \\ 0 & 1+2^{m} \end{bmatrix},
\begin{bmatrix} 1 & 2^{m-2} \\ 0 & 1 \end{bmatrix},
\begin{bmatrix} 1 & 0 \\ 2^{m+2} & 1 \end{bmatrix}
\right\rangle$$
for $i \in \left\{1, \ldots, s\right\}$.
\item $\overline{\rho}_{E/\left\langle P_{4}+Q_{2} \right\rangle, 32}(G_{\QQ})$ is conjugate to
$$\left\langle
\begin{bmatrix} 2B_{i} + D_{i} & \frac{B_{i}}{4} \\ 4 \cdot (2A_{i} + C_{i} - 4B_{i} - 2D_{i}) & A_{i}-2B_{i} \end{bmatrix},
\begin{bmatrix} 1+2^{m} & 0 \\ 0 & 1 \end{bmatrix},
\begin{bmatrix} 1 & 0 \\ 0 & 1+2^{m} \end{bmatrix},
\begin{bmatrix} 1 & 2^{m-2} \\ 0 & 1 \end{bmatrix},
\begin{bmatrix} 1 & 0 \\ 2^{m+2} & 1 \end{bmatrix}
\right\rangle$$
for $i \in \left\{1, \ldots, s\right\}$.
\end{itemize}
\end{corollary}

\begin{proof}
In our case, the isogeny graph associated to the $\QQ$-isogeny class of $E$ is of $T_{8}$ type
\begin{center}
    \begin{tikzcd}
                                             & E / \left\langle P_{4} \right\rangle                               &                                                   & E / \left\langle Q_{4} \right\rangle                              &                                              \\
                                             & E/\left\langle P_{2} \right\rangle \arrow[u, no head, "2"] \arrow[ld, no head, "2"'] &                                                   & E/\left\langle Q_{2} \right\rangle \arrow[u, no head, "2"] \arrow[rd, no head, "2"] &                                              \\
E / \left\langle P_{4} + Q_{2} \right\rangle &                                                                    & E \arrow[lu, no head, "2"'] \arrow[d, no head, "2"] \arrow[ru, no head, "2"] &                                                                   & E / \left\langle Q_{4} + P_{2} \right\rangle \\
                                             &                                                                    & E/\left\langle Q_{2} \right\rangle                &                                                                   &                                             
\end{tikzcd}
\end{center}The elliptic curve $E$ has two independent, cyclic, $\QQ$-rational subgroups of order $4$. Then the level of $\rho_{E,2^{\infty}}(G_{\QQ})$ is equal to $2^{m}$ where $m$ is an integer greater than or equal to $2$. With the basis $E[2^{m}] = \left\langle P_{2^{m}}, Q_{2^{m}} \right\rangle$, by Lemma \ref{Borel 1}, $\overline{\rho}_{E,2^{m}}(G_{\QQ})$ is conjugate to a subgroup of $\operatorname{GL}(2, \ZZ / 2^{m} \ZZ)$ of the form
$$\left\langle
\begin{bmatrix}
A_{1} & C_{1} \\ B_{1} & D_{1}
\end{bmatrix},
\ldots,
\begin{bmatrix}
A_{s} & C_{s} \\ B_{s} & D_{s}
\end{bmatrix}
\right\rangle$$
such that $C_{1}$, \ldots, $C_{s}$ are divisible by $4$. With the basis $E[4] = \left\langle Q_{2^{m}}, P_{2^{m}} \right\rangle$, $\overline{\rho}_{E,2^{m}}(G_{\QQ})$ is conjugate to
$$\left\langle
\begin{bmatrix}
D_{1} & B_{1} \\ C_{1} & A_{1}
\end{bmatrix},
\ldots,
\begin{bmatrix}
D_{s} & B_{s} \\ C_{s} & A_{s}
\end{bmatrix}
\right\rangle$$
and by Lemma \ref{Borel 1}, $B_{1}$, \ldots, $B_{s}$ are divisible by $4$. Moreover, $A_{1}, \ldots, A_{s}, D_{1}, \ldots, D_{s}$ are odd. The proofs to classify the remaining groups, $\overline{\rho}_{E/\left\langle g \right\rangle,32}(G_{\QQ})$, where $g \in \left\{P_{2}, Q_{2}, P_{2}+Q_{2}, Q_{4}, P_{2}+Q_{4}\right\}$, are similar to the proofs in Corollary \ref{2-adic Galois Images corollary 3}.

We move on to the case of $\overline{\rho}_{E/\left\langle P_{4} \right\rangle,32}(G_{\QQ})$ and $\overline{\rho}_{E/\left\langle P_{4}+Q_{2} \right\rangle,32}(G_{\QQ})$. Note that with the basis $E[2^{m}] = \left\langle P_{2^{m}}, Q_{2^{m}} \right\rangle$, $\overline{\rho}_{E,2^{m}}(G_{\QQ})$ is conjugate to
$$\left\langle
\begin{bmatrix}
A_{1} & C_{1} \\ B_{1} & D_{1}
\end{bmatrix},
\ldots,
\begin{bmatrix}
A_{s} & C_{s} \\ B_{s} & D_{s}
\end{bmatrix}
\right\rangle.$$
Then switching the basis to $E[2^{m}] = \left\langle Q_{2^{m}}, P_{2^{m}} \right\rangle$, we have $\overline{\rho}_{E,2^{m}}(G_{\QQ})$ is conjugate to
$$\left\langle
\begin{bmatrix}
D_{1} & B_{1} \\ C_{1} & A_{1}
\end{bmatrix},
\ldots,
\begin{bmatrix}
D_{s} & B_{s} \\ C_{s} & A_{s}
\end{bmatrix}
\right\rangle.$$
By Lemma \ref{lifting with kernel} and Lemma \ref{kernel of pi 1}, $\overline{\rho}_{E,128}(G_{\QQ})$ is conjugate to
$$\left\langle
\begin{bmatrix}
D_{i} & B_{i} \\ C_{i} & A_{i}
\end{bmatrix},
\begin{bmatrix}
1+2^{m} & 0 \\ 0 & 1
\end{bmatrix},
\begin{bmatrix}
1 & 0 \\ 0 & 1+2^{m}
\end{bmatrix},
\begin{bmatrix}
1 & 2^{m} \\ 0 & 1
\end{bmatrix},
\begin{bmatrix}
1 & 0 \\ 2^{m} & 1
\end{bmatrix}
\right\rangle$$
for $i \in \{1, \ldots, s \}$. By Lemma \ref{2-adic Galois Images}, $\rho_{E / \left\langle P_{4} \right\rangle , 32}(G_{\QQ})$ is conjugate to
$$\left\langle
\begin{bmatrix}
D_{i} & \frac{B_{i}}{4} \\ 4 \cdot C_{i} & A_{i}
\end{bmatrix},
\begin{bmatrix}
1+2^{m} & 0 \\ 0 & 1
\end{bmatrix},
\begin{bmatrix}
1 & 0 \\ 0 & 1+2^{m}
\end{bmatrix},
\begin{bmatrix}
1 & 2^{m-2} \\ 0 & 1
\end{bmatrix},
\begin{bmatrix}
1 & 0 \\ 2^{m+2} & 1
\end{bmatrix}
\right\rangle$$
for $i \in \{1, \ldots, s\}$.

Finally, with the basis $E[2^{m}] = \left\langle P_{2^{m}}+Q_{2^{m-1}}, Q_{2^{m}} \right\rangle$, $\overline{\rho}_{E,2^{m}}(G_{\QQ})$ is conjugate to
$$\left\langle
\begin{bmatrix}
A_{1}+2C_{1} & C_{1} \\ B_{1}+2D_{1}-2A_{1}-4C_{1} & D_{1}-2C_{1}
\end{bmatrix},
\ldots,
\begin{bmatrix}
A_{s}+2C_{s} & C_{s} \\ B_{s}+2D_{s}-2A_{s}-4C_{s} & D_{s}-2C_{s}
\end{bmatrix}
\right\rangle.$$
By Lemma \ref{lifting with kernel} and Lemma \ref{kernel of pi 1}, $\overline{\rho}_{E,128}(G_{\QQ})$ is conjugate to
$$\left\langle
\begin{bmatrix}
A_{i}+2C_{i} & C_{i} \\ B_{i}+2D_{i}-2A_{i}-4C_{i} & D_{i}-2C_{i}
\end{bmatrix},
\begin{bmatrix}
1+2^{m} & 0 \\ 0 & 1
\end{bmatrix},
\begin{bmatrix}
1 & 0 \\ 0 & 1+2^{m}
\end{bmatrix},
\begin{bmatrix}
1 & 2^{m} \\ 0 & 1
\end{bmatrix},
\begin{bmatrix}
1 & 0 \\ 2^{m} & 1
\end{bmatrix}
\right\rangle$$
for $i \in \{1, \ldots, s\}$.
By Lemma \ref{2-adic Galois Images}, $\overline{\rho}_{E/\left\langle P_{4}+Q_{2} \right\rangle,32}(G_{\QQ})$ is conjugate to
$$\left\langle
\begin{bmatrix}
A_{i}+2C_{i} & \frac{C_{i}}{4} \\ 4 \cdot (B_{i}+2D_{i}-2A_{i}-4C_{i}) & D_{i}-2C_{i}
\end{bmatrix},
\begin{bmatrix}
1+2^{m} & 0 \\ 0 & 1
\end{bmatrix},
\begin{bmatrix}
1 & 0 \\ 0 & 1+2^{m}
\end{bmatrix},
\begin{bmatrix}
1 & 2^{m-2} \\ 0 & 1
\end{bmatrix},
\begin{bmatrix}
1 & 0 \\ 2^{m+2} & 1
\end{bmatrix}
\right\rangle$$
for $i \in \{1, \ldots, s\}$.
\end{proof}
\begin{remark}
Let $\mathcal{G}$ be an isogeny graph associated to the $\QQ$-isogeny class of elliptic curves defined over $\QQ$ and let $\mathcal{G}_{2^{\infty}}$ be the $2$-adic subgraph of $\mathcal{G}$; the subgraph of $\mathcal{G}$ where all the edges are $2$-power degree. Suppose that $\rho_{E,2^{\infty}}(G_{\QQ})$ is given where $E$ is one of the ideal elliptic curves in $\mathcal{G}$; in the case where $\mathcal{G}_{2^{\infty}}$ is of $L_{1}(1)$ or $L_{2}(2)$ type, $E$ is any elliptic curve, in the case where $\mathcal{G}_{2^{\infty}}$ is of $T_{4}$ or $T_{6}$ type, then $E$ is any elliptic curve with full two-torsion defined over $\QQ$ and in the case that $\mathcal{G}_{2^{\infty}}$ is of $T_{8}$ type, then $E$ is the elliptic curve with two independent, cyclic, $\QQ$-rational subgroups of order $4$.

Using Corollary \ref{2-adic Galois Images corollary 1} - Corollary \ref{2-adic Galois Images corollary 4}, one can compute the $2$-adic Galois Image of all elliptic curves over $\QQ$ in $\mathcal{G}_{2^{\infty}}$ and then use Corollary \ref{coprime isogeny-degree} to compute the $2$-adic Galois Image of all elliptic curves over $\QQ$ in $\mathcal{G}$. One does have to be careful that the RZB database determines the $2$-adic Galois Image using right actions. In other words, if $\rho_{E,2^{\infty}}(G_{\QQ})$ is said to be conjugate to a group $\operatorname{H}$ of level $2^{m}$ in the RZB database, then actually, $\rho_{E,2^{\infty}}(G_{\QQ})$ is conjugate to the transpose of $H$ using left actions.

Many, but not all examples of $2$-adic Galois Images appear in the LMFDB. The ones that do not appear in the LMFDB were computed by ``hand''; the image of the mod-$32$ Galois representation was computed for each vertex and then cross referenced in the RZB database.
\end{remark}

\section{Product groups}

A product group is a subgroup of $\operatorname{GL}(2, \ZZ / MN \ZZ) \cong \operatorname{GL}(2, \ZZ / M \ZZ) \times \operatorname{GL}(2, \ZZ / N \ZZ)$ for some integers $M , N \geq 2$ such that $\operatorname{gcd}(M,N) = 1$.

\begin{lemma}\label{product groups}

Let $M$ be a positive integer and let $N$ be an odd integer greater than or equal to $3$. Let $A$ be a subgroup of $\operatorname{GL}\left(2, \ZZ / N \ZZ\right)$ and let $B$ be a subgroup of $\operatorname{GL}\left(2, \ZZ / 2^{M} \ZZ\right)$. Denote the respective natural canonical maps as $\pi_{1} \colon \operatorname{GL}\left(2, \ZZ / 2^{M}N \ZZ\right) \to \operatorname{GL}(2, \ZZ / N \ZZ)$ and $\pi_{2} \colon \operatorname{GL}(2, \ZZ / 2^{M}N \ZZ) \to \operatorname{GL}(\ZZ / 2^{M} \ZZ)$. Then the group of matrices in $\operatorname{GL}(2, \ZZ / 2^{M}N\ZZ)$ that reduces modulo $N$ to $A$ and reduces modulo $2^{M}$ to $B$ is conjugate to $\left\langle \widehat{A}, \operatorname{Ker}(\pi_{1}) \right\rangle \bigcap \left\langle \widehat{B}, \operatorname{Ker}(\pi_{2}) \right\rangle$, where $\widehat{A}$ is any subgroup of $\operatorname{GL}(2, \ZZ / 2^{M} N \ZZ)$ that reduces modulo $N$ to $A$ and $\widehat{B}$ is any subgroup of $\operatorname{GL}(2, \ZZ / 2^{M} N \ZZ)$ that reduces modulo $2^{M}$ to $B$.

\end{lemma}

\begin{proof}
Let $H$ be a subgroup of $\operatorname{GL}\left(2, 2^{M}N \ZZ\right)$ that reduces modulo $N$ to $A$. Using a similar proof to Lemma \ref{lifting with kernel}, we see that $H$ is conjugate to a subgroup of $\left\langle \widehat{A}, \operatorname{Ker}(\pi_{1}) \right\rangle$ where $\widehat{A}$ is any lift of $A$ to level $2^{M}N$. Also, $H$ reduces modulo $2^{M}$ to $B$. Similarly, we see that $H$ is conjugate to a subgroup of $\left\langle \widehat{B}, \operatorname{Ker}(\pi_{2}) \right\rangle$ where $\widehat{B}$ is any lift of $B$ to level $2^{M}N$. Hence, $H$ is conjugate to a subgroup of $\left\langle \widehat{A}, \operatorname{Ker}(\pi_{1}) \right\rangle \bigcap \left\langle \widehat{B}, \operatorname{Ker}(\pi_{2}) \right\rangle$. Conversely, a matrix in $\left\langle \widehat{A}, \operatorname{Ker}(\pi) \right\rangle \bigcap \left\langle \widehat{B}, \operatorname{Ker}(\pi_{2}) \right\rangle$ reduces modulo $N$ to a matrix in $A$ and simultaneously reduces modulo $2^{M}$ to a matrix in $B$.
\end{proof}

\section{Galois images and \textit{j}-invariants}

Let $N$ be a positive integer and let $\operatorname{H}$ be a subgroup of $\operatorname{GL}(2, \ZZ / N \ZZ)$ that contains $\operatorname{-Id}$ and such that $\operatorname{det}(\operatorname{H}) = \left(\ZZ / N \ZZ\right)^{\times}$. Then there is a modular curve $\operatorname{X}_{\operatorname{H}}$ defined over $\QQ$, generated by $\operatorname{H}$. There is a non-constant morphism $\pi_{\operatorname{H}} \colon \operatorname{X}_{\operatorname{H}} \to \mathbb{P}^{1}(\QQ)$ of degree $\left[\operatorname{GL}(\ZZ / N \ZZ) : \operatorname{H}\right]$. Let $E/\QQ$ be an elliptic curve. Then $\overline{\rho}_{E,N}(G_{\QQ})$ is conjugate to a subgroup of $\operatorname{H}$ if and only if $\textit{j}_{E} \in \pi_{\operatorname{H}}(\operatorname{X}_{\operatorname{H}}(\QQ))$. Suppose that $H'$ is a subgroup of $\operatorname{GL}\left(2, \ZZ / N \ZZ\right)$ that contains $H$. Then there is a rational morphism $\phi \colon \operatorname{X}_{\operatorname{H}} \to \operatorname{X}_{\operatorname{H}'}$ of degree $\left[\operatorname{H}' : \operatorname{H} \right]$ that fits into the following commutative diagram
$$\begin{tikzcd}
\operatorname{X}_{\operatorname{H}} \arrow[rrd, "\operatorname{\pi}_{\operatorname{H}}"'] \arrow[rr, "\phi"] &  & \operatorname{X}_{\operatorname{H'}} \arrow[d, "\pi_{\operatorname{H'}}"] \\
                                                                                                             &  & \mathbb{P}^{1}(\mathbb{Q})                                               
\end{tikzcd}.$$
A point $P$ on $\operatorname{X}_{\operatorname{H}}$ is a cusp if $\pi_{\operatorname{H}}(P) = \pi_{\operatorname{H}'}(\phi(P))$ is the point at infinity. For more information, see Section 2 of \cite{SZ}.

\begin{lemma}\label{Galois images}

Let $N$ be a positive integer and let $\operatorname{H}$ be a subgroup of $\operatorname{GL}(2, \ZZ / N \ZZ)$ such that $\operatorname{H}$ contains $\operatorname{-Id}$ and $\operatorname{det}(\operatorname{H}) = \left(\ZZ / N \ZZ\right)^{\times}$. Let $\operatorname{X}_{\operatorname{H}}$ be the modular curve generated by $\operatorname{H}$. Let $\operatorname{H'}$ be a subgroup of $\operatorname{GL}(2, \ZZ / N \ZZ)$ that contains $\operatorname{H}$ and let $\operatorname{X}_{\operatorname{H}'}$ be the modular curve generated by $\operatorname{H}'$. If all of the rational points on $X_{\operatorname{H}'}$ are cusps or CM points, then all of the rational points on $X_{\operatorname{H}}$ are cusps or CM points.

\end{lemma}

\begin{proof}
Let $\phi \colon \operatorname{X}_{\operatorname{H}} \to \operatorname{X}_{\operatorname{H}'}$ be the rational morphism such that
$$\begin{tikzcd}
\operatorname{X}_{\operatorname{H}} \arrow[rrd, "\operatorname{\pi}_{\operatorname{H}}"'] \arrow[rr, "\phi"] &  & \operatorname{X}_{\operatorname{H'}} \arrow[d, "\pi_{\operatorname{H'}}"] \\
                                                                                                             &  & \mathbb{P}^{1}(\mathbb{Q})                                               
\end{tikzcd}$$
is commutative. Let $P$ be a rational point on $\operatorname{X}_{\operatorname{H}}$. Then $\phi(P)$ is a rational point on $\operatorname{X}_{\operatorname{H}'}$. Hence, either $\pi_{\operatorname{H}'}(\phi(P)) = \pi_{\operatorname{H}}(P) = \infty$ or is a CM \textit{j}-invariant.
\end{proof}

\section{Classification of $2$-adic Galois Images of isogeny-torsion graphs}\label{Determination of 2-adic Galois Images}

Here we classify the $2$-adic Galois Images of some isogeny-torsion graphs over $\QQ$ while at times leaving the full proof of the classification of other isogeny-torsion graphs to Sections \ref{elliptic curves} and \ref{hyperelliptic curves}.

\subsection{Isogeny-torsion graphs of $L_{1}$, $L_{2}(2)$, $T_{4}$, $T_{6}$, and $T_{8}$ type}
\begin{proposition}
Let $\mathcal{G}$ be an isogeny-torsion graph associated to a $\QQ$-isogeny class of elliptic curves over $\QQ$ without CM.
\begin{itemize}
    \item If $\mathcal{G}$ is of $L_{1}$ type, then the $2$-adic Galois image of $\mathcal{G}$ is one of the $22$ entries in Table \ref{2-adic Galois Images of L1 Graphs},
    
    \item if $\mathcal{G}$ is of $L_{2}(2)$ type, then the $2$-adic Galois image of $\mathcal{G}$ is one of the $80$ entries in Table \ref{2-adic Galois Images of L22 Graphs},
    
    \item if $\mathcal{G}$ is of $T_{4}$ type, then the $2$-adic Galois image of $\mathcal{G}$ is one of the $60$ entries in Table \ref{2-adic Galois Images of T4 Graphs},
    
    \item if $\mathcal{G}$ is of $T_{6}$ type, then the $2$-adic Galois image of $\mathcal{G}$ is one of the $81$ entries in Table \ref{2-adic Galois Images of T6 Graphs},
    
    \item if $\mathcal{G}$ is of $T_{8}$ type, then the $2$-adic Galois image of $\mathcal{G}$ is one of the $53$ entries in Table \ref{2-adic Galois Images of T8 Graphs}.
\end{itemize}
\end{proposition}

\begin{proof}
If $\mathcal{G}$ is of $L_{1}$ type, all that remains is to look up the groups in the RZB database that reduce modulo $2$ to a group of order $3$ or $6$. If $\mathcal{G}$ is of $L_{2}(2)$ type, we can apply Corollary \ref{2-adic Galois Images corollary 1}, if $\mathcal{G}$ is of $T_{4}$ type, we can apply Corollary \ref{2-adic Galois Images corollary 2}, if $\mathcal{G}$ is of $T_{6}$ type, we can apply Corollary \ref{2-adic Galois Images corollary 3}, and if $\mathcal{G}$ is of $T_{8}$ type, we can apply Corollary \ref{2-adic Galois Images corollary 4}. The rest of the proof requires going through all entries in the RZB database (which the author insists he has done at least three times).
\end{proof}

\subsection{Isogeny-torsion graphs of $R_{4}$ type}

\begin{proposition}
Let $E/\QQ$ be an elliptic curve such that the isogeny-torsion graph associated to the $\QQ$-isogeny class of $E$ is of $R_{4}(15)$ type or $R_{4}(21)$ type. Then $\rho_{E,2^{\infty}}(G_{\QQ})$ is conjugate to the full lift of $\operatorname{H}_{4}$.
\end{proposition}

\begin{proof}
Let $\mathcal{E}_{21}$ be the $\QQ$-isogeny class with LMFDB notation \texttt{162.b}. Then the isogeny graph associated to $\mathcal{E}_{21}$ is of $R_{4}(21)$ type. The isogeny graph along with the respective \textit{j}-invariants are below
\begin{center} \begin{tikzcd}
{E_{1}, \frac{-1159088625}{2097152}} \arrow[rr, no head, "3"]       &  & {E_{2}, \frac{-189613868625}{128}} \arrow[d, no head, "7"] \\
{E_{3}, \frac{-140625}{8}} \arrow[u, no head, "7"] \arrow[rr, no head, "3"'] &  & {E_{4}, \frac{3375}{2}}                          
\end{tikzcd} \end{center}
Let $\mathcal{E}_{15}$ be the $\QQ$-isogeny class with LMFDB notation \texttt{50.a}. Then the isogeny graph associated to $\mathcal{E}_{15}$ is of $R_{4}(15)$ type. The isogeny graph along with the respective \textit{j}-invariants are below
\begin{center} \begin{tikzcd}
{E_{1}, \frac{-25}{2}} \arrow[rr, no head, "3"]       &  & {E_{2}, \frac{-349938025}{8}} \arrow[d, no head, "5"] \\
{E_{3}, \frac{-121945}{32}} \arrow[u, no head, "5"] \arrow[rr, no head, "3"'] &  & {E_{4}, \frac{46969655}{32768}}                          
\end{tikzcd} \end{center}
The elliptic curve \texttt{162.b1} is in the $\QQ$-isogeny class \texttt{162.b} and the elliptic curve \texttt{50.a1} is in the $\QQ$-isogeny class \texttt{50.a}. The $2$-adic Galois Image attached to both \texttt{162.b1} and \texttt{50.a1} is conjugate to the full lift of $\operatorname{H}_{4}$.

Those eight \textit{j}-invariants are the only \textit{j}-invariants associated to isogeny graphs of $R_{4}(21)$ and $R_{4}(15)$ type. As none of them equal $0$ or $1728$, all isogeny-torsion graphs of type $R_{4}(21)$ or $R_{4}(15)$ type are a quadratic twist of \texttt{162.b} and \texttt{50.a}. As $\operatorname{H}_{4}$ has no non-trivial quadratic twists, the proof is finished.
\end{proof}

\begin{proposition}\label{R4 Graphs Proposition}
Let $\mathcal{G}$ be an isogeny-torsion graph of $R_{4}(6)$ or $R_{4}(10)$ type. Then the $2$-adic configuration of $\mathcal{G}$ is one of the entries in Table \ref{2-adic Galois Images of R4 Graphs Table}.
\end{proposition}

\begin{proof}

Let $\mathcal{E}$ be a $\QQ$-isogeny class of elliptic curves defined over $\QQ$ and let $\mathcal{G}$ be the isogeny-torsion graph associated to $\mathcal{E}$. Let $\mathcal{G}_{2^{\infty}}$ be the $2$-adic subgraph of $\mathcal{G}$ and suppose that $\mathcal{G}_{2^{\infty}}$ is of $L_{2}(2)$ type. Suppose that the $2$-adic Galois image attached to all elliptic curves in $\mathcal{E}$ contain $\operatorname{-Id}$. Then the $2$-adic Galois image attached to $\mathcal{G}$ is one of the fifty five arrangements in the third column of Table \ref{Reductions of L2 Graphs}. Each arrangement in the third column of Table \ref{Reductions of L2 Graphs} reduces to one of the nineteen arrangements in the second column of Table \ref{Reductions of L2 Graphs}.

Let $E/\QQ$ be an elliptic curve. To prove Proposition \ref{R4 Graphs Proposition}, we need to prove that if $E$ has a $\QQ$-rational subgroup of order $5$, then $\rho_{E,2^{\infty}}(G_{\QQ})$ is not conjugate to $\operatorname{H}_{9}$, $\operatorname{H}_{11}$, $\operatorname{H}_{12}$, $\operatorname{H}_{16}$, $\operatorname{H}_{18}$, $\operatorname{H}_{23}$, or $\operatorname{H}_{26}$. Note that $\operatorname{H}_{23}$ and $\operatorname{H}_{26}$ are proper subgroups of $\operatorname{H}_{11}$ so the final list is $\operatorname{H}_{9}$, $\operatorname{H}_{11}$, $\operatorname{H}_{12}$, $\operatorname{H}_{16}$, and $\operatorname{H}_{18}$. Moreover, we need to prove that if $E(\QQ)_{\text{tors}} \cong \ZZ / 10 \ZZ$, then $\rho_{E,2^{\infty}}(G_{\QQ})$ is not conjugate to $\operatorname{H}_{14}$. We do this last case now. The fiber product of $\operatorname{H}_{14}$ and the group $\left\{\begin{bmatrix} \pm 1 & \ast \\ 0 & \ast \end{bmatrix}\right\}$ is a modular curve on the LMFDB with label \texttt{40.72.3.27}. The modular curve \texttt{40.72.3.27} is generated by a group $\operatorname{H}$ that contains the group $\operatorname{H}'$ such that $\operatorname{H}'$ generates the modular curve with LMFDB label \texttt{40.36.2.1}. The modular curve with LMFDB label \texttt{40.36.2.1} is a curve of genus $2$ and rank $0$ defined by the equation $y^{2} = -4x^{5} - 6x^{4} - 6x^{2} + 4x$.
We may use the command \texttt{Chabauty0} on Jacobian variety of the hyperelliptic curve defined by $y^{2} = -4 x^{5} - 6 x^{4} - 6 x^{2} + 4 x$ to get four rational points, all of which are cusps.

Additionally, we have to prove that if $E$ has a $\QQ$-rational subgroup of order $3$, then $\rho_{E,2^{\infty}}(G_{\QQ})$ is not conjugate to $\operatorname{H}_{11}$, $\operatorname{H}_{12}$, $\operatorname{H}_{14}$, $\operatorname{H}_{18}$, $\operatorname{H}_{23}$, $\operatorname{H}_{26}$, or $\operatorname{H}_{37}$. Note that the group $\operatorname{H}_{37}$ is a proper subgroup of $\operatorname{H}_{18}$ so the final list is $\operatorname{H}_{11}$, $\operatorname{H}_{12}$, $\operatorname{H}_{14}$, and $\operatorname{H}_{18}$. Moreover, we need to prove that if $E(\QQ)_{\text{tors}} \cong \ZZ / 6 \ZZ$, then $\rho_{E,2^{\infty}}(G_{\QQ})$ is not conjugate to $\operatorname{H}_{9a}$ or $\operatorname{H}_{9b}$. The proof of Proposition \ref{R4 Graphs Proposition} will be completed in Section \ref{elliptic curves} and Section \ref{hyperelliptic curves}.

\end{proof}

\newpage

\begin{center}
\begin{table}[h!]
\renewcommand{\arraystretch}{1.3}
\scalebox{0.57}{
    \begin{tabular}{|c|c|c|}

\hline

Isogeny-Torsion Graph & Reductive $2$-adic classification & $2$-adic classification \\
    \hline
\multirow{54}*{\includegraphics[scale=0.1]{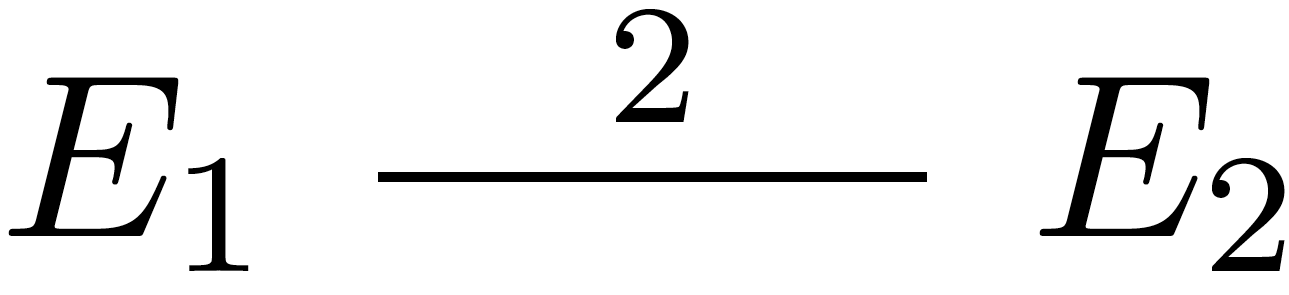}} & \multirow{10}*{$(\operatorname{H}_{6}, \operatorname{H}_{6})$} & $(\operatorname{H}_{6}, \operatorname{H}_{6})$ \\
\cline{3-3}
& & $(\operatorname{H}_{9}, \operatorname{H}_{10})$ \\
\cline{3-3}
& & $(\operatorname{H}_{11}, \operatorname{H}_{11})$ \\
\cline{3-3}
& & $(\operatorname{H}_{12}, \operatorname{H}_{12})$ \\
\cline{3-3}
& & $(\operatorname{H}_{14}, \operatorname{H}_{14})$ \\
\cline{3-3}
& & $(\operatorname{H}_{15}, \operatorname{H}_{19})$ \\
\cline{3-3}
& & $(\operatorname{H}_{16}, \operatorname{H}_{17})$ \\
\cline{3-3}
& & $(\operatorname{H}_{18}, \operatorname{H}_{18})$ \\
\cline{3-3}
& & $(\operatorname{H}_{23}, \operatorname{H}_{43})$ \\
\cline{3-3}
& & $(\operatorname{H}_{26}, \operatorname{H}_{28})$ \\
\cline{2-3}
& $(\operatorname{H}_{9}, \operatorname{H}_{10})$ & $(\operatorname{H}_{37}, \operatorname{H}_{42})$ \\
\cline{2-3}
& \multirow{10}*{$(\operatorname{H}_{11}, \operatorname{H}_{11})$} & $(\operatorname{H}_{29}, \operatorname{H}_{41})$ \\
\cline{3-3}
& & $(\operatorname{H}_{35}, \operatorname{H}_{49})$ \\
\cline{3-3}
& & $(\operatorname{H}_{39}, \operatorname{H}_{45})$ \\
\cline{3-3}
& & $(\operatorname{H}_{47}, \operatorname{H}_{47})$ \\
\cline{3-3}
& & $(\operatorname{H}_{50}, \operatorname{H}_{50})$ \\
\cline{3-3}
& & $(\operatorname{H}_{63}, \operatorname{H}_{70})$ \\
\cline{3-3}
& & $(\operatorname{H}_{73}, \operatorname{H}_{90})$ \\
\cline{3-3}
& & $(\operatorname{H}_{77}, \operatorname{H}_{80})$ \\
\cline{3-3}
& & $(\operatorname{H}_{81}, \operatorname{H}_{83})$ \\
\cline{3-3}
& & $(\operatorname{H}_{97}, \operatorname{H}_{97})$ \\
\cline{2-3}
& \multirow{4}*{$(\operatorname{H}_{12}, \operatorname{H}_{12})$} & $(\operatorname{H}_{30}, \operatorname{H}_{30})$ \\
\cline{3-3}
& & $(\operatorname{H}_{31}, \operatorname{H}_{31})$ \\
\cline{3-3}
& & $(\operatorname{H}_{40}, \operatorname{H}_{40})$ \\
\cline{3-3}
& & $(\operatorname{H}_{89}, \operatorname{H}_{93})$ \\
\cline{2-3}
& \multirow{6}*{$(\operatorname{H}_{23}, \operatorname{H}_{43})$} & $(\operatorname{H}_{64}, \operatorname{H}_{64})$ \\
\cline{3-3}
& & $(\operatorname{H}_{65}, \operatorname{H}_{71})$ \\
\cline{3-3}
& & $(\operatorname{H}_{68}, \operatorname{H}_{74})$ \\
\cline{3-3}
& & $(\operatorname{H}_{69}, \operatorname{H}_{91})$ \\
\cline{3-3}
& & $(\operatorname{H}_{76}, \operatorname{H}_{82})$ \\
\cline{3-3}
& & $(\operatorname{H}_{191}, \operatorname{H}_{196})$ \\
\cline{2-3}
& $(\operatorname{H}_{30}, \operatorname{H}_{30})$ & $(\operatorname{H}_{103}, \operatorname{H}_{104})$ \\
\cline{2-3}
& \multirow{4}*{$(\operatorname{H}_{39}, \operatorname{H}_{45})$} & $(\operatorname{H}_{105}, \operatorname{H}_{112})$ \\
\cline{3-3}
& & $(\operatorname{H}_{106}, \operatorname{H}_{111})$ \\
\cline{3-3}
& & $(\operatorname{H}_{107}, \operatorname{H}_{110})$ \\
\cline{3-3}
& & $(\operatorname{H}_{109}, \operatorname{H}_{124})$ \\
\cline{2-3}
& $(\operatorname{H}_{37}, \operatorname{H}_{42})$ & $(\operatorname{H}_{108}, \operatorname{H}_{123})$ \\
\cline{2-3}
& $(\operatorname{H}_{50}, \operatorname{H}_{50})$ & $(\operatorname{H}_{113}, \operatorname{H}_{114})$ \\
\cline{2-3}
& $(\operatorname{H}_{63}, \operatorname{H}_{70})$ & $(\operatorname{H}_{302}, \operatorname{H}_{326})$ \\
\cline{2-3}
& $(\operatorname{H}_{65}, \operatorname{H}_{71})$ & $(\operatorname{H}_{318}, \operatorname{H}_{328})$ \\
\cline{2-3}
& \multirow{3}*{$(\operatorname{H}_{69}, \operatorname{H}_{91})$} & $(\operatorname{H}_{288}, \operatorname{H}_{323})$ \\
\cline{3-3}
& & $(\operatorname{H}_{289}, \operatorname{H}_{320})$ \\
\cline{3-3}
& & $(\operatorname{H}_{291}, \operatorname{H}_{324})$ \\
\cline{2-3}
& \multirow{2}*{$(\operatorname{H}_{73}, \operatorname{H}_{90})$} & $(\operatorname{H}_{281}, \operatorname{H}_{349})$ \\
\cline{3-3}
& & $(\operatorname{H}_{284}, \operatorname{H}_{350})$ \\
\cline{2-3}
& $(\operatorname{H}_{77}, \operatorname{H}_{80})$ & $(\operatorname{H}_{216}, \operatorname{H}_{224})$ \\
\cline{2-3}
& \multirow{2}*{$(\operatorname{H}_{81}, \operatorname{H}_{83})$} & $(\operatorname{H}_{218}, \operatorname{H}_{237})$ \\
\cline{3-3}
& & $(\operatorname{H}_{220}, \operatorname{H}_{232})$ \\
\cline{2-3}
& \multirow{3}*{$(\operatorname{H}_{89}, \operatorname{H}_{93})$} & $(\operatorname{H}_{295}, \operatorname{H}_{297})$ \\
\cline{3-3}
& & $(\operatorname{H}_{556}, \operatorname{H}_{563})$ \\
\cline{3-3}
& & $(\operatorname{H}_{558}, \operatorname{H}_{566})$ \\
\cline{2-3}
& \multirow{2}*{$(\operatorname{H}_{97}, \operatorname{H}_{97})$} & $(\operatorname{H}_{304}, \operatorname{H}_{309})$ \\
\cline{3-3}
& & $(\operatorname{H}_{308}, \operatorname{H}_{312})$ \\
\cline{2-3}
& $(\operatorname{H}_{108}, \operatorname{H}_{123})$ & $(\operatorname{H}_{238}, \operatorname{H}_{239})$ \\
\cline{2-3}
& $(\operatorname{H}_{288}, \operatorname{H}_{323})$ & $(\operatorname{H}_{619}, \operatorname{H}_{649})$ \\
\hline
\end{tabular}}
\caption{Reductions of $L_{2}$ Graphs}
\label{Reductions of L2 Graphs}
\end{table}
\end{center}

\subsection{Isogeny-torsion graphs of $R_{6}$ type}

\begin{proposition}\label{R6 graphs proposition}
Let $\mathcal{G}$ be an isogeny-torsion graph of $R_{6}$ type. Then the $2$-adic configuration of $\mathcal{G}$ is one of the two configurations in Table \ref{2-adic Galois Images of R6 Graphs Table}.
\end{proposition}

\begin{proof}\label{R6 graphs proposition remark}
We proceed similarly to the proof in Proposition \ref{R4 Graphs Proposition} but with less detail. Once we fully prove Proposition \ref{R4 Graphs Proposition}, then we only need to check if any of the possible arrangements of $2$-adic Galois images from Table \ref{2-adic Galois Images of R4 Graphs Table} are possible for isogeny-torsion graphs of type $R_{6}$. In other words, we must prove that if $E/\QQ$ is an elliptic curve with two $\QQ$-rational subgroups of order $3$, then $\rho_{E,2^{\infty}}(G_{\QQ})$ is not conjugate to $\operatorname{H}_{9}$ or $\operatorname{H}_{15}$. The proof of Proposition \ref{R6 graphs proposition} will be completed in Section \ref{elliptic curves}.

\end{proof}

\subsection{Isogeny-torsion graphs of $S$ type}

\begin{proposition}\label{S graphs proposition}
Let $\mathcal{G}$ be an isogeny-torsion graph of $S$ type. Then the $2$-adic configuration of $\mathcal{G}$ is one of the entries in Table \ref{2-adic Galois Images of S Graphs Table}.
\end{proposition}

\begin{proof}
\phantom{h}

\begin{center}
\begin{table}[h!]
\renewcommand{\arraystretch}{1.3}
\scalebox{0.9}{
    \begin{tabular}{|c|c|c|}

\hline

Isogeny-Torsion Graph & Reductive $2$-adic classification & $2$-adic classification \\
    \hline
\multirow{10}*{\includegraphics[scale=0.1]{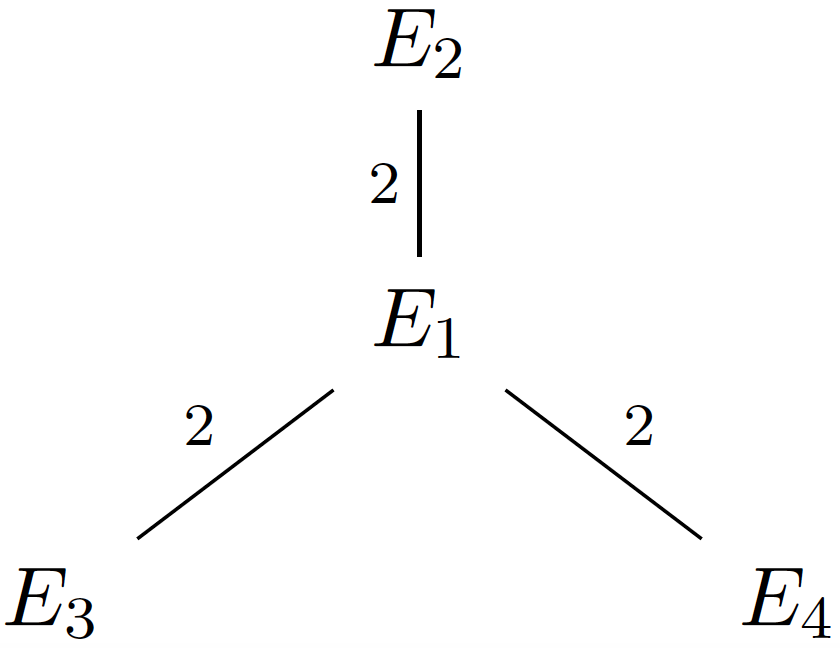}} & \multirow{4}*{$(\operatorname{H}_{8}, \operatorname{H}_{13}, \operatorname{H}_{13}, \operatorname{H}_{13})$} & $(\operatorname{H}_{8}, \operatorname{H}_{13}, \operatorname{H}_{13}, \operatorname{H}_{13})$ \\
\cline{3-3} & & $(\operatorname{H}_{24}, \operatorname{H}_{27}, \operatorname{H}_{32}, \operatorname{H}_{32})$ \\
\cline{3-3}
& & $(\operatorname{H}_{38}, \operatorname{H}_{34}, \operatorname{H}_{34}, \operatorname{H}_{44})$ \\
\cline{3-3}
& & $(\operatorname{H}_{46}, \operatorname{H}_{33}, \operatorname{H}_{33}, \operatorname{H}_{48})$ \\
\cline{2-3}
& \multirow{3}*{$(\operatorname{H}_{24}, \operatorname{H}_{27}, \operatorname{H}_{32}, \operatorname{H}_{32})$} & $(\operatorname{H}_{61}, \operatorname{H}_{95}, \operatorname{H}_{116}, \operatorname{H}_{116})$ \\
\cline{3-3}
& & $(\operatorname{H}_{66}, \operatorname{H}_{79}, \operatorname{H}_{79}, \operatorname{H}_{94})$ \\
\cline{3-3}
& & $(\operatorname{H}_{67}, \operatorname{H}_{60}, \operatorname{H}_{115}, \operatorname{H}_{115})$ \\
\cline{2-3}
& \multirow{2}*{$(\operatorname{H}_{61}, \operatorname{H}_{95}, \operatorname{H}_{116}, \operatorname{H}_{116})$} & $(\operatorname{H}_{209}, \operatorname{H}_{222}, \operatorname{H}_{241}, \operatorname{H}_{241})$ \\
\cline{3-3}
& & $(\operatorname{H}_{210}, \operatorname{H}_{221}, \operatorname{H}_{242}, \operatorname{H}_{242})$ \\
\cline{2-3}
& $(\operatorname{H}_{66}, \operatorname{H}_{79}, \operatorname{H}_{79}, \operatorname{H}_{94})$ & $(\operatorname{H}_{214}, \operatorname{H}_{226}, \operatorname{H}_{226}, \operatorname{H}_{231})$ \\
\hline
\end{tabular}}
\caption{Reductions of $T_{4}$ Graphs}
\label{Reductions of T4 Graphs}
\end{table}
\end{center}

Let $\mathcal{E}$ be a $\QQ$-isogeny class of elliptic curves defined over $\QQ$ and let $\mathcal{G}$ be the isogeny-torsion graph associated to $\mathcal{E}$. Let $\mathcal{G}_{2^{\infty}}$ be the $2$-adic subgraph of $\mathcal{G}$. Suppose that $\mathcal{G}_{2^{\infty}}$ is of $T_{4}$ type and that the $2$-adic Galois Image attached to each elliptic curve in $\mathcal{E}$ contains $\operatorname{-Id}$. Then the $2$-adic Galois image attached to $\mathcal{G}$ is one of the ten arrangements in the third column of Table \ref{Reductions of T4 Graphs}. Each arrangement in the third column of Table \ref{Reductions of T4 Graphs} reduces to one of the four arrangements in the second column of Table \ref{Reductions of T4 Graphs}.

Hence, to prove Proposition \ref{S graphs proposition}, we need to prove that if $E/\QQ$ is an elliptic curve with a $\QQ$-rational subgroup of order $3$, then $\rho_{E,2^{\infty}}(G_{\QQ})$ is not conjugate to $\operatorname{H}_{24}$, $\operatorname{H}_{38}$, or $\operatorname{H}_{46}$, and if $E(\QQ)_{\texttt{tors}} \cong \ZZ / 2 \ZZ \times \ZZ / 6 \ZZ$, then $\rho_{E,2^{\infty}}(G_{\QQ})$ is not conjugate to $\operatorname{H}_{8a}$ or $\operatorname{H}_{8b}$.

Note that $\operatorname{H}_{24}$ is a proper subgroup of $\operatorname{H}_{11}$, $\operatorname{H}_{38}$ is a proper subgroup of $\operatorname{H}_{14}$, and $\operatorname{H}_{46}$ is a proper subgroup of $\operatorname{H}_{18}$. By Lemma \ref{Galois images}, it suffices to prove that if $\overline{\rho}_{E,3}(G_{\QQ})$ is conjugate to a subgroup of $\left\{\begin{bmatrix} \ast & \ast \\ 0 & \ast \end{bmatrix} \right\} \subseteq \operatorname{GL}(2, \ZZ / 3 \ZZ)$, then $\rho_{E,2^{\infty}}(G_{\QQ})$ is not conjugate to a subgroup of $\operatorname{H}_{11}$, $\operatorname{H}_{14}$, or $\operatorname{H}_{18}$ but this is already taken care of in Proposition \ref{R4 Graphs Proposition}. Thus, the proof of Proposition \ref{S graphs proposition} will be completed in Section \ref{elliptic curves} and \ref{hyperelliptic curves}.

\end{proof}

\subsection{Isogeny-torsion graphs of $L_{2}(p)$ type}

\begin{lemma}\label{H3 11-isogeny}
Let $E/\QQ$ be an elliptic curve without CM such that the isogeny graph associated to the $\QQ$-isogeny class of $E$ is of $L_{2}(11)$ type. Then $\rho_{E,2^{\infty}}(G_{\QQ})$ is conjugate to the full lift of $\operatorname{H}_{3}$, $\operatorname{H}_{3a}$, or $\operatorname{H}_{3b}$.
\end{lemma}

\begin{proof}
Let $E/\QQ$ be a non-CM elliptic curve that has a $\QQ$-rational subgroup of order $11$. Then the isogeny graph associated to the $\QQ$-isogeny class of $E$ is below with the respective \textit{j}-invariants
\begin{center}
    \begin{tikzcd}
{E_{1}, \textit{j}_{E_{1}}=-121} \arrow[rr, no head, "11"] &  & {E_{2}, \textit{j}_{E_{2}}=-24729001}
\end{tikzcd}
\end{center}
By Corollary \ref{coprime isogeny-degree}, $\rho_{E_{1},2^{\infty}}(G_{\QQ})$ is conjugate to $\rho_{E_{2},2^{\infty}}(G_{\QQ})$. Let $E$ be the elliptic curve with LMFDB label \texttt{121.a1} and let $E'$ be the elliptic curve with LMFDB label \texttt{121.a2}. Then $\textit{j}_{E} = -121$ and $\textit{j}_{E'} = -24729001$. The \textit{j}-invariants $-121$ and $-24729001$ are the only two \textit{j}-invariants associated to non-CM elliptic curves over $\QQ$ with a $\QQ$-rational subgroup of order $11$. As none of those \textit{j}-invariants are equal to $0$ or $1728$, all non-CM elliptic curves over $\QQ$ with a $\QQ$-rational subgroup of order $11$ are quadratic twists of $E$ or $E'$. Both $\rho_{E,2^{\infty}}(G_{\QQ})$ and $\rho_{E',2^{\infty}}(G_{\QQ})$ are conjugate to $\operatorname{H}_{3} = \left\langle \begin{bmatrix} 3 & 3 \\ 0 & 1 \end{bmatrix}, \begin{bmatrix} 0 & 1 \\ 3 & 1 \end{bmatrix} \right\rangle$. And the only quadratic twists of $\operatorname{H}_{3}$ are $\operatorname{H}_{3a}$, $\operatorname{H}_{3b}$, and of course, $\operatorname{H}_{3}$ itself.

\end{proof}

\begin{lemma}\label{H1 37 isogeny and 17 isogeny}
Let $E/\QQ$ be an elliptic curve such that the isogeny graph associated to the $\QQ$-isogeny class of $E$ is of $L_{2}(17)$ or $L_{2}(37)$ type. Then $\rho_{E,2^{\infty}}$ is surjective.
\end{lemma}

\begin{proof}
The isogeny graph associated to the $\QQ$-isogeny class of $E$ is one of the two graphs below with respective \textit{j}-invariants
\begin{center} \begin{tikzcd}
{E_{1}, \frac{-882216989}{131072}} \arrow[rr, no head, "17"] &  & {E_{2}, \frac{-297756989}{2}}
\end{tikzcd}

\begin{tikzcd}
{E'_{1}, -162677523113838677} \arrow[rrr, no head, "37"] &  &  & {E'_{2}, -9317}
\end{tikzcd}
\end{center}
The \textit{j}-invariants $\frac{-882216989}{131072}$ and $\frac{-297756989}{2}$ are the only two \textit{j}-invariants associated to non-CM elliptic curves over $\QQ$ with a $\QQ$-rational subgroup of order $17$ and $-162677523113838677$ and $-9317$ are the only two \textit{j}-invariants associated to non-CM elliptic curves over $\QQ$ with a $\QQ$-rational subgroup of order $37$. By Corollary \ref{coprime isogeny-degree}, $\rho_{E_{1},2^{\infty}}(G_{\QQ})$ is conjugate to $\rho_{E_{2},2^{\infty}}(G_{\QQ})$ and $\rho_{E'_{1},2^{\infty}}(G_{\QQ})$ is conjugate to $\rho_{E'_{2},2^{\infty}}(G_{\QQ})$. By the fact that none of the \textit{j}-invariants are equal to $0$ or $1728$, each non-CM elliptic curve over $\QQ$ with a $\QQ$-rational subgroup of order $17$ is a quadratic twist of $E_{1}$ or $E_{2}$ and each elliptic curve over $\QQ$ with a $\QQ$-rational subgroup of order $37$ is a quadratic twist of $E_{1}'$ or $E_{2}'$.

Let $\widehat{E}$ be the elliptic curve with LMFDB label \texttt{1225.b2} and let $\widetilde{E}$ be the elliptic curve with LMFDB label \texttt{14450.b2}. Then $\textit{j}_{\widehat{E}} = -9317$ and $\textit{j}_{\widetilde{E}} = \frac{-297756989}{2}$. Moreover, $\rho_{\widehat{E},2^{\infty}}(G_{\QQ})$ and $\rho_{\widetilde{E},2^{\infty}}(G_{\QQ})$ are both conjugate to $\operatorname{H}_{1} = \operatorname{GL}(2, \ZZ_{2})$. Finally, the only quadratic twist of $\operatorname{H}_{1}$ is $\operatorname{H}_{1}$ itself.
\end{proof}

\begin{proposition}\label{L2p Graphs Proposition}
Let $\mathcal{G}$ be an isogeny-torsion graph of $L_{2}(p)$ type where $p = 3$, $5$, $7$, or $13$. Then the $2$-adic configuration of $\mathcal{G}$ is one of the entries in Table \ref{2-adic Galois Images of L2 Graphs Table Odd}.
\end{proposition}

\begin{proof}\label{L2p Graphs Proposition Remark}
The right column of Table \ref{Reductions of L1 Graphs} shows arrangements of subgroups of $\operatorname{GL}(2, \ZZ_{2})$ that reduce to the groups in the middle column of Table \ref{Reductions of L1 Graphs}.

\begin{center}
\begin{table}[h!]
\renewcommand{\arraystretch}{1.3}
\scalebox{0.55}{
    \begin{tabular}{|c|c|c|}

\hline

Isogeny-Torsion Graph & Reductive $2$-adic classification & $2$-adic classification \\
    \hline
\multirow{14}*{$([1])$} & \multirow{7}*{$\operatorname{H}_{1}$} & $\operatorname{H}_{1}$ \\
\cline{3-3}
& & $\operatorname{H}_{2}$ \\
\cline{3-3}
& & $\operatorname{H}_{3}$ \\
\cline{3-3}
& & $\operatorname{H}_{4}$ \\
\cline{3-3}
& & $\operatorname{H}_{5}$ \\
\cline{3-3}
& & $\operatorname{H}_{7}$ \\
\cline{3-3}
& & $\operatorname{H}_{20}$ \\
\cline{2-3}
& \multirow{2}*{$\operatorname{H}_{7}$} & $\operatorname{H}_{22}$ \\
\cline{3-3}
& & $\operatorname{H}_{55}$ \\
\cline{2-3}
& \multirow{2}*{$\operatorname{H}_{22}$} & $\operatorname{H}_{56}$ \\
\cline{3-3}
& & $\operatorname{H}_{57}$ \\
\cline{2-3}
& $\operatorname{H}_{55}$ & $\operatorname{H}_{441}$ \\
\cline{2-3}
& \multirow{2}*{$\operatorname{H}_{56}$} & $\operatorname{H}_{177}$ \\
\cline{3-3}
& & $\operatorname{H}_{178}$ \\
\hline
\end{tabular}}
\caption{Reductions of $L_{1}$ Graphs}
\label{Reductions of L1 Graphs}
\end{table}
\end{center}

Let $\mathcal{G}_{2^{\infty}}$ be the $2$-adic subgraph of $\mathcal{G}$. Then $\mathcal{G}_{2^{\infty}}$ is an isogeny-torsion graph of type $L_{1}$. Note that $\operatorname{H}_{20}$ is a proper subgroup of $\operatorname{H}_{7}$. Let $E/\QQ$ be an elliptic curve.

We need to prove that if $E$ contains a $\QQ$-rational subgroup of order $13$, then $\rho_{E,2^{\infty}}(G_{\QQ})$ is not conjugate to a subgroup of $\operatorname{H}_{2}, \operatorname{H}_{3}, \operatorname{H}_{5}$, or $\operatorname{H}_{7}$. We need to investigate $\rho_{E,2^{\infty}}(G_{\QQ})$ when $E$ contains a $\QQ$-rational subgroup of order $7$; in particular, prove that $\rho_{E,2^{\infty}}(G_{\QQ})$ is not conjugate to $\operatorname{H}_{7}$ exactly. Moreover, we need to prove that if $E(\QQ)_{\text{tors}} \cong \ZZ / 7 \ZZ$, then $\rho_{E,2^{\infty}}(G_{\QQ})$ is not conjugate to a subgroup of $\operatorname{H}_{2}$, $\operatorname{H}_{3}$, $\operatorname{H}_{4}$, $\operatorname{H}_{5}$, or $\operatorname{H}_{7}$. We need to prove that if $E$ contains a $\QQ$-rational subgroup of order $5$, then $\rho_{E,2^{\infty}}(G_{\QQ})$ is not conjugate to a subgroup of $\operatorname{H}_{2}$, $\operatorname{H}_{3}$, $\operatorname{H}_{5}$, or $\operatorname{H}_{20}$. Moreover, we need to prove that if $E(\QQ)_{\text{tors}} \cong \ZZ / 5 \ZZ$, then $\rho_{E,2^{\infty}}(G_{\QQ})$ is not conjugate to a subgroup of $\operatorname{H}_{7}$. We need to investigate what kind of subgroup of $\operatorname{H}_{7}$ that $\rho_{E,2^{\infty}}(G_{\QQ})$ can be when $E$ has a $\QQ$-rational subgroup of order $3$; in particular, prove that $\rho_{E,2^{\infty}}(G_{\QQ})$ can be conjugate to a quadratic twist of $\operatorname{H}_{20}$ but cannot be conjugate to $\operatorname{H}_{7}$ exactly. Finally, we have to prove that if $E(\QQ)_{\text{tors}} \cong \ZZ / 3 \ZZ$, then $\rho_{E,2^{\infty}}(G_{\QQ})$ is not conjugate to $\operatorname{H}_{20a}$.

Let $\operatorname{B}_{1}(3)$ denote the subgroup of $\operatorname{GL}(2, \ZZ / 3 \ZZ)$ consisting of all matrices of the form $\begin{bmatrix}
1 & c \\ 0 & d
\end{bmatrix}$ and let $\operatorname{B}_{0}(13)$ denote the subgroup of $\operatorname{GL}(2, \ZZ / 13 \ZZ)$ consisting of all upper-triangular matrices. We prove below that all of the rational points on the fiber product of $\operatorname{H}_{7}$ and $\operatorname{B}_{0}(13)$ are cusps and all of the rational points on the fiber product of $\operatorname{H}_{20a}$ and $\operatorname{B}_{1}(3)$ are cusps.

\begin{itemize}
    \item $\operatorname{H}_{7} \times \operatorname{B}_{0}(13)$
    
    The author would like to thank David Zureick-Brown for help on classifying the rational points on the fiber product of $\operatorname{H}_{7}$ and $\operatorname{B}_{0}(13)$. We take inspiration from the study of non-hyperelliptic curves of genus $3$ in Subsection 8.3 of \cite{Rouse2021elladicIO}; especially the first and second examples in the subsection. The modular curve $\operatorname{X}$ generated by the group $\operatorname{H}_{7} \times \operatorname{B}_{0}(13)$ is denoted \texttt{52.56.3.1} in the LMFDB.
    
    In the LMFDB, $\operatorname{X}$ has canonical model
    $$\operatorname{X'} : x^{2}y^{2} + 2x^{3}z + 2x^{2}yz - 2xy^{2}z - 2y^{3}z - x^{2}z^{2} - 8xyz^{2} - y^{2}z^{2} - 4z^{4} = 0.$$
    Both $\operatorname{X}$ and $\operatorname{X}'$ are curves of genus $3$ that are not hyperelliptic. There are two obvious rational points on $\operatorname{X}'$, namely, $R_{1} = (1,0,0)$ and $R_{2} = (0,1,0)$. The denominator of the \textit{j}-invariant of $\operatorname{X}'$ listed in the LMFDB of $\operatorname{X}$ has a factor of $z^{2}$ and hence, $R_{1}$ and $R_{2}$ are cusps. It remains to prove that $\operatorname{X}'$ has no other rational points. Let $\iota \colon \operatorname{X}' \to \operatorname{X}'$ be the map that maps a point $(a,b,c)$ on $\operatorname{X}'$ to $(-b,-a,c)$. Then $\iota$ is an automorphism on $\operatorname{X}'$ of order $2$. Note that $\iota(R_{1}) = R_{2}$ and hence, neither $R_{1}$ nor $R_{2}$ are fixed points of $\iota$. Let $(a,b,1)$ be a point on $\operatorname{X}'$ that is a fixed point of $\iota$. Then $(a,b,1) = (-b,-a,1)$. Thus, $b = -a$ and $a$ satisfies the expression $a^{4} + 6a^{2} - 4$, which has no rational solutions. Hence, $(a,b,1)$ is not a rational point on $\operatorname{X}'$.
    
    The quotient curve of $\operatorname{X}'$ modulo the subgroup of the automorphism group on $\operatorname{X}'$ generated by $\iota$ is isomorphic to the elliptic curve $E/\QQ$ with LMFDB label \texttt{208.b.1}. Note that $E(\QQ) \cong \ZZ$. Let $\operatorname{J}_{\operatorname{X}'}$ be the Jacobian of $\operatorname{X}'$. We can consider the elements of $\operatorname{J}_{\operatorname{X}'}$ to be $\operatorname{Pic}^{0}(C)$, the group of divisors on $C$ modulo the group of principal divisors on $C$. There is an abelian surface $A$ such that $\operatorname{J}_{\operatorname{X}'}$ is isogenous to $E \times A$. Both $\operatorname{J}_{\operatorname{X}'}$ and $E$ have analytic rank equal to $1$ and hence, the analytic rank of $A$ is equal to $0$. By Corollary A.1.8 in \cite{Rouse2021elladicIO}, the algebraic rank of $A$ is equal to $0$. Hence, $A(\QQ)$ is torsion.
    
    Note that for a rational point $P \in \operatorname{X}'(\QQ)$, $P$ and $\iota(P)$ have the same image in $E(\QQ)$. Thus, $P - \iota(P)$ is the identity element in $E$. Extending this map, we see that $P - \iota(P)$ is torsion in $E(\QQ) \times A(\QQ)$. The map $\operatorname{J}_{\operatorname{X}'}(\QQ) \to E(\QQ) \times A(\QQ)$ is finite and hence, $P - \iota(P)$ is a divisor on $\operatorname{J}_{\operatorname{X}'}$ of finite order. We compute $\operatorname{J}_{\operatorname{X}' \texttt{tors}}$ and the preimage of the Abel-Jacobi map $\tau \colon \operatorname{X}'(\QQ) \to \operatorname{J}_{\operatorname{X}' \texttt{tors}}$ that maps a point $P \in \operatorname{X}'(\QQ)$ to $P - \iota(P)$.
    
    Let $p$ be a prime of $\operatorname{X}'$ of good reduction. Then the map $\operatorname{J}_{\operatorname{X}' \texttt{tors}}(\QQ) \to \operatorname{J}_{\operatorname{X}'}(\FF_{p})$ is injective. Using the Magma command \texttt{ClassGroup} on $\operatorname{X}'_{3}$ gives a group isomorphic to $\ZZ \times \ZZ / 80 \ZZ$ and using the Magma command \texttt{ClassGroup} on $\operatorname{X}'_{7}$ gives a group isomorphic to $\ZZ \times \ZZ / 13 \ZZ \times \ZZ / 52 \ZZ$. By $\operatorname{GCD}$ computations, $\operatorname{J}_{\operatorname{X}' \texttt{tors}}(\QQ)$ is a cyclic group of order $1$, $2$, or $4$. By the fact that $R_{1} - R_{2} = R_{1} - \iota(R_{1})$ is a divisor on $\operatorname{X}'$ of order $4$, we have that $\operatorname{J}_{\operatorname{X}' \texttt{tors}}(\QQ)$ is in fact a cyclic group of order $4$.
    
    Using another Magma computation, we see that $R_{1} - \iota(R_{1})$ is a divisor over $\operatorname{X}'$ and $2 \cdot R_{1} - 2 \cdot \iota(R_{1})$ is not a principal divisor. Hence, the order of $R_{1} - \iota(R_{1})$ in $\operatorname{J}_{\operatorname{X}'}$ is equal to $4$. We note that the the map $\tau$ is injective. To show this, let $P, Q \in \operatorname{X}' (\QQ)$ and first, suppose that $P - \iota(P) = Q - \iota(Q)$ as divisors over $\operatorname{X}'$ (without mention of equivalence in $\operatorname{J}_{\operatorname{X}'}$). As $P \neq \iota(P)$ and $Q \neq \iota(Q)$ we have $P = Q$. Next, suppose that $P - \iota(P)$ is equivalent to $Q - \iota(Q)$ as an element of $\operatorname{J}_{\operatorname{X}'}$. As $P \neq \iota(P)$ and $Q \neq \iota(Q)$, this gives a $g_{2}^{1}$ on a curve that is not hyperelliptic, which is impossible as mentioned in \cite{Rouse2021elladicIO}.
    
    Note that $\tau(P) = P - \iota(P)$ is a principal divisor if and only if $P = \iota(P)$, which is not possible for rational points on $\operatorname{X}'$. Next, we see that $\tau(R_{1}) = R_{1} - \iota(R_{1}) = R_{1} - R_{2}$ and $\tau(R_{2}) = R_{2} - \iota(R_{2}) = R_{2} - R_{1} = -(R_{1} - R_{2})$. It remains to prove that there is no rational point $P$ on $\operatorname{X}'$ such that $P - \iota(P)$ is linearly equivalent to $2 \cdot R_{1} - 2 \cdot R_{2}$. We use the sieving method employed in the second example in Subsection 8.3 of \cite{Rouse2021elladicIO}. Let $\operatorname{X}_{3}'$ be the reduction of $\operatorname{X}'$ modulo $3$. Then $\operatorname{X}_{3}'$ is a smooth curve which has six rational points. Running a simple search, we see that there is no rational point $r$ in $\operatorname{X}_{3}'$ such that $r - \iota(r)$ is linearly equivalent to the image of $2 \cdot (R_{1}) - 2 \cdot R_{2}$ in $\operatorname{X}_{3}'$. Thus, there are two rational points on $\operatorname{X}'$, both of which are cusps.
    
    \item $\operatorname{H}_{20a} \times \operatorname{B}_{1}(3)$
    
    The author would like to thank David Zureick-Brown for help with this question.
    
    Let $\operatorname{H} = \operatorname{H}_{20a} \times \operatorname{B}_{1}(3)$. The modular curve generated by $\operatorname{H}$ has LMFDB label \texttt{24.128.1.13}. Let $\operatorname{H}' = \left\langle \operatorname{H}, \operatorname{-Id} \right\rangle$. The modular curve $\operatorname{X}'$ generated by $\operatorname{H}'$ has LMFDB label \texttt{24.64.1.23}. The embedded model $\operatorname{M}$ of $\operatorname{X}'$ in $\mathbb{P}^{3}$ has defining equations
    $$0 = 2x^{2}-z^{2}-2zw = 4xy-6y^{2}-w^{2}.$$
    Using the command \texttt{IsSoluble(M,2)} produces no points on $M$ defined over $\QQ_{2}$. Hence, both $\operatorname{M}$ and the modular curve with LMFDB label \texttt{24.128.1.13} have no rational points. The proof of Proposition \ref{L2p Graphs Proposition} will be completed in Section \ref{elliptic curves} and Section \ref{hyperelliptic curves}.

\end{itemize}

\end{proof}

\subsection{Isogeny-torsion graphs of $L_{3}\left(p^{2}\right)$ type}

\begin{proposition}\label{L3p Graphs proposition}
Let $E/\QQ$ be an elliptic curve such that the isogeny graph associated to the $\QQ$-isogeny class of $E$ is of $L_{3}$ type. Then $\rho_{E,2^{\infty}}$ is surjective.
\end{proposition}

\begin{proof}
After we prove Proposition \ref{L2p Graphs Proposition} and note the work in Section \ref{elliptic curves}, we have to prove that if the image of $\overline{\rho}_{E,3}(G_{\QQ})$ is conjugate to a set of diagonal matrices modulo $3$, then $\rho_{E,2^{\infty}}(G_{\QQ})$ is not conjugate to a subgroup of $\operatorname{H}_{2}$, $\operatorname{H}_{3}$, $\operatorname{H}_{4}$, or $\operatorname{H}_{5}$. Finally, we have to prove that if $E/\QQ$ is a non-CM elliptic curve such that $\overline{\rho}_{E,5}(G_{\QQ})$ is conjugate to a set of diagonal matrices modulo $5$, then $\rho_{E,2^{\infty}}(G_{\QQ})$ is not conjugate to a subgroup of $\operatorname{H}_{4}$, which we do now.

Let $\mathcal{N}_{sp}(5)$ denote the normalizer of the split Cartan group modulo $5$. By Lemma \ref{Galois images}, it is enough to prove that all of the rational points on the modular curve generated by $\mathcal{N}_{sp}(5) \times \operatorname{H}_{4}$ are cusps or CM points. The (homogenized) \textit{j}-invariant associated to $\mathcal{N}_{sp}(5)$ is equal to $\frac{(s+5z)^{3}(s^{2}-5z^{2})^{3}(s^{2}+5sz+10z^{2})^{3}}{(s^{2}+5sz+5z^{2})^{5}z^{5}}$ and the (homogenized) \textit{j}-invariant associated to $\operatorname{H}_{4}$ is equal to $\frac{1728z^{2}-2t^{2}}{z^{2}}$. Equating the two \textit{j}-invariants, we get a projective curve
    $$C : (s+5z)^{3}(s^{2}-5z^{2})^{3}(s^{2}+5sz+10z^{2})^{3} - z^{3}(s^{2}+5sz-5z^{2})^{5}(1728z^{2}-2t^{2}) = 0.$$
Using the command, \texttt{RationalPoints(C)} returns four points of some bounded height, namely, $(2:58:1)$, $(2:-58:1)$, $(0:1:0)$, and $(0:0:1)$. The \textit{j}-invariant associated to the first two points is $-5000$, the point $(0:1:0)$ is a cusp, and the \textit{j}-invariant associated to $(0:0:1)$ is $1728$. Let $E/\QQ$ be the elliptic curve with LMFDB label \texttt{800.b1} and let $E'$ be an elliptic curve over $\QQ$ with \textit{j}-invariant equal to $-5000$. Then $\overline{\rho}_{E,5}(G_{\QQ})$ is conjugate to $\mathcal{N}_{sp}(5)$ and $E'$ is a quadratic twist of $E$. Let $H$ be a subgroup of $\mathcal{N}_{sp}(5)$ of index $2$. As $H$ is a normal subgroup of $\mathcal{N}_{sp}(5)$, $H$ contains all squares of $\mathcal{N}_{sp}(5)$. Hence, $\begin{bmatrix} 2 & 0 \\ 0 & 1 \end{bmatrix}^{2} = \begin{bmatrix} -1 & 0 \\ 0 & 1 \end{bmatrix}$ and $\begin{bmatrix} 1 & 0 \\ 0 & 2 \end{bmatrix}^{2} = \begin{bmatrix} 1 & 0 \\ 0 & -1 \end{bmatrix}$ are elements of $H$. Thus, $\operatorname{-Id} = \begin{bmatrix} -1 & 0 \\ 0 & 1 \end{bmatrix} \cdot \begin{bmatrix}
    1 & 0 \\ 0 & -1
    \end{bmatrix}$ is an element of $H$. By Lemma \ref{subgroups of index 2 contain -Id}, $\overline{\rho}_{E',5}(G_{\QQ})$ is conjugate to $\mathcal{N}_{sp}(5)$. Finally, we prove that $C$ contains four rational points.

The curve $C$ is isomorphic to the hyperelliptic curve $C' : y^{2} = -4x^{5} + 6x^{4} + 21x^{3} + 6x^{2} - 4x$. Let $\operatorname{J}'$ be the Jacobian variety of $C'$. Then the $\QQ$-rank of $\operatorname{J}'$ is equal to $1$. There are four known rational points on $C'$, namely, $(1:0:0)$, $(0:0:1)$, $(1:5:1)$, and $(1:-5:1)$. Then the point $(1:5:1) / (1:0:0)$ is a point of infinite order on $\operatorname{J}'$ and hence, generates $\operatorname{J}'(\QQ)$. Computing \texttt{Chabauty} on this point returns four rational points. The proof of Proposition \ref{L3p Graphs proposition} will be completed in Section \ref{elliptic curves} and Section \ref{hyperelliptic curves}.

\end{proof}

\begin{center}
\begin{table}[h!]
\renewcommand{\arraystretch}{1.3}
\scalebox{0.45}{
    \begin{tabular}{|c|c|c|c|c|c|c|c|c|c|c|}
    \hline
    \# & graph & $p$ & $\overline{\rho}_{E,p}(G_{\QQ})$ & $\rho_{E,2^{\infty}}(G_{\QQ})$ & $\textit{j}_{p}(E)$ & $\textit{j}_{2^{\infty}}(E)$ & $H$ & $\operatorname{X}_{\operatorname{H}}(\QQ)$ & \# $\QQ$-cusps & LMFDB \\
    \hline
    \multirow{2}*{$1$} & \multirow{18}*{$R_{4}(2p)$} & \multirow{10}*{$5$} & \multirow{18}*{$\left\{\begin{bmatrix} \ast & \ast \\ 0 & \ast \end{bmatrix} \right\}$} & \multirow{2}*{$\operatorname{H}_{9}$} & \multirow{10}*{$\frac{(s^{2}+10s+5)^{3}}{s}$} & \multirow{2}*{$\frac{(t^{2}+48)^{3}}{t^{2}+64}$} & \multirow{2}*{$\begin{bmatrix} 1 & 1 \\ 0 & 1 \end{bmatrix}, \begin{bmatrix} 9 & 0 \\ 0 & 1 \end{bmatrix}, \begin{bmatrix} 13 & 0 \\ 0 & 1 \end{bmatrix}, \begin{bmatrix} 3 & 0 \\ 0 & 3 \end{bmatrix}, \begin{bmatrix} 3 & 0 \\ 0 & 11 \end{bmatrix}, \begin{bmatrix} 3 & 0 \\ 10 & 1 \end{bmatrix}$} & \multirow{2}*{$\ZZ / 2 \ZZ$} & \multirow{2}*{$2$} & \multirow{2}*{\texttt{20.36.1.3}} \\
    & & & & & & & & & & \\
    \cline{1-1}
    \cline{5-5}
    \cline{7-11}
    \multirow{2}*{$2$} & & & & \multirow{2}*{$\operatorname{H}_{11}$} & & \multirow{2}*{$\frac{(t-4)^{3}(t+4)^{3}}{t^{2}}$} & \multirow{2}*{$\begin{bmatrix} 1 & 2 \\ 0 & 1 \end{bmatrix}, \begin{bmatrix} 9 & 0 \\ 0 & 1 \end{bmatrix}, \begin{bmatrix} 13 & 0 \\ 0 & 1 \end{bmatrix}, \begin{bmatrix} 7 & 9 \\ 0 & 1 \end{bmatrix}, \begin{bmatrix} 1 & 0 \\ 0 & 9 \end{bmatrix}, \begin{bmatrix} 1 & 0 \\ 0 & 13 \end{bmatrix}, \begin{bmatrix} 9 & 2 \\ 10 & 7 \end{bmatrix}$} & \multirow{2}*{$\ZZ / 2 \ZZ \times \ZZ / 2 \ZZ$} & \multirow{2}*{$4$} & \multirow{2}*{\texttt{20.36.1.4}} \\
    & & & & & & & & & & \\
    \cline{1-1}
    \cline{5-5}
    \cline{7-11}
    \multirow{2}*{$3$} & & & & \multirow{2}*{$\operatorname{H}_{12}$} & & \multirow{2}*{$\frac{(t^{2}+16)^{3}}{t^{2}}$} & \multirow{2}*{$\begin{bmatrix} 3 & 0 \\ 0 & 1 \end{bmatrix}, \begin{bmatrix} 11 & 0 \\ 0 & 1 \end{bmatrix}, \begin{bmatrix} 1 & 2 \\ 0 & 1 \end{bmatrix}, \begin{bmatrix} 1 & 0 \\ 0 & 3 \end{bmatrix}, \begin{bmatrix} 1 & 0 \\ 0 & 11 \end{bmatrix}, \begin{bmatrix} 11 & 1 \\ 10 & 1 \end{bmatrix}$} & \multirow{2}*{$\ZZ / 2 \ZZ \times \ZZ / 2 \ZZ$} & \multirow{2}*{$4$} & \multirow{2}*{\texttt{20.36.1.2}} \\
    & & & & & & & & & & \\
    \cline{1-1}
    \cline{5-5}
    \cline{7-11}
    \multirow{2}*{$4$} & & & & \multirow{2}*{$\operatorname{H}_{16}$} & & \multirow{2}*{$4 \cdot \frac{(t^{2}-24)^{3}}{t^{2}-32}$} & \multirow{2}*{$\begin{bmatrix} 1 & 1 \\ 0 & 1 \end{bmatrix}, \begin{bmatrix} 31 & 0 \\ 0 & 1 \end{bmatrix}, \begin{bmatrix} 9 & 0 \\ 0 & 1 \end{bmatrix}, \begin{bmatrix} 33 & 0 \\ 0 & 1 \end{bmatrix}, \begin{bmatrix} 3 & 0 \\ 0 & 3 \end{bmatrix}, \begin{bmatrix} 3 & 0 \\ 0 & 11 \end{bmatrix}, \begin{bmatrix} 1 & 0 \\ 0 & 31 \end{bmatrix}, \begin{bmatrix} 3 & 0 \\ 10 & 1 \end{bmatrix}$} & \multirow{2}*{$\ZZ / 2 \ZZ$} & \multirow{2}*{$2$} & \multirow{2}*{\texttt{40.36.1.1}} \\
    & & & & & & & & & & \\
    \cline{1-1}
    \cline{5-5}
    \cline{7-11}
    \multirow{2}*{$5$} & & & & \multirow{2}*{$\operatorname{H}_{18}$} & & \multirow{2}*{$64 \cdot \frac{(t^{2}-2)^{3}}{t^{2}}$} & \multirow{2}*{$\begin{bmatrix} 3 & 0 \\ 0 & 1 \end{bmatrix}, \begin{bmatrix} 11 & 0 \\ 0 & 1 \end{bmatrix}, \begin{bmatrix} 1 & 2 \\ 0 & 1 \end{bmatrix}, \begin{bmatrix} 21 & 1 \\ 0 & 1 \end{bmatrix}, \begin{bmatrix} 1 & 0 \\ 0 & 3 \end{bmatrix}, \begin{bmatrix} 1 & 0 \\ 0 & 11 \end{bmatrix}, \begin{bmatrix} 1 & 31 \\ 0 & 31 \end{bmatrix}, \begin{bmatrix} 11 & 1 \\ 10 & 1 \end{bmatrix}$} & \multirow{2}*{$\ZZ / 2 \ZZ \times \ZZ / 2 \ZZ$} & \multirow{2}*{$4$} & \multirow{2}*{\texttt{40.36.1.3}} \\
    & & & & & & & & & & \\
    \cline{1-1}
    \cline{3-3}
    \cline{5-11}
    \multirow{2}*{$6$} & & \multirow{8}*{$3$} & & \multirow{2}*{$\operatorname{H}_{11}$} & \multirow{8}*{$\frac{(s+3)^{3}(s+27)}{s}$} & \multirow{2}*{$\frac{(t-4)^{3}(t+4)^{3}}{t^{2}}$} & \multirow{2}*{$\begin{bmatrix} 1 & 2 \\ 0 & 1 \end{bmatrix}, \begin{bmatrix} 5 & 0 \\ 0 & 1 \end{bmatrix}, \begin{bmatrix} 7 & 3 \\ 0 & 1 \end{bmatrix}, \begin{bmatrix} 1 & 0 \\ 0 & 5 \end{bmatrix}, \begin{bmatrix} 1 & 0 \\ 6 & 7 \end{bmatrix}$} & \multirow{2}*{$\ZZ / 2 \ZZ \times \ZZ / 4 \ZZ$} & \multirow{2}*{$4$} & \multirow{2}*{\texttt{12.24.1.6}} \\
    & & & & & & & & & & \\
    \cline{1-1}
    \cline{5-5}
    \cline{7-11}
    \multirow{2}*{$7$} & & & & \multirow{2}*{$\operatorname{H}_{12}$} & & \multirow{2}*{$\frac{(t^{2}+16)^{3}}{t^{2}}$} & \multirow{2}*{$\begin{bmatrix} 5 & 0 \\ 0 & 1 \end{bmatrix}, \begin{bmatrix} 1 & 4 \\ 0 & 1 \end{bmatrix}, \begin{bmatrix} 1 & 6 \\ 0 & 1 \end{bmatrix}, \begin{bmatrix} 11 & 4 \\ 0 & 1 \end{bmatrix}, \begin{bmatrix} 1 & 0 \\ 0 & 7 \end{bmatrix}, \begin{bmatrix} 1 & 0 \\ 0 & 5 \end{bmatrix}, \begin{bmatrix} 1 & 1 \\ 6 & 7 \end{bmatrix}$} & \multirow{2}*{$\ZZ / 2 \ZZ \times \ZZ / 2 \ZZ$} & \multirow{2}*{$4$} & \multirow{2}*{\texttt{12.24.1.5}} \\
    & & & & & & & & & & \\
    \cline{1-1}
    \cline{5-5}
    \cline{7-11}
    \multirow{2}*{$8$} & & & & \multirow{2}*{$\operatorname{H}_{14}$} & & \multirow{2}*{$4 \cdot \frac{(t^{2}+8)^{3}}{t^{2}}$} & \multirow{2}*{$\begin{bmatrix} 17 & 0 \\ 0 & 1 \end{bmatrix}, \begin{bmatrix} 1 & 0 \\ 0 & 17 \end{bmatrix}, \begin{bmatrix} 1 & 8 \\ 0 & 1 \end{bmatrix}, \begin{bmatrix} 7 & 8 \\ 0 & 1 \end{bmatrix}, \begin{bmatrix} 17 & 14 \\ 0 & 1 \end{bmatrix}, \begin{bmatrix} 13 & 17 \\ 0 & 1 \end{bmatrix}, \begin{bmatrix} 1 & 0 \\ 0 & 7 \end{bmatrix}, \begin{bmatrix} 1 & 1 \\ 6 & 7 \end{bmatrix}$} & \multirow{2}*{$\ZZ / 2 \ZZ \times \ZZ / 2 \ZZ$} & \multirow{2}*{$4$} & \multirow{2}*{\texttt{24.24.1.21}} \\
    & & & & & & & & & & \\
    \cline{1-1}
    \cline{5-5}
    \cline{7-11}
    \multirow{2}*{$9$} & & & & \multirow{2}*{$\operatorname{H}_{18}$} & & \multirow{2}*{$64 \cdot \frac{(t^{2}-2)^{3}}{t^{2}}$} & \multirow{2}*{$\begin{bmatrix} 17 & 0 \\ 0 & 1 \end{bmatrix}, \begin{bmatrix} 1 & 8 \\ 0 & 1 \end{bmatrix}, \begin{bmatrix} 7 & 7 \\ 0 & 1 \end{bmatrix}, \begin{bmatrix} 5 & 19 \\ 0 & 1 \end{bmatrix}, \begin{bmatrix} 1 & 0 \\ 6 & 7 \end{bmatrix}, \begin{bmatrix} 1 & 0 \\ 0 & 17 \end{bmatrix}$} & \multirow{2}*{$\ZZ / 2 \ZZ \times \ZZ / 2 \ZZ$} & \multirow{2}*{$4$} & \multirow{2}*{\texttt{24.24.1.22}} \\
    & & & & & & & & & & \\
    \hline
    \multirow{2}*{$10$} & \multirow{4}*{$R_{6}$} & \multirow{4}*{$3$} & \multirow{4}*{$\left\{ \begin{bmatrix} \ast & 0 \\ 0 & \ast \end{bmatrix} \right\}$} & \multirow{2}*{$\operatorname{H}_{9}$} & \multirow{4}*{$\frac{s^3(s+6)^3(s^2-6s+36)^3}{(s-3)^3(s^2+3s+9)^3}$} & \multirow{2}*{$\frac{(t^{2}+48)^{3}}{t^{2}+64}$} & \multirow{2}*{$\begin{bmatrix} 5 & 0 \\ 0 & 1 \end{bmatrix}, \begin{bmatrix} 1 & 3 \\ 0 & 1 \end{bmatrix}, \begin{bmatrix} 1 & 0 \\ 0 & 5 \end{bmatrix}, \begin{bmatrix} 7 & 0 \\ 0 & 7 \end{bmatrix}, \begin{bmatrix} 7 & 0 \\ 6 & 1 \end{bmatrix}$} & \multirow{2}*{$\ZZ / 2 \ZZ$} & \multirow{2}*{$2$} & \multirow{2}*{\texttt{12.72.1.2}} \\
    & & & & & & & & & & \\
    \cline{1-1}
    \cline{5-5}
    \cline{7-11}
    \multirow{2}*{$11$} & & & & \multirow{2}*{$\operatorname{H}_{15}$} & & \multirow{2}*{$64 \cdot \frac{(t^{2}+6)^{3}}{t^{2}+8}$} & \multirow{2}*{$\begin{bmatrix} 19 & 0 \\ 0 & 1 \end{bmatrix}, \begin{bmatrix} 1 & 3 \\ 0 & 1 \end{bmatrix}, \begin{bmatrix} 11 & 0 \\ 0 & 1 \end{bmatrix}, \begin{bmatrix} 5 & 0 \\ 0 & 5 \end{bmatrix}, \begin{bmatrix} 5 & 0 \\ 0 & 7 \end{bmatrix}, \begin{bmatrix} 5 & 0 \\ 0 & 13 \end{bmatrix}, \begin{bmatrix} 5 & 0 \\ 6 & 1 \end{bmatrix}$} & \multirow{2}*{$\ZZ / 2 \ZZ$} & \multirow{2}*{$2$} & \multirow{2}*{\texttt{24.72.1.2}} \\
    & & & & & & & & & & \\
    \hline
    \multirow{2}*{$12$} & \multirow{16}*{$L_{2}(p)$} & \multirow{6}*{$13$} & \multirow{16}*{$\left\{\begin{bmatrix} \ast & \ast \\ 0 & \ast \end{bmatrix} \right\}$} & \multirow{2}*{$\operatorname{H}_{2}$} & \multirow{6}*{$\frac{(s^{2}+5s+13)(s^{4}+7s^{3}+20s^{2}+19s+1)^{3}}{s}$} & \multirow{2}*{$t^{2}+1728$} & \multirow{2}*{$\begin{bmatrix} 3 & 0 \\ 0 & 1 \end{bmatrix}, \begin{bmatrix} 5 & 0 \\ 0 & 1 \end{bmatrix}, \begin{bmatrix} 1 & 2 \\ 0 & 1 \end{bmatrix}, \begin{bmatrix} 1 & 0 \\ 0 & 3 \end{bmatrix}, \begin{bmatrix} 1 & 0 \\ 0 & 5 \end{bmatrix}, \begin{bmatrix} 2 & 1 \\ 13 & 1 \end{bmatrix}$} & \multirow{2}*{$\ZZ / 2 \ZZ$} & \multirow{2}*{$2$} & \multirow{2}*{\texttt{26.28.1.1}} \\
    & & & & & & & & & & \\
    \cline{1-1}
    \cline{5-5}
    \cline{7-11}
    \multirow{2}*{$13$} & & & & \multirow{2}*{$\operatorname{H}_{3}$} & & \multirow{2}*{$-t^{2}+1728$} & \multirow{2}*{$\begin{bmatrix} 3 & 3 \\ 0 & 1 \end{bmatrix}, \begin{bmatrix} 5 & 0 \\ 0 & 1 \end{bmatrix}, \begin{bmatrix} 29 & 46 \\ 0 & 1 \end{bmatrix}, \begin{bmatrix} 1 & 0 \\ 0 & 5 \end{bmatrix}, \begin{bmatrix} 1 & 0 \\ 0 & 9 \end{bmatrix}, \begin{bmatrix} 4 & 1 \\ 39 & 1 \end{bmatrix}$} & \multirow{2}*{$\ZZ / 2 \ZZ$} & \multirow{2}*{$2$} & \multirow{2}*{\texttt{52.28.1.1}} \\
    & & & & & & & & & & \\
    \cline{1-1}
    \cline{5-5}
    \cline{7-11}
    \multirow{2}*{$14$} & & & & \multirow{2}*{$\operatorname{H}_{5}$} & & \multirow{2}*{$8t^{2}+1728$} & \multirow{2}*{$\begin{bmatrix} 5 & 5 \\ 0 & 1 \end{bmatrix}, \begin{bmatrix} 9 & 0 \\ 0 & 1 \end{bmatrix}, \begin{bmatrix} 17 & 0 \\ 0 & 1 \end{bmatrix}, \begin{bmatrix} 33 & 0 \\ 0 & 1 \end{bmatrix}, \begin{bmatrix} 47 & 100 \\ 0 & 1 \end{bmatrix}, \begin{bmatrix} 1 & 0 \\ 0 & 9 \end{bmatrix}, \begin{bmatrix} 1 & 0 \\ 0 & 17 \end{bmatrix}, \begin{bmatrix} 1 & 0 \\ 0 & 33 \end{bmatrix}, \begin{bmatrix} 8 & 1 \\ 65 & 1 \end{bmatrix}$} & \multirow{2}*{$\ZZ / 2 \ZZ$} & \multirow{2}*{$2$} & \multirow{2}*{\texttt{104.28.1...}} \\
    & & & & & & & & & & \\
    \cline{1-1}
    \cline{3-3}
    \cline{5-11}
    \multirow{2}*{$15$} & & \multirow{8}*{$5$} & & \multirow{2}*{$\operatorname{H}_{2}$} & \multirow{8}*{$\frac{(s^{2}+10s+5)^{3}}{s}$} & \multirow{2}*{$t^{2}+1728$} & \multirow{2}*{$\begin{bmatrix} 3 & 0 \\ 0 & 1 \end{bmatrix}, \begin{bmatrix} 1 & 2 \\ 0 & 1 \end{bmatrix}, \begin{bmatrix} 1 & 0 \\ 0 & 3 \end{bmatrix}, \begin{bmatrix} 3 & 1 \\ 5 & 2 \end{bmatrix}$} & \multirow{2}*{$\ZZ / 2 \ZZ$} & \multirow{2}*{$2$} & \multirow{2}*{\texttt{10.12.1.1}} \\
    & & & & & & & & & & \\
    \cline{1-1}
    \cline{5-5}
    \cline{7-11}
    \multirow{2}*{$16$} & & & & \multirow{2}*{$\operatorname{H}_{3}$} & & \multirow{2}*{$-t^{2}+1728$} & \multirow{2}*{$\begin{bmatrix} 3 & 3 \\ 0 & 1 \end{bmatrix}, \begin{bmatrix} 9 & 0 \\ 0 & 1 \end{bmatrix}, \begin{bmatrix} 13 & 0 \\ 0 & 1 \end{bmatrix}, \begin{bmatrix} 17 & 6 \\ 0 & 1 \end{bmatrix}, \begin{bmatrix} 1 & 0 \\ 0 & 9 \end{bmatrix}, \begin{bmatrix} 1 & 0 \\ 0 & 13 \end{bmatrix}, \begin{bmatrix} 3 & 1 \\ 10 & 19 \end{bmatrix}, \begin{bmatrix} 4 & 1 \\ 15 & 1 \end{bmatrix}$} & \multirow{2}*{$\ZZ / 2 \ZZ$} & \multirow{2}*{$2$} & \multirow{2}*{\texttt{20.12.1.1}} \\
    & & & & & & & & & & \\
    \cline{1-1}
    \cline{5-5}
    \cline{7-11}
    \multirow{2}*{$17$} & & & & \multirow{2}*{$\operatorname{H}_{5}$} & & \multirow{2}*{$8t^{2}+1728$} & \multirow{2}*{$\begin{bmatrix} 31 & 0 \\ 0 & 1 \end{bmatrix}, \begin{bmatrix} 9 & 0 \\ 0 & 1 \end{bmatrix}, \begin{bmatrix} 3 & 3 \\ 0 & 1 \end{bmatrix}, \begin{bmatrix} 33 & 0 \\ 0 & 1 \end{bmatrix}, \begin{bmatrix} 3 & 0 \\ 0 & 3 \end{bmatrix}, \begin{bmatrix} 3 & 0 \\ 0 & 11 \end{bmatrix}, \begin{bmatrix} 3 & 0 \\ 0 & 21 \end{bmatrix}, \begin{bmatrix} 3 & 0 \\ 5 & 1 \end{bmatrix}$} & \multirow{2}*{$\ZZ / 2 \ZZ$} & \multirow{2}*{$2$} & \multirow{2}*{\texttt{40.12.1.6}} \\
    & & & & & & & & & & \\
    \cline{1-1}
    \cline{5-5}
    \cline{7-11}
    \multirow{2}*{$18$} & & & & \multirow{2}*{$\operatorname{H}_{7}$} & & \multirow{2}*{$\frac{32t-4}{t^{4}}$} & \multirow{2}*{$\begin{bmatrix} 9 & 0 \\ 0 & 1 \end{bmatrix}, \begin{bmatrix} 3 & 3 \\ 0 & 1 \end{bmatrix}, \begin{bmatrix} 13 & 0 \\ 0 & 1 \end{bmatrix}, \begin{bmatrix} 3 & 0 \\ 0 & 3 \end{bmatrix}, \begin{bmatrix} 3 & 0 \\ 0 & 11 \end{bmatrix}, \begin{bmatrix} 3 & 0 \\ 5 & 1 \end{bmatrix}, \begin{bmatrix} 18 & 1 \\ 15 & 1 \end{bmatrix}$} & \multirow{2}*{$\ZZ / 4 \ZZ$} & \multirow{2}*{$2$} & \multirow{2}*{\texttt{20.24.1.1}} \\
    & & & & & & & & & & \\
    \cline{1-1}
    \cline{3-3}
    \cline{5-11}
    \multirow{2}*{$19$} & & \multirow{2}*{$3$} & & \multirow{2}*{$\operatorname{H}_{7}$} & \multirow{2}*{$\frac{(s+3)^{3}(s+27)}{s}$} & \multirow{2}*{$\frac{32t-4}{t^{4}}$} & \multirow{2}*{$\begin{bmatrix} 5 & 0 \\ 0 & 1 \end{bmatrix}, \begin{bmatrix} 1 & 4 \\ 0 & 1 \end{bmatrix}, \begin{bmatrix} 7 & 7 \\ 0 & 1 \end{bmatrix}, \begin{bmatrix} 1 & 0 \\ 0 & 5 \end{bmatrix}, \begin{bmatrix} 7 & 0 \\ 0 & 7 \end{bmatrix}, \begin{bmatrix} 7 & 2 \\ 3 & 1 \end{bmatrix}$} & \multirow{2}*{$\ZZ / 8 \ZZ$} & \multirow{2}*{$2$} & \multirow{2}*{\texttt{12.16.1.1}} \\
    & & & & & & & & & & \\
    \cline{1-1}
    \cline{7-11}
    \hline
    \multirow{2}*{$20$} & \multirow{8}*{$L_{3}(p^{2})$} & \multirow{8}*{$3$} & \multirow{8}*{$\left\{\begin{bmatrix} \ast & 0 \\ 0 & \ast \end{bmatrix} \right\}$} & \multirow{2}*{$\operatorname{H}_{2}$} & \multirow{8}*{$\frac{s^3(s+6)^3(s^2-6s+36)^3}{(s-3)^3(s^2+3s+9)^3}$} & \multirow{2}*{$t^{2}+1728$} & \multirow{2}*{$\begin{bmatrix} 5 & 0 \\ 0 & 1 \end{bmatrix}, \begin{bmatrix} 1 & 0 \\ 0 & 5 \end{bmatrix}, \begin{bmatrix} 2 & 3 \\ 3 & 1 \end{bmatrix}$} & \multirow{2}*{$\ZZ / 2 \ZZ$} & \multirow{2}*{$2$} & \multirow{2}*{\texttt{6.24.1.2}} \\
    & & & & & & & & & & \\
    \cline{1-1}
    \cline{5-5}
    \cline{7-11}
    \multirow{2}*{$21$} & & & & \multirow{2}*{$\operatorname{H}_{3}$} & & \multirow{2}*{$-t^{2}+1728$} & \multirow{2}*{$\begin{bmatrix} 5 & 0 \\ 0 & 1 \end{bmatrix}, \begin{bmatrix} 7 & 9 \\ 0 & 1 \end{bmatrix}, \begin{bmatrix} 7 & 3 \\ 0 & 1 \end{bmatrix}, \begin{bmatrix} 1 & 0 \\ 0 & 5 \end{bmatrix}, \begin{bmatrix} 7 & 0 \\ 0 & 7 \end{bmatrix}, \begin{bmatrix} 7 & 0 \\ 3 & 1 \end{bmatrix}$} & \multirow{2}*{$\ZZ / 2 \ZZ$} & \multirow{2}*{$2$} & \multirow{2}*{\texttt{12.24.1.3}} \\
    & & & & & & & & & & \\
    \cline{1-1}
    \cline{5-5}
    \cline{7-11}
    \multirow{2}*{$22$} & & & & \multirow{2}*{$\operatorname{H}_{4}$} & & \multirow{2}*{$-2t^{2}+1728$} & \multirow{2}*{$\begin{bmatrix} 19 & 0 \\ 0 & 1 \end{bmatrix}, \begin{bmatrix} 11 & 0 \\ 0 & 1 \end{bmatrix}, \begin{bmatrix} 5 & 15 \\ 0 & 1 \end{bmatrix}, \begin{bmatrix} 5 & 0 \\ 0 & 5 \end{bmatrix}, \begin{bmatrix} 5 & 0 \\ 0 & 7 \end{bmatrix}, \begin{bmatrix} 5 & 0 \\ 0 & 13 \end{bmatrix}, \begin{bmatrix} 5 & 0 \\ 3 & 1 \end{bmatrix}$} & \multirow{2}*{$\ZZ / 2 \ZZ$} & \multirow{2}*{$2$} & \multirow{2}*{\texttt{24.24.1.1}} \\
    & & & & & & & & & & \\
    \cline{1-1}
    \cline{5-5}
    \cline{7-11}
    \multirow{2}*{$23$} & & & & \multirow{2}*{$\operatorname{H}_{5}$} & & \multirow{2}*{$8t^{2}+1728$} & \multirow{2}*{$\begin{bmatrix} 7 & 0 \\ 0 & 1 \end{bmatrix}, \begin{bmatrix} 17 & 0 \\ 0 & 1 \end{bmatrix}, \begin{bmatrix} 5 & 15 \\ 0 & 1 \end{bmatrix}, \begin{bmatrix} 5 & 0 \\ 0 & 5 \end{bmatrix}, \begin{bmatrix} 1 & 0 \\ 0 & 7 \end{bmatrix}, \begin{bmatrix} 5 & 0 \\ 0 & 13 \end{bmatrix}, \begin{bmatrix} 5 & 0 \\ 3 & 1 \end{bmatrix}$} & \multirow{2}*{$\ZZ \times \ZZ / 2 \ZZ$} & \multirow{2}*{$2$} & \multirow{2}*{\texttt{24.24.1.2}} \\
    & & & & & & & & & & \\
\hline
    \end{tabular}}
\caption{Modular curves of genus 1}
\label{Modular curves of genus 1}
\end{table}
\end{center}

\newpage

\section{Elliptic curves}\label{elliptic curves}

The elliptic curves in Table \ref{Modular curves of genus 1} are of interest to our work. We show how we classify the elliptic curves in the table. For example, let $B_{5} = \left\{ \begin{bmatrix} \ast & \ast \\ 0 & \ast \end{bmatrix} \right\} \subseteq \operatorname{GL}(2, \ZZ / 5 \ZZ)$ and let $\operatorname{X}$ be the modular curve generated by $\operatorname{H}_{9} \times B_{5}$. The \textit{j}-invariant associated to $\operatorname{H}_{9}$ is equal to $\frac{(t^{2}+48z^{2})^{3}}{(t^{2}+64z^{2})z^{4}}$ and the \textit{j}-invariant associated to $B_{5}$ is equal to $\frac{(s^{2}+10sz+5z^{2})^{3}}{sz^{5}}$. Equating the \textit{j}-invariants, we get a curve defined by
$$s \cdot z \cdot (t^{2}+48z^{2})^{3} - (t^{2}+64z^{2}) \cdot (s^{2}+10sz+5z^{2})^{3} = 0.$$
We can use the command \texttt{E,map:=EllipticCurve(C)} to find the elliptic curve $E$ associated to $\operatorname{X}$ and finally, \texttt{MordellWeilGroup(E)} in Magma to compute $\operatorname{X}(\QQ) = E(\QQ)$. The LMFDB has the number of rational cusps on each modular curve in Table \ref{Modular curves of genus 1}, computed by \texttt{GL2RationalCuspCount(H)} from the code attached to \cite{Rouse2021elladicIO}. If the number of rational cusps on $\operatorname{X}$ in Table \ref{Modular curves of genus 1} is equal to the cardinality of $\operatorname{X}(\QQ)$, then there are no elliptic curves $E$ over $\QQ$ such that $\rho_{E,2^{\infty}}(G_{\QQ})$ and $\overline{\rho}_{E,p}(G_{\QQ})$ is conjugate to the corresponding entries in the table. For $p = 3$ or $5$, let $\operatorname{B}_{0}(p)$ denote the group of all upper-triangular of $\operatorname{GL}(2, \ZZ / p \ZZ)$ and let $\mathcal{C}_{sp}(3)$ denote the group of all diagonal matrices in $\operatorname{GL}(2, \ZZ / 3 \ZZ)$. It remains to investigate the following modular curves.

\begin{itemize}
    \item $\operatorname{B}_{3} \times \operatorname{H}_{11}$
    
    The modular curve $\operatorname{X}$ generated by $\operatorname{B}_{3} \times \operatorname{H}_{11}$ has LMFDB label \texttt{}{12.24.1.6} and is isomorphic to the elliptic curve $E : y^{2} = x^{3} + x^{2} - 24x + 36$. Moreover,
    $$E(\QQ) = \left\{\mathcal{O} = (0,1,0), (-6,0), (0,\pm 6), (2,0), (3,0), (6, \pm 12) \right\}.$$
    Plugging in each of the eight rational points into the \textit{j}-invariant formula given by the LMFDB, we see that four of the rational points are cusps and the rest are CM points, two of which map to $\textit{j} = 0$ and the other two map to $\textit{j} = 54000$.
    
    \item $\operatorname{B}_{5} \times \operatorname{H}_{7}$
    
    Let $P_{5}$ be the subgroup of $\operatorname{GL}(2, \ZZ / 5 \ZZ)$ consisting of all matrices of the form $\begin{bmatrix} 1 & r \\ 0 & s \end{bmatrix}$. The modular curve $\operatorname{X}$ has LMFDB label \texttt{20.24.1.1} is isomorphic to the elliptic curve $E : y^{2} = x^{3} + 13x + 34$. Moreover, $E(\QQ)$ is a group of size $4$.
    
    The LMFDB says that $E$ has two cusps and that $\frac{1026895}{1024}$ and $\frac{-1723025}{4}$ are \textit{j}-invariants that correspond to non-CM elliptic curves $E/\QQ$ such that $\overline{\rho}_{E,5}(G_{\QQ})$ is conjugate to $B_{5}$ and $\rho_{E,2^{\infty}}(G_{\QQ})$ is conjugate to $\operatorname{H}_{7}$ or their quadratic twists. As $P_{5}$ is a subgroup of $B_{5}$ of index $6$, it is not a quadratic twist of $B_{5}$.
    
    \item $\operatorname{B}_{3} \times \operatorname{H}_{7}$
    
    The modular curve $\operatorname{X}$ has LMFDB label \texttt{12.16.1.1} is isomorphic to the elliptic curve $E : y^{2} = x^{3} + x^{2} + 16x + 180$. Moreover,
    $$E(\QQ) = \left\{(-2,-12,1), (-2,12,1), (22,108,1), (4,-18,1), (-5,0,1), (22,-108,1), \mathcal{O}, (4,18,1)\right\}.$$
    The rational point $(4,-18,1)$ is mapped to the \textit{j}-invariant $\frac{-35937}{4}$ and the rational point $(-5,0,1)$ is mapped to the \textit{j}-invariant $\frac{109503}{64}$. These two \textit{j}-invariants correspond to non-CM elliptic curves $E/\QQ$ such that $E$ has a point of order $3$ defined over $\QQ$ and $\rho_{E,2^{\infty}}(G_{\QQ})$ is conjugate to $\operatorname{H}_{20}$ or their quadratic twists. Remember that the reason why we are investigating the modular curve \texttt{12.16.1.1} is to find non-CM elliptic curves $E'/\QQ$ such that $\rho_{E',2^{\infty}}(G_{\QQ})$ is conjugate to $\operatorname{H}_{7}$. The group $\operatorname{H}_{20}$ is in fact a subgroup of $\operatorname{H}_{7}$ of index $2$ but is not a quadratic twist of $\operatorname{H}_{7}$. The rational points $(-2:-12:1)$, $(-2:12:1)$, $(22:108:1)$ all map to $\textit{j} = 0$ and the rational point $(22:-108:1)$ maps to $\textit{j} = -12288000$, in other words, they are CM points. The remaining two rational points are cusps.
    
    \item $\mathcal{C}_{sp}(3) \times \operatorname{H}_{5}$
    
    See the example \texttt{[8X5, 3Nn]} in Section 5 of \cite{Daniels2018SerresCO} where the authors similarly conclude that there is no non-CM elliptic curve $E / \QQ$ such that $\overline{\rho}_{E,24}(G_{\QQ})$ is conjugate to the group $\mathcal{C}_{sp}(3) \times \operatorname{H}_{5}$. We present an alternate proof.
    
    The modular curve $\operatorname{X}$ has LMFDB label \texttt{24.24.1.2} and is isomorphic to the elliptic curve $E : y^{2} = x^{3} - 216$. Note that $\operatorname{H}$ contains group $\operatorname{H}' = \mathcal{C}_{sp}(3) \times \operatorname{H}_{17}$ which generates the modular curve $\operatorname{X}'$. Moreover, $\operatorname{X}'$ is isomorphic to the elliptic curve $E' : y^{2} = x^{3} + 8$. The \textit{j}-invariant map associated to $E$ is $\displaystyle \textit{j} = \frac{1}{2^3}\cdot\frac{(y^{2}+216z^{2})(y^{2}+1944z^{2})^{3}}{z^{2}y^{6}}$
    and the \textit{j}-invariant map associated to $E'$ is $\textit{j}' = \displaystyle \frac{1}{2^3}\cdot\frac{(y^{2}+24z^{2})^{3}(y^{6}+1800y^{4}z^{2}-25920y^{2}z^{4}+124416z^{6})^{3}}{z^{2}y^{6}(y^{2}-72z^{2})^{6}(y^{2}-8z^{2})^{2}}$. Let $\phi \colon E' \to E$ be the map such that $\displaystyle \phi(a,b,c) = (a \cdot (b^{2}+24c^{2}), b \cdot (b^{2}-72c^{2}), c \cdot (b^{2}-8c^{2}))$. The reader may check that
    the following diagram in fact commutes
    \begin{center}
        \begin{tikzcd}
E' \arrow[rr, "\phi"] \arrow[rrdd, "\textit{j}'"'] &  & E \arrow[dd, "\textit{j}"] \\
                                                   &  &                            \\
                                                   &  & \mathbb{P}^{1}(\mathbb{Q})
\end{tikzcd}
    \end{center}
    
Using the command \texttt{Generators(E')} and \texttt{Generators(E)}, we compute that $E'(\QQ) = \left\langle (-2,0,1), (2,-4,1) \right\rangle$ and $E(\QQ) = \left\langle (6,0,1), (10,28,1) \right\rangle$. Note that $\phi$ is a rational morphism which maps the identity of $E'$ to the identity of $E$, hence, is an isogeny of degree $3$ and a group homomorphism. Note further that $\phi(-2,0,1) = (-48,0,-8) = (6,0,1)$ and $\phi(2,-4,1) = (80,224,8) = (10,28,1)$ and hence, the restriction of $\phi$ on $E'(\QQ)$ is surjective. To prove that the restriction of $\phi$ on $E'(\QQ)$ is injective, note that $\phi(a,b,c) = \mathcal{O}$ if and only if $c \cdot (b^{2}-8c^{2}) = 0$ and as $8$ is not a rational square, we must have $c = 0$. In other words, the rational points on $E$ are in one-to-one bijective correspondence with the rational points on $E'$. This means that all rational points on the modular curve $\operatorname{X}$ are actually rational points on $\operatorname{X}'$ and all non-CM elliptic curves $E/\QQ$ such that $\overline{\rho}_{E,24}(G_{\QQ})$ is conjugate to a subgroup of $\operatorname{H}$ are actually the same as an elliptic curve $E'/\QQ$ such that $\overline{\rho}_{E',24}(G_{\QQ})$ is conjugate to a subgroup of $\operatorname{H}'$. In this sense, $\operatorname{H}$ is something of an ``inefficient'' Galois group with all of the work done by the Galois group $\operatorname{H}'$. 
\end{itemize}

\newpage

\section{Modular curves of genus $2$}\label{hyperelliptic curves}

\begin{center}
\begin{table}[h!]
\renewcommand{\arraystretch}{1.3}
\scalebox{0.57}{
    \begin{tabular}{|c|c|c|c|c|c|c|c|c|c|}
    \hline
    \# & graph & $p$ & $\overline{\rho}_{E,p}(G_{\QQ})$ & $\rho_{E,2^{\infty}}(G_{\QQ})$ & $\textit{j}_{p}(E)$ & $\textit{j}_{2^{\infty}}(G_{\QQ})$ & $\operatorname{Rank}(\operatorname{J}_{\operatorname{X}'}(\QQ))$ & $f(x)$ & LMFDB \\
    \hline
    \multirow{2}*{$1$} & \multirow{4}*{$R_{4}(2p)$} & \multirow{8}*{$3$} & \multirow{8}*{$\left\{\begin{bmatrix} 1 & \ast \\ 0 & \ast \end{bmatrix}\right\}$} & \multirow{2}*{$\operatorname{H}_{9a}$} & \multirow{8}*{$\frac{(s+3)^{3}(s+27)}{s}$} & \multirow{4}*{$\frac{(t^{2}+48)^{3}}{t^{2}+64}$} & \multirow{8}*{$0$} & \multirow{2}*{$x^{5}+10x^{3}+9x$} & \multirow{2}*{\texttt{24.96.2.100}} \\
    & & & & & & & & & \\
    \cline{1-1}
    \cline{5-5}
    \cline{9-10}
    \multirow{2}*{$2$} & & & & \multirow{2}*{$\operatorname{H}_{9b}$} & & & & \multirow{2}*{$-2x^{5}-20x^{3}-18x$} & \multirow{2}*{\texttt{24.96.2.99}}  \\
    & & & & & & & & & \\
    \cline{1-2}
    \cline{5-5}
    \cline{7-7}
    \cline{9-10}
    \multirow{2}*{$3$} & \multirow{4}*{$S$} & & & \multirow{2}*{$\operatorname{H}_{8a}$} & & \multirow{4}*{$64 \cdot \frac{(t^{2}+3)^{3}}{(t^{2}-1)^{2}}$} & & \multirow{2}*{$ 4 x^{5} - 10 x^{4} + 10 x^{2} - 4 x$} & \multirow{2}*{\texttt{24.96.2.3}} \\
    & & & & & & & & & \\
    \cline{1-1}
    \cline{5-5}
    \cline{9-10}
    \multirow{2}*{$4$} & & & & \multirow{2}*{$\operatorname{H}_{8b}$} & & & & \multirow{2}*{$2x^{5} - 5x^{4} + 5x^{2} - 2x$} & \multirow{2}*{\texttt{12.96.2.1}} \\
    & & & & & & & & & \\
\hline
\multirow{2}*{$5$} & \multirow{10}*{$L_{2}(7)$} & \multirow{10}*{$7$} & \multirow{2}*{$\left\{\begin{bmatrix} \ast & \ast \\ 0 & \ast \end{bmatrix}\right\}$} & \multirow{2}*{$\operatorname{H}_{7}$} & \multirow{2}*{$\frac{(s^{2}+5s+1)^{3}(s^{2}+13s+49)}{s}$} & \multirow{2}*{$\frac{32t-4}{t^{4}}$} & \multirow{2}*{$1$} & \multirow{2}*{$ x^{6} + 2 x^{5} - 4 x^{4} + 4 x^{3} - 4 x^{2} + 2 x + 1$} & \multirow{2}*{\texttt{28.32.2.1}} \\
& & & & & & & & & \\
\cline{1-1}
\cline{4-10}
\multirow{2}*{$6$} & & & \multirow{8}*{$\left\{\begin{bmatrix} 1 & \ast \\ 0 & \ast \end{bmatrix}\right\}$} & \multirow{2}*{$\operatorname{H}_{2}$} & \multirow{8}*{$\frac{(s^{2}-s+1)^{3}(s^{6} - 11s^{5} + 30s^{4} - 15s^{3} - 10s^{2} + 5s + 1)^{3}}{(s-1)^{7}s^{7}(s^{3}-8s^{2}+5s+1)}$} & \multirow{2}*{$t^{2}+1728$} & \multirow{8}*{$0$} & \multirow{2}*{$x^{5} - 9 x^{4} + 13 x^{3} - 4 x^{2} - x$} & \multirow{2}*{\texttt{14.96.2.1}} \\
& & & & & & & & & \\
\cline{1-1}
\cline{5-5}
\cline{7-7}
\cline{9-10}
\multirow{2}*{$7$} & & & & \multirow{2}*{$\operatorname{H}_{3}$} & & \multirow{2}*{$-t^{2}+1728$} & & \multirow{2}*{$-x^{5} + 9x^{4} - 13x^{3} + 4x^{2} + x$} & \multirow{2}*{\texttt{28.96.2.7}} \\
& & & & & & & & & \\
\cline{1-1}
\cline{5-5}
\cline{7-7}
\cline{9-10}
\multirow{2}*{$8$} & & & & \multirow{2}*{$\operatorname{H}_{4}$} & & \multirow{2}*{$-2t^{2}+1728$} & & \multirow{2}*{$-2x^{5} + 18x^{4} - 26x^{3} + 8x^{2} + 2x$} & \multirow{2}*{\texttt{56.96.2.6}}\\
& & & & & & & & & \\
\cline{1-1}
\cline{5-5}
\cline{7-7}
\cline{9-10}
\multirow{2}*{$9$} & & & & \multirow{2}*{$\operatorname{H}_{5}$} & & \multirow{2}*{$8t^{2}+1728$} & & \multirow{2}*{$2x^{5} - 18x^{4} + 26x^{3} - 8x^{2} - 2x$} &  \multirow{2}*{\texttt{56.96.2.3}} \\
& & & & & & & & & \\
\hline
    \end{tabular}}
\caption{Modular curves of genus 2}
\label{Modular curves of genus 2}
\end{table}
\end{center}

Consider the modular curves $\operatorname{X}$ generated by the product group $\operatorname{H} = \overline{\rho}_{E,p}(G_{\QQ}) \times \rho_{E,2^{\infty}}(G_{\QQ})$ in Table \ref{Modular curves of genus 2}. Let $\operatorname{H}' = \left\langle \operatorname{H}, \operatorname{-Id} \right\rangle$ and let $\operatorname{X}'$ be the modular curve generated by $\operatorname{H}'$. Then $\operatorname{X}'$ is a curve of genus $2$ and is isomorphic to the hyperelliptic curve $y^{2} = f(x)$ with $f(x)$ being the corresponding polynomial in the table. The polynomial $f(x)$ of the first four curves in the table were taken from the LMFDB. For the remaining, curves, the equation $f(x)$ was found by taking the projective closure of the curve generated by the two respective $\textit{j}$-invariants respectively and then using the command \texttt{IsHyperelliptic} on Magma. Let $\operatorname{J}'$ be the Jacobian variety of $\operatorname{X}'$. If $\operatorname{Rank}(\operatorname{J}'(\QQ)) = 0$, then we can use the command $\texttt{Chabauty0}(\operatorname{J}')$ to find all rational points on $\operatorname{X}'$. In all cases when the rank of the Jacobian variety of $\operatorname{X}'$ is equal to $0$, all of the rational points on $\operatorname{X}'$ are cusps.

Let $\operatorname{B}_{7}$ be the group of all upper-triangular matrices in $\operatorname{GL}(2, \ZZ / 7 \ZZ)$. Let $E_{1}/\QQ$ be the elliptic curve with LMFDB label \texttt{338.c1} and let $E_{2}/\QQ$ be the elliptic curve with LMFDB label \texttt{338.c2}. Then the \textit{j}-invariant of $E_{1}$ is equal to $-\frac{38575685889}{16384}$ and the \textit{j}-invariant of $E_{2}$ is equal to $\frac{351}{4}$. Moreover, both $\rho_{E_{1},2^{\infty}}(G_{\QQ})$ and $\rho_{E_{2},2^{\infty}}(G_{\QQ})$ are conjugate to $\operatorname{H}_{20}$ and $\overline{\rho}_{E_{1},7}(G_{\QQ})$ and $\overline{\rho}_{E_{2},7}(G_{\QQ})$ are both conjugate to $\operatorname{B}_{7}$. Note that even though $\operatorname{H}_{20}$ is a subgroup of $\operatorname{H}_{7}$, it is not a quadratic twist of $\operatorname{H}_{7}$. Let $\operatorname{X}$ be the modular curve generated by $\operatorname{H}_{7} \times \operatorname{B}_{7}$. Using a similar analysis to study the rational points on the modular curve generated by $\mathcal{N}_{sp}(5) \times \operatorname{H}_{4}$ in the proof of Proposition \ref{L3p Graphs proposition}, we see that $\operatorname{X}$ has four rational points, two of which are cusps, the other two corresponding to the \textit{j}-invariants $\frac{351}{4}$ and $-\frac{38575685889}{16384}$.

\newpage

\section{Tables}\label{section tables}

\begin{center}
\begin{table}[h!]
\renewcommand{\arraystretch}{1.3}
\scalebox{0.7}{
\begin{tabular}{|c|c|c|}

\hline

Isogeny-Torsion Graph & $\rho_{E,2^{\infty}}(G_{\QQ})$ & Example Isogeny Class \\
    \hline
\multirow{22}*{$([1])$} & $\operatorname{H}_{1} = \texttt{1.1.0.1}$ & \texttt{37.a} \\
\cline{2-3}
& $\operatorname{H}_{2} = \texttt{2.2.0.1}$ & \texttt{392.a} \\
\cline{2-3}
& $\operatorname{H}_{2a} = \texttt{4.4.0.2}$ & \texttt{392.c} \\
\cline{2-3}
& $\operatorname{H}_{2b} = \texttt{8.4.0.1}$ & \texttt{3136.b} \\
\cline{2-3}
& $\operatorname{H}_{3} = \texttt{4.2.0.1}$ & \texttt{567.a} \\
\cline{2-3}
& $\operatorname{H}_{3a} = \texttt{4.4.0.3}$ & \texttt{2738.b} \\
\cline{2-3}
& $\operatorname{H}_{3b} = \texttt{8.4.0.2}$ & \texttt{2116.c} \\
\cline{2-3}
& $\operatorname{H}_{4} = \texttt{8.2.0.1}$ & \texttt{200.a} \\
\cline{2-3}
& $\operatorname{H}_{5} = \texttt{8.2.0.2}$ & \texttt{1682.a} \\
\cline{2-3}
& $\operatorname{H}_{7} = \texttt{4.4.0.1}$ & \texttt{216.a} \\
\cline{2-3}
& $\operatorname{H}_{20} = \texttt{4.8.0.2}$ & \texttt{648.d} \\
\cline{2-3}
& $\operatorname{H}_{20a} = \texttt{8.16.0.3}$ & \texttt{5184.u} \\
\cline{2-3}
& $\operatorname{H}_{20b} = \texttt{4.16.0.2}$ & \texttt{648.b} \\
\cline{2-3}
& $\operatorname{H}_{22} = \texttt{8.8.0.1}$ & \texttt{44800.r} \\
\cline{2-3}
& $\operatorname{H}_{55} = \texttt{8.16.0.1}$ & \texttt{3388.f} \\
\cline{2-3}
& $\operatorname{H}_{56} = \texttt{16.16.0.1}$ & \texttt{3628800.bg} \\
\cline{2-3}
& $\operatorname{H}_{57} = \texttt{16.16.0.2}$ & \texttt{6400.b} \\
\cline{2-3}
& $\operatorname{H}_{57a} = \texttt{32.32.0.3}$ & \texttt{6400.p} \\
\cline{2-3}
& $\operatorname{H}_{57b} = \texttt{32.32.0.4}$ & \texttt{6400.c} \\
\cline{2-3}
& $\operatorname{H}_{177} = \texttt{32.32.0.2}$ & $y^{2} = x^{3} - 340750x - 76560000$ \\
\cline{2-3}
& $\operatorname{H}_{178} = \texttt{32.32.0.1}$ & \texttt{20736.e} \\
\cline{2-3}
& $\operatorname{H}_{441} = \texttt{16.64.2.1}$ & $y^{2}+y = x^{3} - 3014658660x+150916472601529$ \\
\hline
\end{tabular}}
\caption{2-adic Galois Images of $L_{1}$ Graphs}
\label{2-adic Galois Images of L1 Graphs}
\end{table}
\end{center}

\newpage

\begin{center}
\begin{table}[h!]
\renewcommand{\arraystretch}{1.3}
\scalebox{0.33}{
    \begin{tabular}{|c|c|c|c|c|}
\hline
Isogeny Graph & Isogeny-Torsion Graph & $\rho_{E_{1},2^{\infty}}(G_{\QQ})$ & $\rho_{E_{2},2^{\infty}}(G_{\QQ})$ & Example Isogeny Class \\
\hline
\multirow{80}*{\includegraphics[scale=0.1]{L22_graph.png}} & \multirow{80}*{$([2],[2])$} & $\operatorname{H}_{6} = \texttt{2.3.0.1}$ & $\operatorname{H}_{6} = \texttt{2.3.0.1}$ & \texttt{69.a} \\
\cline{3-5}
& & $\operatorname{H}_{9} = \texttt{4.6.0.3}$ & $\operatorname{H}_{10} = \texttt{4.6.0.5}$ & \texttt{325.d} \\
\cline{3-5}
& & $\operatorname{H}_{9a} = \texttt{8.12.0.24}$ & $\operatorname{H}_{10c} = \texttt{8.12.0.38}$ & \texttt{676.c} \\
\cline{3-5}
& & $\operatorname{H}_{9b} = \texttt{8.12.0.23}$ & $\operatorname{H}_{10b} = \texttt{4.12.0.6}$ & \texttt{1600.t} \\
\cline{3-5}
& & $\operatorname{H}_{9c} = \texttt{8.12.0.22}$ & $\operatorname{H}_{10a} = \texttt{8.12.0.37}$ & \texttt{52.a} \\
\cline{3-5}
& & $\operatorname{H}_{9d} = \texttt{8.12.0.21}$ & $\operatorname{H}_{10d} = \texttt{4.12.0.12}$ & \texttt{136.a} \\
\cline{3-5}
& & $\operatorname{H}_{11} = \texttt{4.6.0.4}$ & $\operatorname{H}_{11} = \texttt{4.6.0.4}$ & \texttt{2772.m} \\
\hline
\end{tabular}}
\caption{2-adic Galois Images of $L_{2}(2)$ Graphs}
\label{2-adic Galois Images of L22 Graphs}
\end{table}
\end{center}

\newpage

\begin{center}
\begin{table}[h!]
\renewcommand{\arraystretch}{1.3}
\scalebox{0.4}{
    \begin{tabular}{|c|c|c|c|c|c|c|}

\hline

Isogeny Graph & Isogeny-Torsion Graph & $\rho_{E_{1},2^{\infty}}(G_{\QQ})$ & $\rho_{E_{2},2^{\infty}}(G_{\QQ})$ & $\rho_{E_{3},2^{\infty}}(G_{\QQ})$ & $\rho_{E_{4},2^{\infty}}(G_{\QQ})$ & Example Isogeny Class \\
\hline
\multirow{60}*{\includegraphics[scale=0.1]{T4_Graph.png}} &

\multirow{7}*{$([2,2],[4],[4],[2])$} & $\operatorname{H}_{24e} = \texttt{4.24.0.7}$ & $\operatorname{H}_{27h} = \texttt{4.24.0.11}$ & $\operatorname{H}_{32l} = \texttt{8.24.0.47}$ & $\operatorname{H}_{32i} = \texttt{8.24.0.57}$ & \texttt{205a} \\
\cline{3-7}
& & $\operatorname{H}_{61h} = \texttt{8.48.0.40}$ & $\operatorname{H}_{95a} = \texttt{8.48.0.64}$ &  $\operatorname{H}_{116l} = \texttt{16.48.0.40}$ & $\operatorname{H}_{116i} = \texttt{16.48.0.56}$ & \texttt{291.b} \\
\cline{3-7}
& & $\operatorname{H}_{66e} = \texttt{8.48.0.37}$ & $\operatorname{H}_{79h} = \texttt{8.48.0.170}$ & $\operatorname{H}_{94d} = \texttt{8.48.0.62}$ & $\operatorname{H}_{79i} = \texttt{8.48.0.185}$ & \texttt{1120.i} \\
\cline{3-7}
& & $\operatorname{H}_{67c} = \texttt{8.48.0.39}$ & $\operatorname{H}_{60d} = \texttt{4.48.0.2}$ & $\operatorname{H}_{115e} = \texttt{16.48.0.39}$ & $\operatorname{H}_{115g} = \texttt{16.48.0.55}$ & \texttt{40.a} \\
\cline{3-7}
& & $\operatorname{H}_{209b} = \texttt{16.96.0.14}$ & $\operatorname{H}_{222a} = \texttt{16.96.0.31}$ & $\operatorname{H}_{241d} = \texttt{32.96.0.9}$ & $\operatorname{H}_{241c} = \texttt{32.96.0.17}$ & \texttt{514.b} \\
\cline{3-7}
& & $\operatorname{H}_{210b} = \texttt{16.96.0.13}$ & $\operatorname{H}_{221a} = \texttt{16.96.0.32}$ & $\operatorname{H}_{242d} = \texttt{32.96.0.10}$ & $\operatorname{H}_{242b} = \texttt{32.96.0.18}$ & \texttt{17.a} \\
\cline{3-7}
& & $\operatorname{H}_{214b} = \texttt{16.96.0.12}$ & $\operatorname{H}_{226d} = \texttt{16.96.0.132}$ & $\operatorname{H}_{231c} = \texttt{16.96.0.29}$ & $\operatorname{H}_{226c} = \texttt{16.96.0.176}$ & \text{84320.k} \\
\hline
\end{tabular}}
\caption{2-adic Galois Images of $T_{4}$ Graphs}
\label{2-adic Galois Images of T4 Graphs}
\end{table}
\end{center}

\newpage

\begin{center}
\begin{table}[h!]
\renewcommand{\arraystretch}{1.5}
\scalebox{0.25}{
    \begin{tabular}{|c|c|c|c|c|c|c|c|c|}
\hline

Isogeny Graph & Isogeny-Torsion Graph & $\rho_{E_{1},2^{\infty}}(G_{\QQ})$ & $\rho_{E_{2},2^{\infty}}(G_{\QQ})$ & $\rho_{E_{3},2^{\infty}}(G_{\QQ})$ & $\rho_{E_{4},2^{\infty}}(G_{\QQ})$ & $\rho_{E_{5},2^{\infty}}(G_{\QQ})$ & $\rho_{E_{6},2^{\infty}}(G_{\QQ})$ & Example Isogeny Class \\
\hline
\multirow{81}*{\includegraphics[scale=0.13]{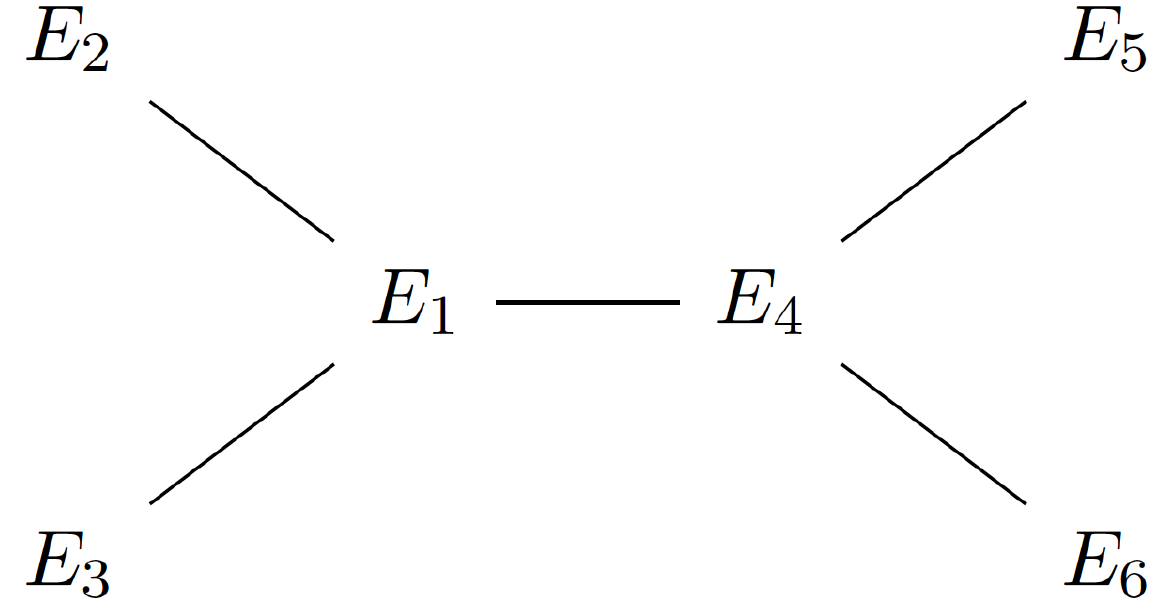}} &

\multirow{5}*{$([2,4],[8],[4],[2,2],[2],[2])$} & $\operatorname{H}_{98e} = \texttt{8.48.0.24}$ & $\operatorname{H}_{102p} = \texttt{8.48.0.159}$ & $\operatorname{H}_{120d} = \texttt{16.48.0.27}$ & $\operatorname{H}_{101d} = \texttt{8.48.0.34}$ & $\operatorname{H}_{122c} = \texttt{16.48.0.48}$ & $\operatorname{H}_{85l} = \texttt{8.48.0.177}$ & \texttt{21.a} \\
\cline{3-9}
& & $\operatorname{H}_{181e} = \texttt{8.96.0.38}$ & $\operatorname{H}_{223d} = \texttt{16.96.0.97}$ & $\operatorname{H}_{229g} = \texttt{16.96.0.99}$ & $\operatorname{H}_{203f} = \texttt{8.96.0.48}$ & $\operatorname{H}_{233a} = \texttt{16.96.0.136}$ & $\operatorname{H}_{199f} = \texttt{8.96.0.165}$ & \texttt{102.c} \\
\cline{3-9}
& & $\operatorname{H}_{188c} = \texttt{8.96.0.37}$ & $\operatorname{H}_{230e} = \texttt{16.96.0.98}$ & $\operatorname{H}_{236b} = \texttt{16.96.0.103}$ & $\operatorname{H}_{200b} = \texttt{8.96.0.50}$ & $\operatorname{H}_{202g} = \texttt{8.96.0.164}$ & $\operatorname{H}_{217c} = \texttt{16.96.0.137}$ & \texttt{42.a} \\
\cline{3-9}
& & $\operatorname{H}_{200a} = \texttt{8.96.0.33}$ & $\operatorname{H}_{202h} = \texttt{8.96.0.157}$ & $\operatorname{H}_{217b} = \texttt{16.96.0.91}$ & $\operatorname{H}_{188d} = \texttt{8.96.0.56}$ & $\operatorname{H}_{236h} = \texttt{16.96.0.147}$ & $\operatorname{H}_{230b} = \texttt{16.96.0.146}$ & \texttt{336.d} \\
\hline
\end{tabular}}
\caption{2-adic Galois Images of $T_{6}$ Graphs}
\label{2-adic Galois Images of T6 Graphs}
\end{table}
\end{center}

\newpage

\begin{center}
\begin{table}[h!]
\renewcommand{\arraystretch}{1.5}
\scalebox{0.30}{
    \begin{tabular}{|c|c|c|c|c|c|c|c|c|c|c|}
\hline

Isogeny Graph & Isogeny-Torsion Graph & $\rho_{E_{1},2^{\infty}}(G_{\QQ})$ & $\rho_{E_{2},2^{\infty}}(G_{\QQ})$ & $\rho_{E_{3},2^{\infty}}(G_{\QQ})$ & $\rho_{E_{4},2^{\infty}}(G_{\QQ})$ & $\rho_{E_{5},2^{\infty}}(G_{\QQ})$ & $\rho_{E_{6},2^{\infty}}(G_{\QQ})$ & $\rho_{E_{7},2^{\infty}}(G_{\QQ})$ & $\rho_{E_{8},2^{\infty}}(G_{\QQ})$ & Example Isogeny Class \\
\hline
\multirow{53}*{\includegraphics[scale=0.15]{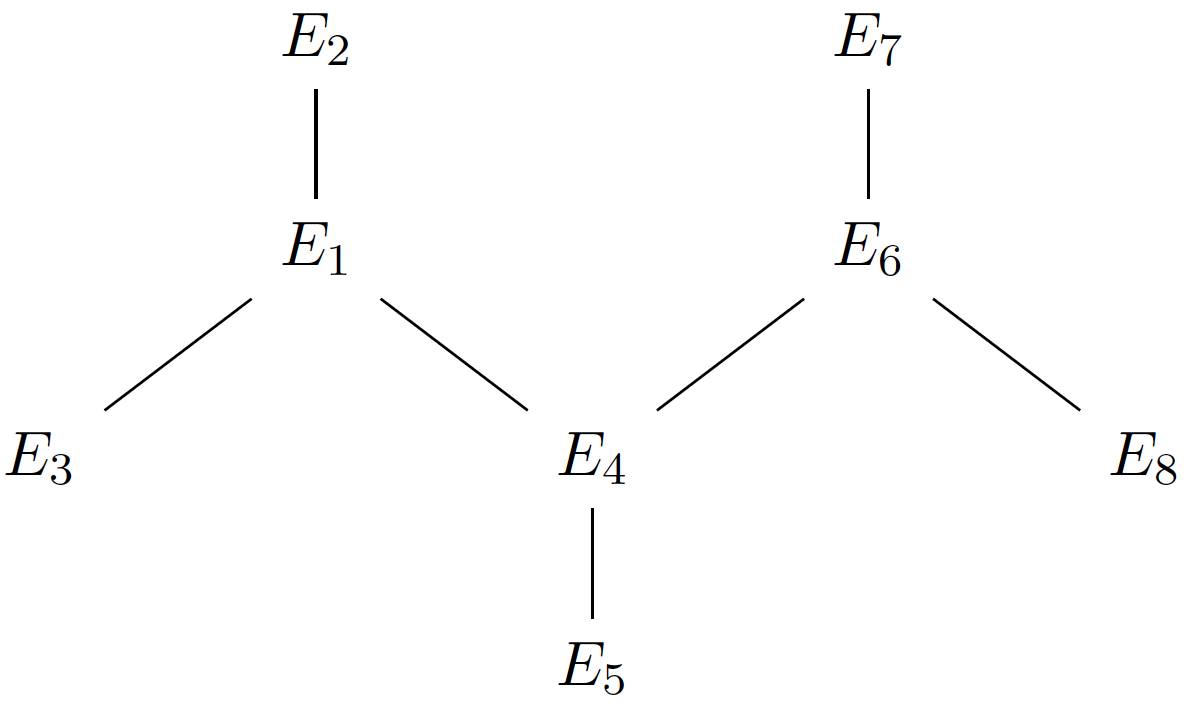}} 

& $([2,8],[8],[8],[2,4],[4],[2,2],[2],[2])$ & $\operatorname{H}_{193n} = \texttt{8.96.0.40}$ & $\operatorname{H}_{235l} = \texttt{16.96.0.95}$ & $\operatorname{H}_{235l} = \texttt{16.96.0.95}$ & $\operatorname{H}_{189d} = \texttt{8.96.0.2}$ & $\operatorname{H}_{207n} = \texttt{16.96.0.25}$ & $\operatorname{H}_{192i} = \texttt{8.96.0.60}$ & $\operatorname{H}_{211m} = \texttt{16.96.0.149}$ & $\operatorname{H}_{211m} = \texttt{16.96.0.149}$ & \texttt{210.e} \\
\cline{2-11}

& $([2,4],[8],[4],[2,4],[4],[2,2],[2],[2])$ & $\operatorname{H}_{194l} = \texttt{8.96.0.39}$ & $\operatorname{H}_{213i} = \texttt{16.96.0.96}$ & $\operatorname{H}_{227i} = \texttt{16.96.0.107}$ & $\operatorname{H}_{183i} = \texttt{8.96.0.8}$ & $\operatorname{H}_{205h} = \texttt{8.96.0.102}$ & $\operatorname{H}_{208a} = \texttt{16.96.0.10}$ & $\operatorname{H}_{243g} = \texttt{32.96.0.14}$ & $\operatorname{H}_{243g} = \texttt{32.96.0.14}$ & \texttt{1230.f} \\
\cline{2-11}

& $([2,4],[4],[4],[2,4],[8],[2,2],[2],[2])$ & $\operatorname{H}_{215c} = \texttt{16.96.0.7}$ & $\operatorname{H}_{240l} = \texttt{32.96.0.5}$ & $\operatorname{H}_{240l} = \texttt{32.96.0.5}$ & $\operatorname{H}_{187d} = \texttt{8.96.0.3}$ & $\operatorname{H}_{195l} = \texttt{8.96.0.99}$ & $\operatorname{H}_{185h} = \texttt{8.96.0.58}$ & $\operatorname{H}_{225g} = \texttt{16.96.0.148}$ & $\operatorname{H}_{212h} = \texttt{16.96.0.153}$ & \texttt{15.a} \\
\hline
\end{tabular}}
\caption{2-adic Galois Images of $T_{8}$ Graphs}
\label{2-adic Galois Images of T8 Graphs}
\end{table}
\end{center}

\newpage

\begin{center}
\begin{table}[h!]
\renewcommand{\arraystretch}{1}
\scalebox{0.48}{
    \begin{tabular}{|c|c|c|c|c|c|c|c|c|c|c|}
    \hline
    Isogeny Graph & $\rho_{E_{1},2^{\infty}}(G_{\QQ})$ & $\rho_{E_{2},2^{\infty}}(G_{\QQ})$ & $\rho_{E_{3},2^{\infty}}(G_{\QQ})$ & $\rho_{E_{4},2^{\infty}}(G_{\QQ})$ & $\rho_{E_{5},2^{\infty}}(G_{\QQ})$ & $\rho_{E_{6},2^{\infty}}(G_{\QQ})$ & $\rho_{E_{7},2^{\infty}}(G_{\QQ})$ & $\rho_{E_{8},2^{\infty}}(G_{\QQ})$ & Isogeny-Torsion Graph & Example Isogeny Class\\
    \hline
    \multirow{8}*{\includegraphics[scale=0.25]{Snew}} &
    \multirow{2}*{$\operatorname{H}_{8}$} & \multirow{2}*{$\operatorname{H}_{8}$} & \multirow{2}*{$\operatorname{H}_{13}$} & \multirow{2}*{$\operatorname{H}_{13}$} & \multirow{2}*{$\operatorname{H}_{13}$} & \multirow{2}*{$\operatorname{H}_{13}$} & \multirow{2}*{$\operatorname{H}_{13}$} & \multirow{2}*{$\operatorname{H}_{13}$} & $([2,6],[2,2],[6],[2],[6],[2],[6],[2])$ & \texttt{630.f} \\
    \cline{10-11}
    & & & & & & & & & $([2,2],[2,2],[2],[2],[2],[2],[2],[2])$ & \texttt{3630.w} \\
    \cline{2-11}
    & $\operatorname{H}_{8a}$ & $\operatorname{H}_{8a}$ & $\operatorname{H}_{13a}$ & $\operatorname{H}_{13a}$ & $\operatorname{H}_{13a}$ & $\operatorname{H}_{13a}$ & $\operatorname{H}_{13a}$ & $\operatorname{H}_{13a}$ & $([2,2],[2,2],[2],[2],[2],[2],[2],[2])$ & \texttt{450.d} \\
    \cline{2-11}
    & $\operatorname{H}_{8b}$ & $\operatorname{H}_{8b}$ & $\operatorname{H}_{13b}$ & $\operatorname{H}_{13b}$ & $\operatorname{H}_{13b}$ & $\operatorname{H}_{13b}$ & $\operatorname{H}_{13b}$ & $\operatorname{H}_{13b}$ & $([2,2],[2,2],[2],[2],[2],[2],[2],[2])$ & \texttt{14400.o} \\
    \cline{2-11}
    & \multirow{2}*{$\operatorname{H}_{8c}$} & \multirow{2}*{$\operatorname{H}_{8c}$} & \multirow{2}*{$\operatorname{H}_{13c}$} & \multirow{2}*{$\operatorname{H}_{13c}$} & \multirow{2}*{$\operatorname{H}_{13e}$} & \multirow{2}*{$\operatorname{H}_{13e}$} & \multirow{2}*{$\operatorname{H}_{13g}$} & \multirow{2}*{$\operatorname{H}_{13g}$} & $([2,6],[2,2],[6],[2],[6],[2],[6],[2])$ & \texttt{30.a} \\
\hline
\end{tabular}}
\caption{2-adic Galois Images of S Graphs}
\label{2-adic Galois Images of S Graphs Table}
\end{table}
\end{center}

\begin{center}
\begin{table}[h!]
\renewcommand{\arraystretch}{1}
\scalebox{0.63}{
    \begin{tabular}{|c|c|c|c|c|c|c|c|c|}
    \hline
    
    Isogeny Graph & $\rho_{E_{1},2^{\infty}}(G_{\QQ})$ & $\rho_{E_{2},2^{\infty}}(G_{\QQ})$ & $\rho_{E_{3},2^{\infty}}(G_{\QQ})$ & $\rho_{E_{4},2^{\infty}}(G_{\QQ})$ & $\rho_{E_{5},2^{\infty}}(G_{\QQ})$ & $\rho_{E_{6},2^{\infty}}(G_{\QQ})$ & Isogeny-Torsion Graph & Example Isogeny Class \\
    \hline
    \multirow{4}*{\includegraphics[scale=0.25]{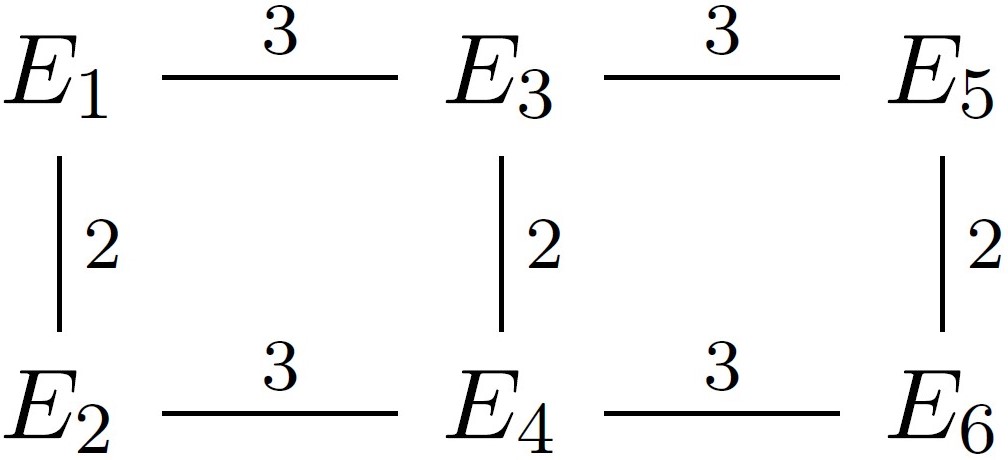}} & \multirow{2}*{$\operatorname{H}_{16}$} & \multirow{2}*{$\operatorname{H}_{17}$} & \multirow{2}*{$\operatorname{H}_{16}$} & \multirow{2}*{$\operatorname{H}_{17}$} & \multirow{2}*{$\operatorname{H}_{16}$} & \multirow{2}*{$\operatorname{H}_{17}$} & $([6],[6],[6],[6],[2],[2])$ & \texttt{14.a} \\
    \cline{8-9}
    & & & & & & & $([2],[2],[2],[2],[2],[2])$ & \texttt{98.a} \\
    \cline{2-9}
    & \multirow{2}*{$\operatorname{H}_{6}$} & \multirow{2}*{$\operatorname{H}_{6}$} & \multirow{2}*{$\operatorname{H}_{6}$} & \multirow{2}*{$\operatorname{H}_{6}$} & \multirow{2}*{$\operatorname{H}_{6}$} & \multirow{2}*{$\operatorname{H}_{6}$} & $([6],[6],[6],[6],[2],[2])$ & \texttt{2394.j} \\
    \cline{8-9}
    & & & & & & & $([2],[2],[2],[2],[2],[2])$ & \texttt{271950.do} \\
    \hline
    
	\end{tabular}}
\caption{2-adic Galois Images of $R_{6}$ Graphs}
\label{2-adic Galois Images of R6 Graphs Table}
\end{table}
\end{center}

\begin{center}
\begin{table}[h!]
\renewcommand{\arraystretch}{1}
\scalebox{0.7}{
    \begin{tabular}{|c|c|c|c|c|c|c|c|}
    \hline
    
    Isogeny Graph & $pq$ & $\rho_{E_{1},2^{\infty}}(G_{\QQ})$ & $\rho_{E_{2},2^{\infty}}(G_{\QQ})$ & $\rho_{E_{3},2^{\infty}}(G_{\QQ})$ & $\rho_{E_{4},2^{\infty}}(G_{\QQ})$ & Isogeny-Torsion Graph & Example Isogeny Class \\
    \hline
    \multirow{24}*{\includegraphics[scale=0.10]{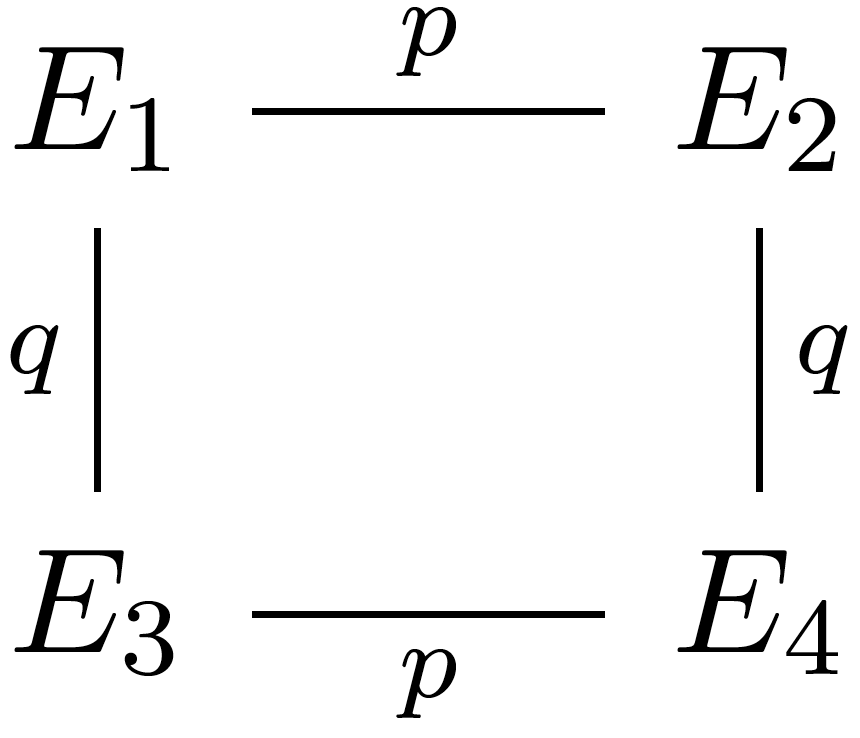}} & \multirow{2}*{$21$} & \multirow{2}*{$\operatorname{H}_{4}$} & \multirow{2}*{$\operatorname{H}_{4}$} & \multirow{2}*{$\operatorname{H}_{4}$} & \multirow{2}*{$\operatorname{H}_{4}$} & $([3],[3],[1],[1])$ & \texttt{162.b} \\
    \cline{7-8}
    & & & & & & $([1],[1],[1],[1])$ & \texttt{1296.f} \\
    \cline{2-8}
    & \multirow{3}*{$15$} & \multirow{3}*{$\operatorname{H}_{4}$} & 
    \multirow{3}*{$\operatorname{H}_{4}$} & \multirow{3}*{$\operatorname{H}_{4}$} & \multirow{3}*{$\operatorname{H}_{4}$} & $([5],[5],[1],[1])$ & \texttt{50.a} \\
    \cline{7-8}
    & & & & & & $([3],[3],[1],[1])$ & \texttt{50.b} \\
    \cline{7-8}
    & & & & & & $([1],[1],[1],[1])$ & \texttt{400.d} \\
    \cline{2-8}
    & \multirow{5}*{$10$} & \multirow{2}*{$\operatorname{H}_{15}$} & \multirow{2}*{$\operatorname{H}_{19}$} & \multirow{2}*{$\operatorname{H}_{15}$} & \multirow{2}*{$\operatorname{H}_{19}$} & $([10],[10],[2],[2])$ & \texttt{66.c} \\
    \cline{7-8}
    & & & & & & $([2],[2],[2],[2])$ & \texttt{198.c} \\
\hline
    \end{tabular}}
\caption{2-adic Galois Images of $R_{4}$ Graphs}
\label{2-adic Galois Images of R4 Graphs Table}
\end{table}
\end{center}

\newpage

\begin{center}
\begin{table}[h!]
\renewcommand{\arraystretch}{1.3}
\scalebox{0.51}{
    \begin{tabular}{|c|c|c|c|c|c|c|}
    \hline
    Isogeny Graph & $p$ & $\rho_{E_{1},2^{\infty}}(G_{\QQ})$ & $\rho_{E_{2},2^{\infty}}(G_{\QQ})$ & $\rho_{E_{3},2^{\infty}}(G_{\QQ})$ & Isogeny-Torsion Graph & Example Isogeny Class \\
    \hline
    \multirow{5}*{\includegraphics[scale=0.1]{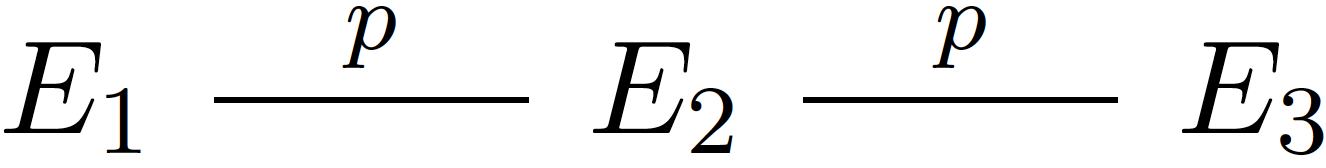}} & \multirow{2}*{$5$} & \multirow{2}*{$\operatorname{H}_{1}$} & \multirow{2}*{$\operatorname{H}_{1}$} & \multirow{2}*{$\operatorname{H}_{1}$} & $([5],[5],[1])$ & \texttt{11.a} \\
    \cline{6-7}
    & & & & & $([1],[1],[1])$ & \texttt{99.d} \\
    \cline{2-7}
    & \multirow{3}*{$3$} & \multirow{3}*{$\operatorname{H}_{1}$} & \multirow{3}*{$\operatorname{H}_{1}$} & \multirow{3}*{$\operatorname{H}_{1}$} & $([9],[3],[1])$ & \texttt{54.b} \\
    \cline{6-7}
    & & & & & $([3],[3],[1])$ & \texttt{19.a} \\
    \cline{6-7}
    & & & & & $([1],[1],[1])$ & \texttt{175.b} \\
    \hline

\end{tabular}}
\caption{2-adic Galois Images of $L_{3}$ Graphs}
\label{2-adic Galois Images of L3 Graphs Table}
\end{table}
\end{center}

\begin{center}
\begin{table}[h!]
\renewcommand{\arraystretch}{1}
\scalebox{0.7}{
    \begin{tabular}{|c|c|c|c|c|c|}

\hline

Isogeny Graph & $p$ & $\rho_{E_{1},2^{\infty}}(G_{\QQ})$ & $\rho_{E_{2},2^{\infty}}(G_{\QQ})$ & Isogeny-Torsion Graph & Example Isogeny Class \\
    \hline

\multirow{48}*{\includegraphics[scale=0.075]{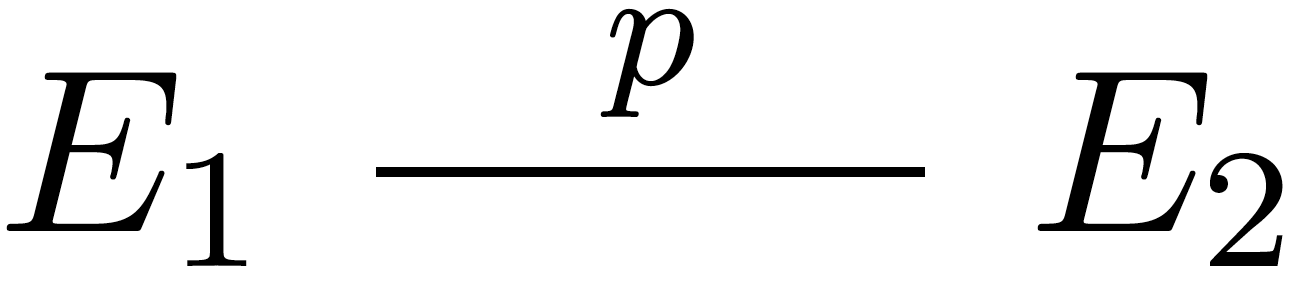}} & $37$ & $\operatorname{H}_{1}$ & $\operatorname{H}_{1}$ & $([1],[1])$ & \texttt{1225.b} \\
\cline{2-6}
& $17$ & $\operatorname{H}_{1}$ & $\operatorname{H}_{1}$ & $([1],[1])$ & \texttt{14450.b} \\
\cline{2-6}
& \multirow{2}*{$13$} & $\operatorname{H}_{1}$ & $\operatorname{H}_{1}$ & $([1],[1])$ & \texttt{147.b} \\
\cline{3-6}
& & $\operatorname{H}_{4}$ & $\operatorname{H}_{4}$ & $([1],[1])$ & \texttt{11094.g} \\
\cline{2-6}
& \multirow{3}*{$11$} & $\operatorname{H}_{3}$ & $\operatorname{H}_{3}$ & $([1],[1])$ & \texttt{121.a} \\
\cline{3-6}
& & $\operatorname{H}_{3a}$ & $\operatorname{H}_{3a}$ & $([1],[1])$ & \texttt{7744.c} \\
\cline{3-6}
& & $\operatorname{H}_{3b}$ & $\operatorname{H}_{3b}$ & $([1],[1])$ & \texttt{121.c} \\
\hline
\end{tabular}}
\caption{2-adic Galois Images of odd-degree $L_{2}$ Graphs}
\label{2-adic Galois Images of L2 Graphs Table Odd}
\end{table}
\end{center}

\bibliography{bibliography}
\bibliographystyle{plain}

\end{document}